\pdfoutput=1
\documentclass[11pt,reqno]{amsart}

\usepackage[margin=1.25in]{geometry}
\usepackage{amsmath,amssymb,amsthm}
\usepackage[authoryear,round]{natbib}
\usepackage{mathtools}
\usepackage{enumitem}
\usepackage{aliascnt}
\usepackage{graphicx}
\usepackage{booktabs}
\usepackage{hyperref}
\usepackage{url}
\usepackage[capitalize,noabbrev]{cleveref}
\usepackage{tikz}
\usetikzlibrary{arrows.meta,calc,positioning,fit,decorations.markings,backgrounds}
\definecolor{wellblue}{RGB}{45,91,150}
\definecolor{saddlered}{RGB}{177,48,48}
\definecolor{graphgray}{RGB}{104,112,122}

\allowdisplaybreaks

\makeatletter
\renewcommand{\paragraph}{\@startsection{paragraph}{4}%
  {\z@}{1.25ex \@plus .5ex \@minus .2ex}{-1em}%
  {\normalfont\normalsize\bfseries}}
\makeatother

\newtheorem{theorem}{Theorem}
\newaliascnt{proposition}{theorem}
\newtheorem{proposition}[proposition]{Proposition}
\aliascntresetthe{proposition}
\newaliascnt{lemma}{theorem}
\newtheorem{lemma}[lemma]{Lemma}
\aliascntresetthe{lemma}
\newaliascnt{corollary}{theorem}
\newtheorem{corollary}[corollary]{Corollary}
\aliascntresetthe{corollary}
\newaliascnt{definition}{theorem}
\newtheorem{definition}[definition]{Definition}
\aliascntresetthe{definition}
\theoremstyle{remark}
\newaliascnt{remark}{theorem}
\newtheorem{remark}[remark]{Remark}
\aliascntresetthe{remark}

\newcommand{\R}{\mathbb{R}}
\newcommand{\E}{\mathbb{E}}
\newcommand{\Pp}{\mathbb{P}}
\newcommand{\SphereN}{S^{N-1}(\sqrt N)}
\newcommand{\grad}{\nabla}
\newcommand{\pLiminf}{\operatorname*{\Pp\text{-}\liminf}}
\newcommand{\pLimsup}{\operatorname*{\Pp\text{-}\limsup}}

\newcommand{\vol}{\operatorname{vol}}
\newcommand{\Hess}{\operatorname{Hess}}
\newcommand{\ind}{\operatorname{ind}}
\newcommand{\Var}{\operatorname{Var}}
\newcommand{\spec}{\operatorname{spec}}
\newcommand{\op}{\mathrm{op}}
\newcommand{\Crt}{\operatorname{Crt}}
\newcommand{\dd}{\,\mathrm d}
\newcommand{\eps}{\varepsilon}
\newcommand{\defeq}{\stackrel{\mathrm{def}}{=}}
\newcommand{\sclaw}{\mathrm{sc}}
\newcommand{\cA}{\mathcal A}
\newcommand{\cB}{\mathcal B}
\newcommand{\cC}{\mathcal C}
\newcommand{\cE}{\mathcal E}

\newcommand{\cI}{\mathcal I}
\newcommand{\cK}{\mathcal K}
\newcommand{\cL}{\mathcal L}
\newcommand{\cM}{\mathcal M}

\newcommand{\cS}{\mathcal S}

\newcommand{\cW}{\mathcal W}
\newcommand{\Cov}{\operatorname{Cov}}

\title[Spectral gap bounds for mixed spherical spin glass dynamics]{Spectral
Gap Bounds for Langevin Dynamics in Mixed Spherical Spin Glasses}

\author{Masoud Badiei Khuzani}
\address{}
\email{}

\subjclass[2020]{Primary 82D30, 60K35; Secondary 60G15, 82C44, 60B20}
\keywords{Mixed spherical spin glass, spherical $p$-spin glass,
Langevin dynamics, spectral gap, Eyring--Kramers formula, Kac--Rice
formula, metastability}

\begin{document}

\begin{abstract}
We prove rigorous \emph{lower} bounds on the relaxation time
$1/\gamma_{N,\beta}$ of Langevin dynamics for mixed spherical spin
glasses with even mixture $\xi(x)=\sum_{p\ge4}\gamma_p^2x^p$, under
a one-step-replica-symmetry-breaking-type standing assumption of
strict threshold separation $E_0(\xi)>E_1(\xi)>E_2(\xi)$; the pure
spherical $p$-spin glass with even $p\ge4$, for which the assumption
is a theorem, is recovered as a corollary.  The main result is an
aggregate Eyring--Kramers bound whose exponent sums conductances
over the exponentially many index-one saddles above the lowest
saddle level $-NE_1(\xi)$, combining the saddle complexity
$\Theta_{1,\xi}$ with a half-determinant Hessian statistic---an
entropic contribution that the classical single-saddle picture
misses.  The single new random-matrix ingredient of the mixed model
is that the conditional Hessian at a critical point is a
\emph{randomly} shifted GOE matrix: the radial derivative is no
longer determined by the energy (Euler's identity degenerates
exactly in the pure case), and every landscape rate becomes a
one-dimensional supremum over the scalar shift with a Gaussian
penalty.  At low temperature the aggregate exponent exceeds the
unconditional free-energy bound by
$\tfrac12\log(\beta e)-\Xi^{\mathrm{EK}}_{1,\xi}(-E_1(\xi))
-C_{\xi,b}\beta^{-1/2}$ with
$-\Xi^{\mathrm{EK}}_{1,\xi}(-E_1(\xi))\ge\tfrac14\log\xi''(1)-\tfrac14>0$,
so the refinement is a strict improvement for all sufficiently large
fixed $\beta$; the onset temperature and the constants $b_*(\xi)$ and
$C_{\xi,b}$ produced by the proof are not numerically explicit.  Two
companion results---a sequential Arrhenius bound with the explicit
constant $E_0(\xi)-E_1(\xi)$ and a fixed-temperature free-energy
bound---come with complete, self-contained proofs.  All bounds are one-sided; identifying the mechanism the
dynamics actually realizes would require a matching upper bound,
which remains open.
\end{abstract}

\maketitle

\section{Introduction}

Langevin dynamics on a high-dimensional random landscape is the basic
model of noisy local search: it is the sampling scheme for Gibbs
posteriors, the continuum limit of noisy gradient methods, and the
physical dynamics of mean-field glasses.  For spherical spin-glass
Hamiltonians---the canonical Gaussian landscapes whose covariance
$N\xi(R)$ depends only on the overlap $R$ of configurations---the
statics are understood in great detail: the landscape carries
exponentially many critical points organized by index-dependent
energy thresholds, in the pure case $\xi(x)=x^p$ by
\citet{auffinger2013random} and for general mixtures by
\citet{auffinger2013complexity}, and at low temperature the Gibbs
measure condenses on pairs of antipodal wells near the ground-state
level \citep{subag2017geometry}.  The
\emph{dynamics} is understood asymptotically in time-rescaled limits
\citep{cugliandolo1993analytical,benarous2020bounding} and, at the
level of the spectral gap $\gamma_{N,\beta}$ of the Langevin
generator, is known to be exponentially slow at low temperature
\citep{gheissari2019spectral}.  The spectral gap is the natural
quantitative object: its inverse is the slowest $L^2$ relaxation time,
and an exponentially small gap is the mathematical formulation of
trapping in metastable wells.

What has been missing is the \emph{geometry inside the exponent}: how
the relaxation-time exponent is built from the landscape data---the
depth of the wells, the level of the lowest saddles, and, crucially,
the exponential \emph{number} of saddle channels connecting distant
wells.  Classical Eyring--Kramers theory
\citep{bovier2004metastability} resolves the prefactor of a single
gate saddle, but on a landscape with $e^{cN}$ candidate gates,
selecting one saddle and assuming its Hessian is typical is exactly
the kind of unjustified step a rigorous treatment must avoid.

This paper provides three lower bounds on the relaxation time whose
exponents are explicit functions of the landscape thresholds, and
whose main novelty is an \emph{aggregate} formulation: the exponent of
our refined bound sums Eyring--Kramers conductances over \emph{all}
index-one saddles in an energy window, weighted by complexity and a
half-determinant Hessian statistic, before any saddle is selected.
We work throughout with \emph{mixed} even spherical models
$\xi(x)=\sum_{p\ge4\ \mathrm{even}}\gamma_p^2x^p$ under a standing
assumption of strict threshold separation
$E_0(\xi)>E_1(\xi)>E_2(\xi)$ (\Cref{ass:mixture})---the landscape
signature of a one-step-replica-symmetry-breaking (1RSB) phase,
which holds throughout the pure-like mixture class of
\citet{auffinger2013complexity} and is a theorem in the pure case.
The pure spherical $p$-spin results are stated as a corollary
(\Cref{cor:pure}), in which every rate function collapses to a
closed form.  The price of mixing is a single structural change in
the random-matrix input: Euler's identity ties the radial derivative
to the energy only in the pure case, so the conditional Hessian at a
critical point of a genuine mixture is a GOE matrix with a
\emph{random} scalar shift, and all Kac--Rice rates become suprema
over the shift with a quadratic penalty.  

\paragraph{Contributions.}
\begin{itemize}
  \item Our main theoretical result is an \emph{aggregate}
  Eyring--Kramers lower bound (\Cref{thm:ek}): the exponent involves
  the conductance pressure $\Lambda^{\mathrm{EK}}_{\xi,\beta}$, a sum
  over \emph{all} index-one saddles in an energy window weighted by
  the complexity $\Theta_{1,\xi}$ and a half-determinant Hessian
  statistic, rather than a single selected saddle with an
  assumed-typical Hessian.  The bound is proved for every even
  mixture satisfying the standing assumption; the mixed-model
  Kac--Rice analysis rests on the randomly shifted GOE Hessian law
  (\Cref{prop:shifted-GOE}) and a scalar-tilt Laplace argument that
  reduces every rate to its pure-model counterpart optimized over
  the shift.  At low temperature the aggregate bound
  strictly improves the unconditional fixed-temperature bound by
  $\tfrac12\log(\beta e)-\Xi^{\mathrm{EK}}_{1,\xi}(-E_1(\xi))
  -C_{\xi,b}\beta^{-1/2}$ with
  $-\Xi^{\mathrm{EK}}_{1,\xi}(-E_1(\xi))>0$,
  hence for all sufficiently large fixed $\beta$
  (\Cref{cor:lowtemp,rem:improvement}); the constants $b_*(\xi)$,
  $C_{\xi,b}$, and the onset $\beta_0(\xi,b)$ are not numerically
  explicit, so the improvement is certified asymptotically in
  $\beta$, not at a given $(\xi,\beta)$.  Two companion results---a
  sequential Arrhenius bound with the explicit constant
  $E_0(\xi)-E_1(\xi)$, which sharpens the exponential spectral-gap bounds
  of \citet{gheissari2019spectral} into a variational constant built
  from the thresholds of
  \citet{auffinger2013random,auffinger2013complexity}, and an
  elementary fixed-temperature bound $F_\xi(\beta)-\beta E_1(\xi)$ in the
  free-energy-barrier spirit of \citet{benarous2018spectral}---come
  with complete, self-contained proofs
  (\Cref{sec:gate,sec:deep-cutoff}).  The pure even $p$-spin model is
  \Cref{cor:pure}.
  \item We probe the theory numerically at small sizes
  (\Cref{sec:experiments}).  The central experiment
  enumerates critical points of $p=4$ instances at $N=8,10,12$ by
  projected Newton iteration, classifies them by Hessian index, and
  measures the determinant-weighted index-one saddle sum
  $\cW_{N,\beta}$ of \Cref{thm:ek} directly---its energy-resolved
  profile against the predicted complexity
  $\Xi^{\mathrm{EK}}_{1,\xi}=\Theta_{1,\xi}-\tfrac12D_\xi$ (its pure-case
  closed form), and the share
  of the largest single channel in the aggregate as a function of
  $\beta$ and $N$.  At $N\le12$ the asymptotic regime is out of
  reach, so these are consistency checks of shape and trend, not
  quantitative validations.  Langevin simulations at even $p=4$
  reproduce threshold relaxation, aging, and activated escape with an
  effective barrier that grows with $N$ (bootstrap confidence
  intervals, explicit censoring).
\end{itemize}

All code and simulation scripts are publicly available at
\url{https://github.com/mbadieik/spectral_gap_for_mixed_spherical_spin_glasses}.

\subsection{Related work}\label{sec:related}

\paragraph{Landscape complexity of spherical spin glasses.}
The critical-point structure of the pure spherical $p$-spin model was
computed by \citet{auffinger2013random}
through Kac--Rice formulas and GOE large deviations: no local minima
below $-NE_0$, no index-$k$ critical points below $-NE_k$, and
exponential proliferation of critical points of all fixed indices near
the threshold $-NE_\infty$.  \citet{auffinger2013complexity} extended
the complexity theory to general mixtures and classified them into
\emph{pure-like}, \emph{full}, and transitory classes; our standing
assumption places the mixture in the pure-like (1RSB-type) regime.
The annealed complexity \emph{upper} bounds our proofs consume are
rederived here in a self-contained way for mixed models
(\Cref{prop:complexity-upper}), because the mixed conditional Hessian
carries a random scalar shift that couples into every determinant
rate.  \citet{subag2017complexity}
upgraded the annealed complexity of minima to a quenched statement by
a second-moment argument in the pure case; whether a quenched
analogue holds for the
\emph{determinant-weighted} saddle sums introduced here is an open
problem we state in \Cref{sec:discussion}.  The statics are governed
by the Crisanti--Sommers/Parisi free energy
\citep{crisanti1992sphericalpspin,talagrand2006free,chen2013mixed},
with ground-state formulas in
\citet{chen2017parisi,jagannath2017low}, and at low
temperature the Gibbs measure concentrates on antipodal pairs of deep
wells \citep{subag2017geometry}.  In the statistical-physics
literature, which saddles dominate activated escape has been studied
via barrier complexity and dynamical instantons
\citep{ros2019complexity,ros2021dynamical}; our aggregate bound is a
rigorous, one-sided counterpart of that question.

\paragraph{Mixed models in the physics literature.}
Spherical $s+p$ mixtures have been studied in detail by replica
methods.  \citet{crisanti2004spherical,crisanti2006spherical} solved
the $2+p$ model, including a mixed one-step/full (1-FRSB) phase, and
\citet{crisanti2007amorphous} treated $s+p$ mixtures with large
$p-s$, deriving a criterion on $(s,p)$ for the appearance of a stable
two-step (2RSB) phase and of mixed continuous--discontinuous phases;
the criterion is independent of parity and therefore applies to even
mixtures such as $4+16$.  These static phase diagrams bear on when
our landscape assumption \textup{(A2)} coincides with a 1RSB
low-temperature phase (\Cref{sec:rsb-frsb}).  On the dynamical side,
\citet{crisanti2007equilibrium} formulated the equilibrium dynamics
for any number of replica-symmetry-breaking steps and showed that
the gap between the threshold and the equilibrium free energy shrinks
as the number of steps grows and vanishes in the full-RSB phase;
\citet{crisanti2011statistical} report two-time-scale bifurcations
in a $3+p$ model.  For the out-of-equilibrium dynamics of mixed
models, \citet{folena2020rethinking,folena2021gradient} showed that
gradient descent and Langevin dynamics from random starts need not
relax to the threshold energy and that the asymptotic energy depends
on the initial condition, and
\citet{lang2026dynamite,lang2026aging} integrated the dynamical
mean-field equations to very long times, obtaining the aging phase
diagram and exact asymptotic energies, with transitions between
aging states with one, two, and continuously many effective
temperatures.  We draw on these works in \Cref{sec:rsb}.

\paragraph{Langevin dynamics and spectral gaps.}
\citet{cugliandolo1993analytical} solved the out-of-equilibrium
Langevin dynamics in the thermodynamic limit, finding relaxation
toward the threshold energy $-E_\infty$ rather than the ground state,
and aging of two-time correlations; \citet{benarous2020bounding}
proved bounding flows for the limiting dynamics.  For the spectral
gap, \citet{gheissari2019spectral} established exponentially slow
mixing at low temperature, and \citet{benarous2018spectral} obtained
spectral-gap exponents from free-energy barriers in mean-field
glasses.  We are explicit about attribution: that the gap is
exponentially small is due to \citet{gheissari2019spectral}, the
free-energy-barrier route to gap exponents is in the spirit of
\citet{benarous2018spectral}, and the energy difference $E_0-E_1$ as
a landscape barrier is already visible in the complexity computations
of \citet{auffinger2013random}.  Relative to this line,
\Cref{thm:arrhenius,thm:fixed} contribute the explicit variational
constants $E_0(\xi)-E_1(\xi)$ and $F_\xi(\beta)-\beta E_1(\xi)$, tied to the
index-one threshold rather than to a generic free-energy barrier;
their proofs are elementary and self-contained, and we regard them as
companions.  The main new result is the aggregate Eyring--Kramers
bound of \Cref{thm:ek}, whose exponent---a conductance pressure over
all index-one channels---no prior work provides.  Classical
Eyring--Kramers theory for fixed (non-random) landscapes is due to
\citet{bovier2004metastability}.  A complementary recent line
analyzes when sampling mean-field Gibbs measures is possible or hard
for broad algorithm classes via algorithmic stochastic localization
\citep{elalaoui2022sampling}; our results concern the specific
canonical dynamics (fixed-temperature Langevin) and measure its
worst-case $L^2$ relaxation time.

\section{Notation and landscape preliminaries}\label{app:proofs}

This section collects the model, the standing assumption, the free
energy, the critical-point layers, and the derivative estimates used
throughout the proofs of the three theorems of \Cref{sec:theory}.
\Cref{sec:gate} proves the pathwise gate theorem and
Theorem~\ref{thm:arrhenius}; \Cref{sec:deep-cutoff} proves
Theorem~\ref{thm:fixed}; \Cref{sec:weighted-KR} derives the
complexity-weighted saddle pressure; \Cref{sec:network} states the
co-core localization theorem and derives Theorem~\ref{thm:ek};
\Cref{app:weighted-KR,app:geometry,app:near-threshold-cocore} supply
the random-matrix machinery, the spherical-geometry estimates, and the
proof of the localization theorem.

\subsection{The mixed spherical model and the standing assumption}

Let
\begin{equation}\label{eq:sphere}
 \SphereN=\{\sigma\in\R^N:\|\sigma\|^2=N\},
\end{equation}
and let $\nu_N$ denote normalized Riemannian surface measure on
$\SphereN$.  Fix a \emph{mixture polynomial}
\begin{equation}\label{eq:mixture}
 \xi(x)=\sum_{p\in P}\gamma_p^2\,x^p,
 \qquad \gamma_p\ne0\ (p\in P),
\end{equation}
where $P$ is a finite set of integers.  The Hamiltonian is the
centered Gaussian field
\begin{equation}\label{eq:hamiltonian}
 H_N(\sigma)
 =\sum_{p\in P}\gamma_p\,H_{N,p}(\sigma),
 \qquad
 H_{N,p}(\sigma)
 =\frac{1}{N^{(p-1)/2}}
 \sum_{i_1,\ldots,i_p=1}^N
 J^{(p)}_{i_1\cdots i_p}\sigma_{i_1}\cdots\sigma_{i_p},
\end{equation}
where all coefficients $J^{(p)}_{i_1\cdots i_p}$ are independent
standard Gaussian random variables, so that the pure components
$H_{N,p}$ are independent.  Equivalently,
\begin{equation}\label{eq:covariance}
 \E\bigl[H_N(\sigma)H_N(\tau)\bigr]
 =N\,\xi\bigl(R(\sigma,\tau)\bigr),
 \qquad
 R(\sigma,\tau)=\frac1N\langle\sigma,\tau\rangle.
\end{equation}

\begin{definition}[Standing assumption]\label{ass:mixture}
Throughout the paper the mixture \eqref{eq:mixture} satisfies:
\begin{enumerate}[label=\textup{(A\arabic*)},leftmargin=*]
 \item \emph{(Even $1$RSB-type mixture.)}  Every $p\in P$ is even and
 $p\ge4$, and the mixture is normalized by
 $\xi(1)=\sum_{p\in P}\gamma_p^2=1$.
 \item \emph{(Strict threshold separation.)}  The landscape constants
 defined below satisfy
 \begin{equation}\label{eq:threshold-order}
  E_0(\xi)>E_1(\xi)>E_2(\xi),
 \end{equation}
 where $E_0(\xi)$ is the ground-state constant of
 \eqref{eq:ground-limit} and $E_1(\xi),E_2(\xi)$ are the annealed
 complexity zeros of \eqref{eq:Ek}.
\end{enumerate}
\end{definition}

The normalization $\xi(1)=1$ costs no generality: replacing $H_N$ by
$cH_N$ only rescales $\beta$.  Assumption (A2) is the
one-step-replica-symmetry-breaking-type hypothesis under which the
deep landscape is layered by index; in the pure case
$\xi(x)=x^p$ it is a theorem of
\citet{auffinger2013random}, and within the \emph{pure-like} mixture
class of \citet{auffinger2013complexity} the strict layering of the
annealed thresholds persists, while for \emph{full} mixtures the
separation degenerates and our results become vacuous rather than
false (see \Cref{sec:rsb-frsb}).  Finiteness of $P$ is assumed for
convenience in the derivative estimates; mixtures with sufficiently
fast decay of $\gamma_p$ pose no additional difficulty.

Since every $p\in P$ is even,
\begin{equation}\label{eq:even-symmetry}
 H_N(-\sigma)=H_N(\sigma).
\end{equation}
The derivatives of the mixture at $x=1$ appear throughout:
\begin{equation}\label{eq:xi-derivatives}
 \xi'(1)=\sum_{p\in P}p\,\gamma_p^2,
 \qquad
 \xi''(1)=\sum_{p\in P}p(p-1)\,\gamma_p^2,
\end{equation}
together with the \emph{radial fluctuation scale}
\begin{equation}\label{eq:v-xi}
 v_\xi^2
 \defeq\xi'(1)+\xi''(1)-\xi'(1)^2
 =\sum_{p\in P}p^2\gamma_p^2
 -\Bigl(\sum_{p\in P}p\,\gamma_p^2\Bigr)^2
 \ \ge\ 0,
\end{equation}
where nonnegativity is the Cauchy--Schwarz inequality for the weights
$(\gamma_p^2)$, with equality if and only if $P$ is a singleton.
Thus $v_\xi=0$ characterizes the \emph{pure} model $\xi(x)=x^p$, in
which case Euler's identity
$\langle\sigma,\nabla H_N(\sigma)\rangle=pH_N(\sigma)$ makes the
radial derivative a deterministic function of the energy; for genuine
mixtures the radial derivative retains a Gaussian fluctuation of
variance $Nv_\xi^2$ around $N\xi'(1)u$ at energy $Nu$
(\Cref{app:geometry}), and this scalar fluctuation is the single new
random-matrix ingredient of the mixed model.  Note also that
$\xi''(1)\ge12$ by (A1), a crude bound used in
\Cref{rem:improvement}.

The partition function and free energy are
\begin{align}
 Z_{N,\beta}
 &\defeq\int_{\SphereN}e^{-\beta H_N(\sigma)}\,\nu_N(\dd\sigma),
 &F_{N,\xi}(\beta)&\defeq\frac1N\log Z_{N,\beta}.
 \label{eq:partition-free}
\end{align}
The Gibbs measure and Langevin generator are
\begin{align}
 \pi_{N,\beta}(\dd\sigma)
 &\defeq\frac{e^{-\beta H_N(\sigma)}}{Z_{N,\beta}}\,\nu_N(\dd\sigma),
 \label{eq:gibbs-measure}\\
 L_{N,\beta}
 &\defeq\Delta_{\SphereN}
 -\beta\langle\grad H_N,\grad\,\cdot\,\rangle.
 \label{eq:gibbs-generator}
\end{align}
The nonnegative spectral gap of $-L_{N,\beta}$ is
\begin{equation}\label{eq:spectral-gap}
 \gamma_{N,\beta}
 \defeq\inf_{\Var_{\pi_{N,\beta}}(f)>0}
 \frac{\int_{\SphereN}|\grad f|^2\dd\pi_{N,\beta}}
 {\Var_{\pi_{N,\beta}}(f)}.
\end{equation}

\subsection{Free energy and zero-temperature scale}

The spherical free-energy theorem gives a deterministic limit
\begin{equation}\label{eq:free-energy-limit}
 F_{N,\xi}(\beta)\longrightarrow F_\xi(\beta)
 \qquad\text{in probability for every }\beta>0;
\end{equation}
for even mixtures this is \citet{talagrand2006free}, and
\citet{chen2013mixed} removes the evenness restriction.  The
ground-state energy also has a deterministic limit,
\begin{equation}\label{eq:ground-limit}
 \frac1N\min_{\sigma\in\SphereN}H_N(\sigma)
 \longrightarrow-E_0(\xi)
 \qquad\text{in probability},
\end{equation}
with $E_0(\xi)\in(0,\infty)$ given by the zero-temperature
Crisanti--Sommers/Parisi variational formula; see
\citet{chen2017parisi} and \citet{jagannath2017low}.

\begin{proposition}[Free-energy/ground-state matching]\label{prop:free-energy-ground}
For every mixture satisfying \Cref{ass:mixture},
\begin{equation}\label{eq:F-over-beta}
 \lim_{\beta\to\infty}\frac{F_\xi(\beta)}{\beta}=E_0(\xi).
\end{equation}
In particular, since $E_0(\xi)>E_1(\xi)$ by \textup{(A2)}, there is
$\beta_\xi<\infty$ such that
\begin{equation}\label{eq:positive-fixed-condition}
 F_\xi(\beta)>\beta E_1(\xi)
 \qquad\text{for every }\beta\ge\beta_\xi.
\end{equation}
\end{proposition}

\begin{proof}
Since $\nu_N$ is a probability measure,
\[
 Z_{N,\beta}\le
 \exp\{-\beta\min_{\SphereN}H_N\},
\]
so \eqref{eq:free-energy-limit} and \eqref{eq:ground-limit} imply
$F_\xi(\beta)/\beta\le E_0(\xi)$.

For the reverse zero-temperature bound, fix $\delta>0$.  On the derivative event in \Cref{prop:derivative-scale} below, $H_N$ is $C_\xi\sqrt N$-Lipschitz in geodesic distance.  A geodesic ball of fixed radius $r\sqrt N$, centered at a ground state and with $r=r(\delta,\xi)>0$ sufficiently small, therefore lies in
\[
 \{H_N\le \min H_N+\delta N\}.
\]
Its normalized surface measure is at least $e^{-c_{\xi,\delta}N}$ for all large $N$.  Hence
\[
 F_{N,\xi}(\beta)
 \ge -\frac\beta N\min H_N-\beta\delta-c_{\xi,\delta}.
\]
Let $N\to\infty$, divide by $\beta$, then let $\beta\to\infty$ and finally $\delta\downarrow0$.  This gives the matching lower bound in \eqref{eq:F-over-beta}.  The strict inequality $E_0(\xi)>E_1(\xi)$ and \eqref{eq:F-over-beta} imply \eqref{eq:positive-fixed-condition}.
\end{proof}

\subsection{Complexity functionals and thresholds}\label{subsec:complexity-thresholds}

For a Borel set $B\subset\R$ and $k\in\{0,\ldots,N-1\}$, define
\begin{align}
 \Crt_{N,k}(B)
 \defeq\#\Bigl\{\sigma\in\SphereN:\;&
 \grad H_N(\sigma)=0,\notag\\[-0.1em]
 &\ind\Hess H_N(\sigma)=k,\notag\\[-0.1em]
 &H_N(\sigma)/N\in B\Bigr\},
 \label{eq:Crt}
\end{align}
and write $\Crt_{N,\ge k}(B)$ for the same count with
$\ind\Hess H_N(\sigma)\ge k$.

The conditional law of the Hessian at a critical point of energy
$Nu$ is a GOE matrix of scale $\sqrt{\xi''(1)}$ shifted by the
\emph{random} scalar $-(\xi'(1)u+\omega)$, where the shift
fluctuation $\omega$ is centered Gaussian with variance $v_\xi^2/N$
(\Cref{prop:shifted-GOE}).  At exponential scale the fluctuation
therefore participates in every Kac--Rice rate through the penalty
$\omega^2/(2v_\xi^2)$, and all landscape rate functions of the mixed
model are one-dimensional suprema over $\omega$.  We now fix this
family of functionals.  Define, for $u,\omega\in\R$,
\begin{equation}\label{eq:t-shift}
 t_\xi(u,\omega)\defeq\frac{\xi'(1)u+\omega}{\sqrt{\xi''(1)}},
\end{equation}
the position of the Hessian shift relative to the semicircle support
$[-2,2]$; the \emph{one-sided GOE outlier rate}
\begin{equation}\label{eq:Jgoe}
 J(t)\defeq\frac12\int_t^{-2}\sqrt{x^2-4}\,\dd x
 \quad(t\le-2),
 \qquad
 J(t)\defeq0\quad(t>-2);
\end{equation}
the semicircle logarithmic potential $\Omega$ of
\eqref{eq:Omega}; the \emph{determinant rate}
\begin{equation}\label{eq:Dxi}
 D_\xi(u,\omega)
 \defeq\frac12\log\xi''(1)
 +\Omega\bigl(t_\xi(u,\omega)\bigr),
 \qquad
 D_\xi(u)\defeq D_\xi(u,0);
\end{equation}
and the \emph{shift penalty}
\begin{equation}\label{eq:penalty}
 \mathfrak p_\xi(\omega)\defeq\frac{\omega^2}{2v_\xi^2}
 \quad(v_\xi>0);
 \qquad
 \mathfrak p_\xi(0)\defeq0,\quad
 \mathfrak p_\xi(\omega)\defeq+\infty\ (\omega\ne0)
 \quad(v_\xi=0),
\end{equation}
so that in the pure case every supremum over $\omega$ below is
evaluation at $\omega=0$.  For $r\ge0$ and $k\in\{1,2\}$, define the
\emph{tilted matrix rate}
\begin{equation}\label{eq:Lambda-rk}
 \Lambda^{(r)}_{k,\xi}(u)
 \defeq\sup_{\omega\in\R}
 \Bigl\{-\mathfrak p_\xi(\omega)
 +r\,D_\xi(u,\omega)
 -k\,J\bigl(t_\xi(u,\omega)\bigr)\Bigr\},
\end{equation}
the \emph{Kac--Rice prefactor rate}
\begin{equation}\label{eq:varrho-xi}
 \varrho_\xi(u)
 \defeq\frac12-\frac12\log\xi'(1)-\frac{u^2}2,
\end{equation}
and the \emph{annealed complexity and outlier functionals}
\begin{align}
 \Theta_{k,\xi}(u)
 &\defeq\varrho_\xi(u)+\Lambda^{(1)}_{k,\xi}(u),
 \qquad k\in\{1,2\},
 \label{eq:Theta-mixed}\\
 I_{1,\xi}(u)
 &\defeq\inf_{\omega\in\R}
 \Bigl\{\mathfrak p_\xi(\omega)
 +J\bigl(t_\xi(u,\omega)\bigr)\Bigr\}.
 \label{eq:I1}
\end{align}
Because $\mathfrak p_\xi$ is coercive while $D_\xi$ grows only
logarithmically and $J\ge0$, the supremum in \eqref{eq:Lambda-rk} is
attained on a compact $\omega$-interval, and
$\Lambda^{(r)}_{k,\xi}$ is finite, continuous, and locally Lipschitz
in $u$ (\Cref{lem:B-Lambda-regularity}).  Finally set
\begin{equation}\label{eq:E-infty}
 E_\infty(\xi)\defeq\frac{2\sqrt{\xi''(1)}}{\xi'(1)},
\end{equation}
the energy at which the typical shift $t_\xi(u,0)$ reaches the edge
of the semicircle, and define the index thresholds
\begin{equation}\label{eq:Ek}
 E_k(\xi)\defeq\sup\{a>0:\Theta_{k,\xi}(-a)\ge0\},
 \qquad k\in\{1,2\},
\end{equation}
which are finite because $\Theta_{k,\xi}(u)\to-\infty$ as
$u\to-\infty$ (the Gaussian term $-u^2/2$ dominates the
logarithmically growing matrix rate) and which satisfy
$\Theta_{k,\xi}(-E_k(\xi))=0$ and $\Theta_{k,\xi}(-a)<0$ for every
$a>E_k(\xi)$, by continuity.

\begin{remark}[The pure case]\label{rem:pure-collapse}
For $\xi(x)=x^p$ one has $\xi'(1)=p$, $\xi''(1)=p(p-1)$, $v_\xi=0$,
$t_\xi(u,0)=u\sqrt{p/(p-1)}$, and
$E_\infty(\xi)=2\sqrt{(p-1)/p}$.  The elementary identity
$\Omega(t)=t^2/4-\tfrac12-J(t)$ for $t\le-2$ (compare
\eqref{eq:Omega-explicit} and \eqref{eq:B-antiderivative}) gives, for
$u\le-E_\infty(\xi)$,
\begin{equation}\label{eq:Theta-explicit}
 \Theta_{k,\xi}(u)
 =\frac12\log(p-1)
 -\frac{p-2}{4(p-1)}u^2
 -(k+1)\,J\bigl(t_\xi(u,0)\bigr),
\end{equation}
which is the complexity function
$\Theta_{k,p}$ of \citet[Theorem~2.5]{auffinger2013random}, with
$J(t_\xi(u,0))=I_{1,\xi}(u)$ the GOE outlier rate at the shifted
threshold.  Thus \eqref{eq:Ek} reproduces the thresholds $E_k(p)$ of
the pure model, and \textup{(A2)} holds by
\citet[Theorem~2.12]{auffinger2013random}.  For genuine mixtures,
\eqref{eq:Theta-mixed} is the annealed rate produced by the same
Kac--Rice computation carried out with the randomly shifted Hessian;
the corresponding complexity theory for general mixtures is
developed by \citet{auffinger2013complexity}.  Our proofs consume
only the \emph{upper} bound of the next proposition, which is proved
in \Cref{app:weighted-KR} with the rest of the random-matrix
machinery, so no external complexity formula is invoked.
\end{remark}

\begin{proposition}[Annealed complexity upper bound]\label{prop:complexity-upper}
Let $k\in\{1,2\}$ and $a>0$.  Then
\begin{equation}\label{eq:complexity-cumulative}
 \limsup_{N\to\infty}\frac1N
 \log\E\Crt_{N,\ge k}((-\infty,-a])
 \le\sup_{u\le-a}\Theta_{k,\xi}(u),
\end{equation}
and the supremum is finite and negative whenever $a>E_k(\xi)$.
\end{proposition}

\begin{proof}
This is \Cref{lem:B-index-count} in \Cref{app:weighted-KR}: the
marked Kac--Rice identity \eqref{eq:B-KR-identity}, the per-shift GOE
estimates of \Cref{lem:B-matrix-rates}, the two-outlier estimate of
\Cref{lem:B-two-outlier}, the scalar-tilt Laplace bound of
\Cref{lem:B-tilt-laplace}, and a Gaussian energy tail.  Negativity
for $a>E_k(\xi)$ holds because $\Theta_{k,\xi}(-a')<0$ for every
$a'>E_k(\xi)$ by \eqref{eq:Ek}, $\Theta_{k,\xi}$ is continuous, and
$\Theta_{k,\xi}(u)\to-\infty$ as $u\to-\infty$.
\end{proof}

\begin{proposition}[Low-index layering of the mixed spherical landscape]\label{prop:layering}
Fix $k\in\{1,2\}$ and $a>E_k(\xi)$.  Then
\begin{equation}\label{eq:no-high-index}
 \Pp\left(
 \begin{array}{c}
 \exists\sigma\in\SphereN:\quad
 H_N(\sigma)\le-Na,\quad
 \grad H_N(\sigma)=0,\\[-0.1em]
 \ind\Hess H_N(\sigma)\ge k
 \end{array}
 \right)\longrightarrow0.
\end{equation}
Consequently:
\begin{enumerate}[label=\textup{(\roman*)},leftmargin=*]
 \item below every level $-Na$ with $a>E_1(\xi)$, all critical points are minima;
 \item below every level $-Nb$ with $b>E_2(\xi)$, all critical points have index zero or one.
\end{enumerate}
\end{proposition}

\begin{proof}
Markov's inequality and \Cref{prop:complexity-upper}: the probability
in \eqref{eq:no-high-index} is at most
$\E\Crt_{N,\ge k}((-\infty,-a])\le e^{-cN}$ for some $c>0$, since the
annealed exponent is negative for $a>E_k(\xi)$.
\end{proof}

\subsection{Morse regularity and derivative scale}

\begin{proposition}[Morse regularity of the spherical Hamiltonian]\label{prop:morse-spherical}
For every fixed $N$ and every mixture with $\max P\ge3$, the random function $H_N$ is almost surely Morse on $\SphereN$.  In particular, it has finitely many critical points, every critical value can be crossed by a finite Morse-handle attachment, and every first merger of two sublevel components occurs through index-one critical points.
\end{proposition}

\begin{proof}
The derivative nondegeneracy and regularity hypotheses hold for
general smooth isotropic Gaussian fields on the sphere; they are
verified for mixtures in
\citet[Section~3]{auffinger2013complexity} (see
\citet[Lemma~3.1]{auffinger2013random} for the pure case).  The
remaining assertions are the Morse deformation and attachment
theorems applied to this almost-sure realization; see
\citet{milnor1963morse}.
\end{proof}

\begin{proposition}[Uniform derivative scale]\label{prop:derivative-scale}
Define
\begin{align}
 \mathcal D_{N,\xi}=\sup_{\sigma\in\SphereN}\Biggl\{
 &\frac{|H_N(\sigma)|}{N}
 +\frac{\|\grad H_N(\sigma)\|}{\sqrt N}
 +\|\Hess H_N(\sigma)\|_{\op}\notag\\[-0.2em]
 &+\sqrt N\,\|\nabla^3H_N(\sigma)\|_{\op}
 \Biggr\}.
 \label{eq:derivative-functional}
\end{align}
There are constants $C_\xi,c_\xi>0$ such that
\begin{equation}\label{eq:derivative-event}
 \Pp(\mathcal D_{N,\xi}\le C_\xi)\ge1-e^{-c_\xi N}.
\end{equation}
\end{proposition}

\begin{proof}
For each pure component, a fixed-resolution product net for the Gaussian coefficient tensor, followed by Gaussian concentration and multilinear net extension, gives the ambient derivative scales
\begin{align*}
 |H_{N,p}|&=O(N),
 &\|\nabla H_{N,p}\|&=O(\sqrt N),\\
 \|\nabla^2H_{N,p}\|_{\op}&=O(1),
 &\|\nabla^3H_{N,p}\|_{\op}&=O(N^{-1/2}),
\end{align*}
uniformly on the sphere, with failure probability $e^{-c_pN}$.  Since
$P$ is finite, the triangle inequality extends the four scales to
$H_N=\sum_p\gamma_pH_{N,p}$ at the cost of the constant
$C_\xi=\sum_p|\gamma_p|C_p$ and the exponent
$c_\xi=\min_pc_p$.  Passing to covariant derivatives introduces only curvature factors of order $N^{-1/2}$ on the radius-$\sqrt N$ sphere, and therefore preserves the displayed scales.  A detailed tensor-net derivation is included in \Cref{app:geometry}.
\end{proof}

\section{Main results}
\label{sec:theory}

We now state the main theoretical results, for every mixture
satisfying the standing assumption of \Cref{ass:mixture}: an even
mixture $\xi(x)=\sum_{p\in P}\gamma_p^2x^p$ with all degrees even
and at least $4$ (evenness gives the exact antipodal symmetry used
in the proofs), normalized by $\xi(1)=1$, and with the strict
threshold separation $E_0(\xi)>E_1(\xi)>E_2(\xi)$ that characterizes
the 1RSB-type (pure-like) regime.  Complete, self-contained proofs occupy
\Cref{sec:gate,sec:deep-cutoff,sec:weighted-KR,sec:network}, with
the random-matrix and geometric estimates deferred to
\Cref{app:weighted-KR,app:geometry,app:near-threshold-cocore}.
Recall the notation of \Cref{app:proofs}, and write
$\pLiminf$ for the limit in probability over the coupling
disorder.\footnote{For random variables $X_N$, we use the limits in
probability
$\pLiminf_{N\to\infty}X_N\defeq\sup\{x:\Pp(X_N<x-\eps)\to0
\text{ for every }\eps>0\}$ and
$\pLimsup_{N\to\infty}X_N\defeq\inf\{x:\Pp(X_N>x+\eps)\to0
\text{ for every }\eps>0\}$.}

\begin{theorem}[Sequential Arrhenius bound]\label{thm:arrhenius}
Let $\xi$ satisfy \Cref{ass:mixture}.  Taking first the low-temperature limit at fixed
system size and then the large-$N$ limit in probability over the
disorder, the Arrhenius exponent of the spectral gap satisfies
\begin{equation}
  \pLiminf_{N\to\infty}\ \liminf_{\beta\to\infty}
  \Bigl(-\frac1{\beta N}\log\gamma_{N,\beta}\Bigr)
  \ \ge\ E_0(\xi)-E_1(\xi).
\end{equation}
\end{theorem}

\begin{theorem}[Fixed-temperature bound]\label{thm:fixed}
For every $\beta$ with $F_\xi(\beta)>\beta E_1(\xi)$,
\begin{equation}
  \pLiminf_{N\to\infty}
  \Bigl(-\frac1N\log\gamma_{N,\beta}\Bigr)
  \ \ge\ F_\xi(\beta)-\beta E_1(\xi).
\end{equation}
\end{theorem}

\begin{theorem}[Aggregate Eyring--Kramers bound]\label{thm:ek}
There is $b_*(\xi)\in(E_2(\xi),E_1(\xi))$ such that, for every
$b\in(b_*(\xi),E_1(\xi))$, there are constants $C_{\xi,b}<\infty$ and
$\beta_0(\xi,b)<\infty$ with the following property: for every fixed
$\beta\ge\beta_0(\xi,b)$,
\begin{equation}\label{eq:ek-fixed-b}
  \pLiminf_{N\to\infty}
  \Bigl(-\frac1N\log\gamma_{N,\beta}\Bigr)
  \ \ge\ F_\xi(\beta)
  -\max\Bigl\{\beta b,\
  \Lambda^{\mathrm{EK}}_{\xi,\beta}(b)+\tfrac{C_{\xi,b}}{\sqrt\beta}\Bigr\},
\end{equation}
where
$\Lambda^{\mathrm{EK}}_{\xi,\beta}(b)
=-\tfrac12\log(\beta e)
+\sup_{u\in[-E_1(\xi),-b]}\{-\beta u+\Xi^{\mathrm{EK}}_{1,\xi}(u)\}$
is the aggregate conductance pressure: a sum of Eyring--Kramers
conductances \citep{bovier2004metastability} over \emph{all} index-one
saddles in the energy window, with
$\Xi^{\mathrm{EK}}_{1,\xi}(u)
=\varrho_\xi(u)+\Lambda^{(1/2)}_{1,\xi}(u)$ the conductance
complexity of \eqref{eq:Xi-EK}, combining the saddle complexity and
the half-determinant statistic of the randomly shifted GOE Hessian
through a joint supremum over the shift; in the pure case it
collapses to $\Theta_{1,\xi}(u)-\tfrac12D_\xi(u)$
(equation~\eqref{eq:Xi-EK-explicit} of
\Cref{sec:weighted-KR}).  Consequently, at any fixed $\beta$, writing
$B(\beta)\defeq\{b\in(b_*(\xi),E_1(\xi)):\beta_0(\xi,b)\le\beta\}$ for the
set of cut depths admissible at that temperature,
\begin{equation}\label{eq:ek-variational}
  \pLiminf_{N\to\infty}
  \Bigl(-\frac1N\log\gamma_{N,\beta}\Bigr)
  \ \ge\ F_\xi(\beta)
  -\inf_{b\in B(\beta)}
  \max\Bigl\{\beta b,\
  \Lambda^{\mathrm{EK}}_{\xi,\beta}(b)+\tfrac{C_{\xi,b}}{\sqrt\beta}\Bigr\}
  \qquad\text{whenever }B(\beta)\ne\varnothing.
\end{equation}
The free parameter $b$ is the depth of the co-core cut used in the
proof; \eqref{eq:ek-variational} optimizes only over the $b$
admissible at the given $\beta$---the temperature is \emph{not} an
optimization variable.  The proof forces
$b_*(\xi)\ge E_1(\xi)-\eta_0(\xi)/4$ for a nonexplicit margin
$\eta_0(\xi)>0$ (equation~\eqref{eq:D-bstar} of
\Cref{app:near-threshold-cocore}), so the
admissible window is a possibly very thin left neighborhood of
$E_1(\xi)$; neither $b_*(\xi)$ nor $C_{\xi,b}$ is numerically certified.
The theorem is proved as \Cref{thm:refined-EK} in \Cref{sec:network}.
\end{theorem}

\begin{corollary}[Low-temperature form]\label{cor:lowtemp}
Fix one $b\in(b_*(\xi),E_1(\xi))$.  For all sufficiently large fixed
$\beta$,
\begin{equation}\label{eq:lowtemp-main}
  \pLiminf_{N\to\infty}
  \Bigl(-\frac1N\log\gamma_{N,\beta}\Bigr)
  \ \ge\ F_\xi(\beta)-\beta E_1(\xi)
  -\Xi^{\mathrm{EK}}_{1,\xi}(-E_1(\xi))+\frac12\log(\beta e)
  -\frac{C_{\xi,b}}{\sqrt\beta},
\end{equation}
where
$-\Xi^{\mathrm{EK}}_{1,\xi}(-E_1(\xi))
\ge\tfrac14\log\xi''(1)-\tfrac14>0$, with equality to
$\tfrac12D_\xi(-E_1(\xi))$ in the pure case (the sharper constant
$\tfrac14\log\xi''(1)+\tfrac14$ fails for some mixtures;
\Cref{rem:Xi-constant-example}).
There is no residual $o_\beta(1)$ error: since
$\Xi^{\mathrm{EK}}_{1,\xi}$ is Lipschitz on
$[-E_1(\xi),-b]$, once $\beta$ exceeds its Lipschitz constant the
supremum
defining $\Lambda^{\mathrm{EK}}_{\xi,\beta}(b)$ is attained exactly at
$u=-E_1(\xi)$.  This is \Cref{cor:refined-low-temp} in
\Cref{sec:network}.
\end{corollary}

\begin{remark}[Improvement over \Cref{thm:fixed}, asymptotic in $\beta$]\label{rem:improvement}
The right-hand side of \eqref{eq:lowtemp-main} exceeds the exponent
$F_\xi(\beta)-\beta E_1(\xi)$ of \Cref{thm:fixed} by
\[
  \tfrac12\log(\beta e)-\Xi^{\mathrm{EK}}_{1,\xi}(-E_1(\xi))
  -C_{\xi,b}\beta^{-1/2},
\]
which is positive---and diverges like $\tfrac12\log\beta$---for all
sufficiently large $\beta$: indeed
$-\Xi^{\mathrm{EK}}_{1,\xi}(-E_1(\xi))
\ge\tfrac14\log\xi''(1)-\tfrac14\ge\tfrac14\log12-\tfrac14>0.37$
for every admissible mixture, since $\xi''(1)\ge12$ under
\Cref{ass:mixture}; numerically the pure value is
$\tfrac12D_\xi(-E_1(4))\approx0.896$ at $p=4$, and the mixture
$\xi(q)=0.9q^4+0.1q^{12}$ gives $\approx1.041$
(\Cref{rem:Xi-constant-example}).  Thus the aggregate
mechanism of \Cref{thm:ek} is not vacuous: at low temperature it
yields a strictly larger relaxation-time exponent than the
unconditional bound.  We state the limitation plainly: because
$C_{\xi,b}$ and the onset $\beta_0(\xi,b)$ are not numerically explicit,
the improvement is certified only asymptotically in $\beta$---at no
concrete $(\xi,\beta)$ does the present proof certify more than
\Cref{thm:fixed}.  Obtaining computable bounds on $b_*(\xi)$ and
$C_{\xi,b}$ is an open problem (\Cref{sec:discussion}), and for this
reason we present the finite-$\beta$ comparison of the two branches
of \eqref{eq:ek-fixed-b} only as a formal leading-term comparison
(\Cref{sec:transitions}).
\end{remark}

\begin{corollary}[Pure spherical $p$-spin model]\label{cor:pure}
Let $\xi(x)=x^p$ with $p\ge4$ even.  Then \Cref{ass:mixture} holds:
evenness and normalization are immediate, and the strict separation
$E_0(\xi)>E_1(\xi)>E_2(\xi)$ is a theorem of
\citet{auffinger2013random}, the thresholds coinciding with their
$E_0(p)>E_1(p)>E_2(p)$.  Moreover $v_\xi=0$, so every supremum over
the Hessian shift collapses at $\omega=0$ and the rate functions
take the closed forms
\[
 \Theta_{k,\xi}(u)
 =\frac12\log(p-1)-\frac{p-2}{4(p-1)}u^2-(k+1)I_{1,\xi}(u),
 \qquad
 \Xi^{\mathrm{EK}}_{1,\xi}(u)
 =\Theta_{1,\xi}(u)-\frac12D_\xi(u),
\]
for $u\le-E_\infty(\xi)=-2\sqrt{(p-1)/p}$, with
$D_\xi(u)=\tfrac12\log[p(p-1)]+\Omega(u\sqrt{p/(p-1)})$ and
$I_{1,\xi}$ the GOE outlier rate at the shifted threshold
(\Cref{rem:pure-collapse}).  \Cref{thm:arrhenius,thm:fixed,thm:ek}
and \Cref{cor:lowtemp} therefore hold unconditionally for the pure
even model, with
$-\Xi^{\mathrm{EK}}_{1,\xi}(-E_1(\xi))=\tfrac12D_\xi(-E_1(\xi))$ in
\eqref{eq:lowtemp-main}.
\end{corollary}

\begin{remark}[What the bounds do and do not show]\label{rem:scope}
\Cref{thm:arrhenius,thm:fixed,thm:ek} are proved by
exhibiting trial functions with small Dirichlet energy (or by direct
comparison to such a bound); they are \emph{lower}
bounds on the relaxation time and nothing more.  They do \emph{not}
establish (i) a matching upper bound; (ii) that the communication
height between the deep wells is attained near $-NE_1(\xi)$; (iii) that
the saddles dominating $\Lambda^{\mathrm{EK}}_{\xi,\beta}$ are traversed
by typical transitions; or (iv) that either branch of the maximum in
\Cref{thm:ek} describes the realized mechanism---the maximum
only records which term controls the \emph{bound}.  What the
landscape input supplies unconditionally is an exclusion: below
$-NE_1(\xi)$ there are no index-one critical points at all
(\Cref{prop:layering}; \citealp{auffinger2013random} in the pure
case), so no one-dimensional gate can lie below
that level.  The bounds concern the fixed-temperature reversible
dynamics; they do not constrain annealed, tempered, or nonreversible
variants, or favorable initializations.
\end{remark}

\subsection{Reading the theorems}\label{sec:reading}

\paragraph{Where gates can and cannot lie.}  The Arrhenius constant in
\Cref{thm:arrhenius} is $E_0-E_1$, the separation between the
deepest minima and the \emph{lowest possible} index-one energy: below
$-NE_1(\xi)$ the landscape has no index-one critical points at all, so
any one-dimensional gate between wells lies at or above that level.
Our lower bound is consistent with gates near $-NE_1(\xi)$, and a
crossing there would make the bound tight, but proving that the
communication height is \emph{attained} near $-NE_1(\xi)$ is open
(\Cref{rem:scope}).  The smallness of $E_0(\xi)-E_1(\xi)$ (about
$1.1\times10^{-2}$ in the pure case $p=4$) is a genuine feature of
these models:
the deepest wells and the lowest saddles live in a narrow band, and
it is the extensivity in $N$, not the size of the constant, that
produces exponential trapping.

\paragraph{One channel or exponentially many?}
\Cref{thm:ek} goes beyond the classical single-saddle
Eyring--Kramers picture: the term $\Theta_{1,\xi}(u)$ inside
$\Lambda^{\mathrm{EK}}_{\xi,\beta}$ counts the exponential number of
index-one channels at energy density $u$, and the bound aggregates the
conductances of all of them---whether typical transitions use many
channels or few is exactly the question a matching upper bound would
answer.  The saddle-enumeration experiment of
\Cref{sec:saddle-exp} measures this channel structure directly at
small $N$: it computes the weighted sum $\cW_{N,\beta}$, its
energy-resolved profile, and the share of the largest single channel.

\paragraph{The role of temperature.}
The maximum in \Cref{thm:ek} exchanges dominance at a
$\beta$-dependent level: at very low temperature the entropic term
$-\tfrac12\log(\beta e)$ and the complexity are negligible and the
bound reduces to a pure Arrhenius form; at moderate temperature the
complexity term $\Theta_{1,\xi}$ lowers the exponent of the bound.  If
the true relaxation rate has the same structure as the bound---which
is unproven---this would correspond to a regime in which raising the
noise temperature accelerates mixing \emph{faster than Arrhenius}
because the dynamics recruits an exponential number of higher saddles;
what the theorem itself provides is only that the \emph{obstruction} to
mixing weakens super-Arrheniusly there.

\section{Replica symmetry breaking, ultrametricity, and aging}
\label{sec:rsb}

The constants in \Cref{thm:arrhenius,thm:fixed,thm:ek} are landscape
quantities.  This section places them in the replica-theoretic
classification of the statics: the bounds are a quantitative
expression of the overlap gap of a one-step
replica-symmetry-breaking (1RSB) phase; the finiteness of the co-core
cut mirrors the two-level ultrametric hierarchy; and, set against
the known out-of-equilibrium theory, the bounds constrain worst-case
relaxation, and we state precisely what they do and do not imply for
aging of the two-time overlap.  Except where a theorem is
cited explicitly, the statements of this section are interpretive and
one-sided in the same sense as \Cref{rem:scope}.

\subsection{The 1RSB overlap gap as an energy barrier}
\label{sec:rsb-statics}

At low temperature the statics of the pure spherical model, and of
the 1RSB mixtures our standing assumption is modeled on, are
one-step RSB: the Parisi measure of the
Crisanti--Sommers/Parisi functional
\citep{crisanti1992sphericalpspin,talagrand2006free,chen2013mixed}
consists of two
atoms (see \citealp{auffinger2015properties} for criteria on $\xi$),
so that for two independent
samples $\sigma,\sigma'$ from the same Gibbs measure the overlap
$R(\sigma,\sigma')$ concentrates, for even mixtures, on the three values
$\{-q_\beta,0,q_\beta\}$ with $q_\beta<1$ and $q_\beta\uparrow1$ as
$\beta\to\infty$: two samples either fall in the same well (overlap
$\pm q_\beta$, the sign by antipodal symmetry) or in nearly
orthogonal wells (overlap close to $0$).  Geometrically, the Gibbs
measure concentrates on antipodal pairs of nearly orthogonal
spherical bands around the deepest minima \citep{subag2017geometry};
\Cref{prop:deep-components,prop:deep-mass} record the part of this
picture that our proofs use, namely that below $-Na$ the landscape
consists of finitely many well components, antipodally paired and
carrying half of the deep Gibbs mass on each side.

The forbidden overlap interval $(0,q_\beta)$---the defining feature
of a 1RSB phase---corresponds on the sphere to the region between
distinct deep wells, and the theorems quantify what the dynamics must
pay to traverse it: every low path between the two well unions
crosses an index-one gate no lower than $-NE_1(\xi)$
(\Cref{thm:arrhenius}); at fixed temperature the crossing costs free
energy $F_\xi(\beta)-\beta E_1(\xi)$ (\Cref{thm:fixed}); and the
aggregate saddle sum resolves the subexponential structure of that
cost (\Cref{thm:ek}).  In this sense the variational constant
$E_0(\xi)-E_1(\xi)$ is an energy-scale expression of the 1RSB overlap
gap: it is strictly positive precisely because the deep wells sit at
isolated points of the overlap support rather than at the endpoint of
a continuum of states.

\subsection{Ultrametricity and the finite cut}
\label{sec:rsb-ultra}

In the thermodynamic limit, mean-field Gibbs measures are
ultrametric \citep{panchenko2013parisi}; for a 1RSB measure the
ultrametric tree has exactly two levels---clusters of internal
overlap $q_\beta$, mutually at overlap $0$.  The proof of
\Cref{thm:ek} uses this two-level structure in geometric form: the
low sublevel set decomposes into finitely many low-index components
(\Cref{prop:layering,prop:deep-components}), a \emph{single, finite}
co-core cut at one level pair $(a,b)$ separates them
(\Cref{lem:topological-cocore-cut}), and the Eyring--Kramers sum runs
over the one layer of index-one gates joining nearly orthogonal
clusters.  A full-RSB phase would break exactly this step: for a
continuous Parisi measure the tree has infinitely many levels,
barriers occur on all scales, and no single cut level with a uniform
gate energy separates the state space into finitely many pieces.  The
method of this paper is therefore intrinsically a 1RSB method, and we
regard this as informative rather than incidental: the point where
the proof would fail is the point where the phenomenon changes.

\subsection{An overlap dichotomy for aging}
\label{sec:rsb-aging}

Write
\begin{equation}\label{eq:two-time-overlap}
  C_N(t_w,t_w+t)\defeq R(\sigma_{t_w},\sigma_{t_w+t})
\end{equation}
for the two-time overlap along the Langevin trajectory, the quantity
measured in \Cref{fig:pspin} (right).

\paragraph{Within-well starts: persistence without aging.}
The trial-function computations behind \Cref{thm:fixed,thm:ek} bound
the equilibrium flow out of the deep-well union $\cA_N(a)$; by the
standard relation between capacities and exit times in
potential-theoretic metastability---the mean-hitting-time identity
$\E_{\nu_{A,B}}[\tau_B]=\mu(h_{A,B})/\operatorname{cap}(A,B)$, which
is proved for general reversible Markov processes and applies
verbatim to the Langevin diffusion on the compact sphere
\citep[Part~II]{bovier2016metastability} (the sharp Euclidean
asymptotics of \citealp{bovier2004metastability} are not
needed)---the mean time to reach the antipodal union, for a
trajectory started from the last-exit (equilibrium) distribution
$\nu_{A,B}$ on the boundary of one well union, is at least of order
$e^{N(F_\xi(\beta)-\beta E_1(\xi)-o(1))}$.  The expected reading is
that on subexponential time windows such a trajectory keeps two-time
overlap near $q_\beta$, time-translation invariance holds, and no
aging occurs: this is metastable equilibrium, the regime of the
largest-$\beta$ curves of \Cref{fig:pspin} (right), where $C$ barely
decays over the accessible window.  Only the mean-exit-time bound is
a theorem; the statement about the overlap trajectory is the standard
metastability picture.

\paragraph{Random starts: what the bounds do and do not certify.}
A uniform start has overlap $O(N^{-1/2})$ with every deep well.  In
the pure model, for $\beta>\beta_d(p)$ the limiting dynamics relaxes
toward the marginal threshold states at energy density
$-E_\infty(\xi)$ rather than into the wells.  This is the
Cugliandolo--Kurchan picture \citep{cugliandolo1993analytical},
rigorous at the level of the limiting CHSCK equations
\citep{benarous2006cugliandolo} and of bounding flows
\citep{benarous2020bounding}, and it entails weak ergodicity
breaking of the two-time overlap: $C_N(t_w,t_w+t)$ decays in $t$ for
every fixed $t_w$, on a scale that grows with the age $t_w$.  For
genuine mixtures the picture is richer: the asymptotic energy reached
from a random start is in general not the threshold energy and
depends on the initial condition
\citep{folena2020rethinking,folena2021gradient}, and long-time
integration of the dynamical mean-field equations produces an aging
phase diagram with one, two, or continuously many effective
temperatures, together with exact asymptotic energies
\citep{lang2026dynamite,lang2026aging}.  Which of these aging states
the mixtures satisfying \Cref{ass:mixture} realize is a question for
that theory, not for the present bounds.

What \Cref{thm:fixed} adds concerns the \emph{worst-case} relaxation
time only.  Whenever $F_\xi(\beta)>\beta E_1(\xi)$, the spectral gap
is at most $e^{-N(F_\xi(\beta)-\beta E_1(\xi)-o(1))}$, so there exist
initial data and observables whose $L^2$ relaxation to the Gibbs
measure takes exponential time.  This does \emph{not} by itself
certify that the trajectory started from the uniform measure ages
until exponential times, nor does it fix the overlap trajectory it
follows.  The obstruction is a symmetry.  The bottleneck test
function of \Cref{thm:fixed} is odd under the antipodal map
$\sigma\mapsto-\sigma$, whereas the uniform initial law, the Gibbs
measure, the generator, and the two-time overlap
\eqref{eq:two-time-overlap} are all even.  The law of the trajectory
from a uniform start therefore stays in the even sector for all
times, and the slow odd mode is never excited by it; an
exponentially small spectral gap is consistent with fast relaxation
of every even observable from a symmetric start.  Certifying aging to
exponential times from a random start would require a separate
argument---for instance a lower bound on the time for the energy, an
even observable, to descend from the threshold region into the deep
wells---which this paper does not provide.  A previous version of
this paper claimed that extension; the claim is withdrawn.

\paragraph{The resulting picture.}
For even $p\ge4$ the overlap-based picture is a trichotomy in
$\beta$: (i) for $\beta<\beta_d(p)$, $C_N$ decays on an
$N$-independent scale and is time-translation invariant (equilibrium
relaxation); (ii) for $\beta>\beta_d(p)$ from a random start, $C_N$
ages at threshold energies on the time scales described by the
limiting dynamics, and whether this regime extends to exponential
times is open; (iii) from a start inside a deep well at large
$\beta$, the mean exit time is exponential and $C_N$ is expected to
stay pinned near $q_\beta$ without aging (metastable persistence).
Of these, only the worst-case relaxation bound and the mean exit-time
bound in (iii) are theorems of this paper; the remainder is the
Cugliandolo--Kurchan picture and its mixed-model refinements.
Regimes (ii) and (iii) are the two branches visible in
\Cref{fig:pspin}; the arrest experiment of \Cref{sec:arrest-exp}
locates the boundary, exactly $T_d=1/E_\infty$ at $p=3$.

\subsection{The full-RSB boundary of the method}
\label{sec:rsb-frsb}

For mixed even models $\xi(q)=\sum_p\gamma_p^2q^p$ the low-temperature
Parisi measure need not be one-step, and the complexity landscape
correspondingly splits into the \emph{pure-like} and \emph{full}
mixture classes of \citet{auffinger2013complexity}.  Our standing
assumption (A2) is the landscape expression of the pure-like
regime; it is a condition on the annealed complexity functions, and
we do not claim that it is equivalent to a one-step Parisi measure at
low temperature.  Whether a given mixture is in a 1RSB phase is a
question of the statics, for which the $s+p$ phase diagrams of
\citet{crisanti2004spherical,crisanti2006spherical,crisanti2007amorphous}
give explicit criteria: for $p-s$ large enough a stable 2RSB phase
and mixed continuous--discontinuous phases appear, by a criterion
independent of parity, so that even mixtures such as $4+16$ can leave
the one-step class while remaining within the scope of our theorems
whenever \textup{(A2)} holds.  This paper carries the entire program
out within \textup{(A2)}: the
randomly shifted GOE Hessian law (\Cref{prop:shifted-GOE}) replaces
the deterministic shift of the pure model, and every Kac--Rice rate
acquires a one-dimensional supremum over the shift.  In full-RSB
phases, by contrast, the overlap support is an interval: there is no
forbidden overlap gap, near-ground states are connected through
marginal directions, and the constants driving our bounds degenerate
together with the phase---the strict separation
$E_0(\xi)>E_1(\xi)>E_2(\xi)$ that anchors
\Cref{thm:arrhenius,thm:fixed,thm:ek} fails, and the theorems become
vacuous rather than false.  The dynamical counterpart of this
degeneration is known from the equilibrium-dynamics formulation of
\citet{crisanti2007equilibrium}: the gap between the threshold free
energy, at which the relaxing system gets stuck, and the equilibrium
free energy shrinks as the number of replica-symmetry-breaking steps
grows and vanishes in the full-RSB phase.  The energy gap
$E_0(\xi)-E_1(\xi)$ driving \Cref{thm:arrhenius} is the landscape
version of that statement.  This
degeneration is consistent with, rather than contradicted by, the
algorithmic side: in full-RSB spherical models near-optimal
configurations can be found in polynomial time by descending along
marginal directions \citep{subag2021following}, and in
replica-symmetric high-temperature phases sampling itself is provably
fast \citep{elalaoui2022sampling}.  We therefore read the boundary of
our method---now located precisely at the failure of (A2)---as a
phase boundary rather than a technical artifact.  A testable
consequence remains: a 1RSB phase predicts a \emph{single} aging
sector---one
barrier scale, one plateau in $C_N$---whereas 2RSB phases produce
two time scales \citep{crisanti2011statistical} and full-RSB
dynamics is expected to age hierarchically on a continuum of
scales; the aging phase diagram of \citet{lang2026aging}, with its
transitions between one, two, and continuously many effective
temperatures, is the mean-field prediction against which this should
be compared.  The single-plateau form of the measured two-time curves
in \Cref{fig:pspin}, for the pure $p=4$ model, is the 1RSB
signature.

\section{Antipodal gates and the sequential Arrhenius bound: proof of Theorem~\ref{thm:arrhenius}}\label{sec:gate}

\begin{definition}[Communication height, optimal paths, and gates]
\label{def:communication-height}
For nonempty compact sets $A,B\subset\SphereN$, the \emph{communication
height} between $A$ and $B$ is
\begin{equation}\label{eq:communication-height}
 \widehat H_N(A,B)
 \defeq\inf_{\substack{\omega\in C([0,1],\SphereN)\\
                   \omega(0)\in A,\ \omega(1)\in B}}
 \max_{t\in[0,1]}H_N(\omega(t)).
\end{equation}
An \emph{optimal path} is a path attaining the infimum in
\eqref{eq:communication-height}.  A \emph{pathwise gate} is a closed
subset of the critical level $\{H_N=\widehat H_N(A,B)\}$ met by
every optimal path; a pathwise gate is \emph{minimal} when no proper
closed subset of it is itself a pathwise gate.
\end{definition}

\begin{figure}[t]
\centering
\resizebox{\textwidth}{!}{%
\begin{tikzpicture}[x=1.00cm,y=1.00cm,>=Latex,font=\footnotesize,
  graphnode/.style={circle,draw=graphgray!70,fill=white,minimum size=4.5mm,inner sep=0pt,line width=.5pt},
  wellnode/.style={circle,draw=wellblue!85!black,fill=wellblue!18,minimum size=5.1mm,inner sep=0pt,line width=.8pt},
  saddlemark/.style={circle,draw=saddlered!85!black,fill=saddlered,minimum size=3.4mm,inner sep=0pt,line width=.5pt}]

\coordinate (O) at (3.87,4.02);
\def\Rgate{2.45}
\shade[ball color=blue!9,opacity=.96] (O) circle (\Rgate);
\begin{scope}
  \clip (O) circle (\Rgate);
  \fill[saddlered!12,opacity=.72] (1.20,3.55) rectangle (6.55,4.49);
  \draw[black!14,line width=.35pt] (3.87,4.02) ellipse (2.42 and .52);
  \draw[black!10,line width=.35pt] (3.87,4.85) ellipse (2.06 and .36);
  \draw[black!10,line width=.35pt] (3.87,3.19) ellipse (2.06 and .36);
  \draw[black!10,line width=.35pt] (3.87,4.02) ellipse (.70 and 2.42);
  \draw[black!10,line width=.35pt,rotate around={28:(3.87,4.02)}] (3.87,4.02) ellipse (.70 and 2.42);
  \draw[black!10,line width=.35pt,rotate around={-28:(3.87,4.02)}] (3.87,4.02) ellipse (.70 and 2.42);
  \fill[wellblue!27,opacity=.72] (3.87,5.93) ellipse (1.00 and .38);
  \fill[wellblue!27,opacity=.72] (3.87,2.11) ellipse (1.00 and .38);
\end{scope}
\draw[black!55,line width=.65pt] (O) circle (\Rgate);

\filldraw[wellblue!90!black,fill=wellblue] (3.87,5.97) circle (2.2pt);
\filldraw[wellblue!90!black,fill=wellblue] (3.87,2.07) circle (2.2pt);
\node[anchor=west,text=wellblue!75!black] at (4.08,6.02) {$m_N$};
\node[anchor=west,text=wellblue!75!black] at (4.08,2.03) {$-m_N$};
\node[fill=white,fill opacity=.86,text opacity=1,rounded corners=1.5pt,inner sep=2pt,align=center,text=wellblue!75!black] at (2.08,5.90) {well near $m_N$\\$H\approx-NE_0(\xi)$};
\node[fill=white,fill opacity=.86,text opacity=1,rounded corners=1.5pt,inner sep=2pt,align=center,text=wellblue!75!black] at (2.08,2.15) {well near $-m_N$\\$H\approx-NE_0(\xi)$};

\coordinate (s1) at (1.68,4.02);
\coordinate (s2) at (2.63,4.33);
\coordinate (s3) at (3.87,4.48);
\coordinate (s4) at (5.11,4.33);
\coordinate (s5) at (6.06,4.02);
\foreach \s in {s1,s2,s3,s4,s5}{\node[saddlemark] at (\s) {};}

\draw[black!62,dashed,line width=.75pt]
  (3.87,5.95) .. controls (2.30,5.52) and (1.55,4.82) .. (s1)
  .. controls (1.62,3.22) and (2.38,2.49) .. (3.87,2.09);
\draw[black!62,dashed,line width=.75pt]
  (3.87,5.95) .. controls (4.82,5.59) and (5.20,4.95) .. (s4)
  .. controls (5.18,3.30) and (4.84,2.50) .. (3.87,2.09);
\draw[black!45,densely dashed,line width=.55pt]
  (3.87,5.95) .. controls (3.72,5.37) and (3.74,4.91) .. (s3)
  .. controls (3.99,3.52) and (3.93,2.65) .. (3.87,2.09);

\node[fill=white,fill opacity=.90,text opacity=1,rounded corners=1.5pt,inner sep=2.3pt,align=center,text=saddlered!85!black] (gatelabel) at (3.85,.62)
  {minimal gate: index-one saddles at $H_N^\dagger$};
\draw[saddlered!80!black,->,line width=.65pt] (gatelabel.north) .. controls (3.3,1.35) and (2.9,2.80) .. (s2);
\node[fill=white,fill opacity=.90,text opacity=1,rounded corners=1.5pt,inner sep=1.7pt,align=center,text=saddlered!85!black] at (5.25,5.20)
  {$H_N^\dagger/N\ge -E_1(\xi)-o(1)$};

\node[wellnode] (a1) at (8.82,5.35) {};
\node[graphnode] (a2) at (9.35,4.25) {};
\node[graphnode] (a3) at (8.85,3.10) {};
\node[graphnode] (l1) at (10.45,4.92) {};
\node[graphnode] (l2) at (10.55,3.58) {};

\node[graphnode] (r1) at (13.70,4.92) {};
\node[graphnode] (r2) at (13.60,3.58) {};
\node[wellnode] (b1) at (15.32,5.35) {};
\node[graphnode] (b2) at (14.80,4.25) {};
\node[graphnode] (b3) at (15.30,3.10) {};

\foreach \u/\v in {a1/a2,a2/a3,a1/l1,a2/l1,a2/l2,a3/l2,l1/l2,
                     r1/r2,r1/b1,r1/b2,r2/b2,r2/b3,b1/b2,b2/b3}{
  \draw[graphgray!68,line width=.7pt] (\u) -- (\v);
}
\draw[saddlered!88!black,line width=1.15pt] (l1) -- coordinate (z1) (r1);
\draw[saddlered!88!black,line width=1.15pt] (l1) -- coordinate (z2) (r2);
\draw[saddlered!88!black,line width=1.15pt] (l2) -- coordinate (z3) (r2);
\node[saddlemark] at (z1) {};
\node[saddlemark] at (z2) {};
\node[saddlemark] at (z3) {};

\draw[saddlered!70,dashed,line width=.7pt] (12.08,2.72) -- (12.08,5.78);
\node[text=saddlered!85!black,fill=white,inner sep=1.3pt] at (12.08,6.08) {cut $\cC_N$};
\node[align=center,text=black!75] at (9.62,6.36) {$V_A$ (contains $\cA_N(a)$)};
\node[align=center,text=black!75] at (14.58,6.36) {$V_B$ (contains $\cB_N(a)$)};

\node[draw=wellblue!65,dashed,rounded corners=5pt,fit=(a1)(a2)(a3)(l1)(l2),inner sep=5pt] {};
\node[draw=wellblue!65,dashed,rounded corners=5pt,fit=(r1)(r2)(b1)(b2)(b3),inner sep=5pt] {};

\node[align=center,text=black!72] at (12.08,1.72)
  {Every path from $\cA_N(a)$ to $\cB_N(a)$ that stays below $-Nb$ crosses a cut handle.};

\draw[graphgray!70,line width=.7pt] (8.48,.86) -- (9.03,.86);
\node[anchor=west,text=black!72] at (9.10,.86) {one-handle};
\draw[saddlered!88!black,line width=1.1pt] (10.93,.86) -- (11.48,.86);
\node[saddlemark,minimum size=2.7mm] at (11.205,.86) {};
\node[anchor=west,text=black!72] at (11.58,.86) {cut handle};
\node[wellnode,minimum size=4mm] at (13.65,.86) {};
\node[anchor=west,text=black!72] at (13.90,.86) {deep-well component};

\end{tikzpicture}%
}
\caption{Geometry behind
Theorems~\ref{thm:arrhenius}--\ref{thm:ek}.  \textbf{Left:} on the
sphere, the Gibbs measure concentrates on the two antipodal wells at
energy $\approx-NE_0(\xi)$, and every optimal path between them meets an
index-one saddle at the communication height $H_N^\dagger$, which
cannot lie below $-N(E_1(\xi)+o(1))$
(\Cref{def:communication-height,thm:ground-state-barrier}).
\textbf{Right:} the low-index Morse handle graph used in the proofs:
vertices are connected components of a deep sublevel set (created by
zero-handles), edges are index-one one-handles, and the red edges form
a minimal cut $\cC_N$ separating the deep families $\cA_N(a)$ and
$\cB_N(a)$; every low path between the families crosses a cut handle
(\Cref{sec:network,app:near-threshold-cocore}).}
\label{fig:gate-morse-geometry}
\end{figure}
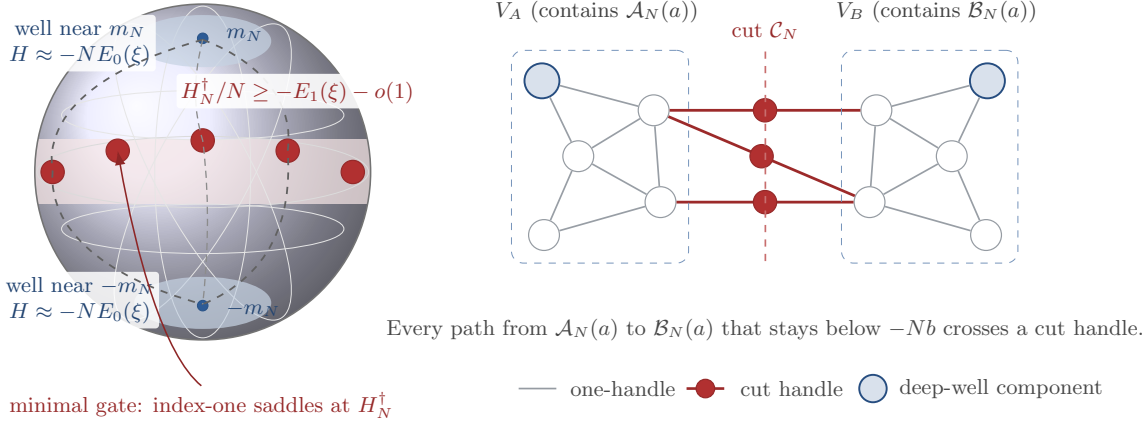

\begin{proposition}[Pathwise gates of the mixed spherical Hamiltonian]\label{prop:spin-gates}
Almost surely, let $m_1,m_2$ be two distinct local minima of $H_N$ and put
\[
 H=\widehat H_N(\{m_1\},\{m_2\}).
\]
Every pathwise minimal gate between $m_1$ and $m_2$ is a nonempty finite set of index-one critical points of $H_N$ at level $H$.  More precisely, choose $\eps>0$ so that $H$ is the only critical value in $[H-2\eps,H+2\eps]$.  The components of
$\{H_N\le H-\eps\}$ are the vertices of a finite graph, and the index-one saddles at level $H$ are its edges.  Minimal gates are exactly the inclusion-minimal edge cuts separating the vertices containing $m_1$ and $m_2$.
\end{proposition}

\begin{proof}
By \Cref{prop:morse-spherical}, the Hamiltonian is Morse.  The communication height is the first level at which the two relevant sublevel components merge, and therefore is a critical value.  By the Morse attachment theorem, an index-zero handle creates a new component, while a handle of index at least two has connected attaching sphere and cannot merge two previously distinct components.  Hence every edge responsible for the first merger is an index-one handle.

For completeness, the pathwise minimality statement requires more than existence of one finite saddle set.  Given any path in the handle graph, connect points inside the corresponding lower sublevel components and traverse the unstable axis in each index-one Morse chart.  The resulting canonical path has maximal energy $H$ and meets the level $H$ exactly at the saddles representing the graph edges.  Conversely, negative-gradient deformation away from the critical charts maps any optimal path to a graph walk; moving between two lower sublevel components forces the path through the center of the corresponding index-one handle.  Thus a closed set hits every optimal path exactly when its saddle points form an edge cut.  Inclusion-minimality on the two sides is the same, and the critical set is finite.
\end{proof}

\begin{theorem}[Antipodal ground-state gate barrier]\label{thm:ground-state-barrier}
Let $m_N$ be any ground state of $H_N$ and set
\begin{equation}\label{eq:Hdagger}
 H_N^\dagger
 \defeq\widehat H_N(\{m_N\},\{-m_N\}).
\end{equation}
Then, for every $\eps>0$,
\begin{equation}\label{eq:barrier-probability}
 \Pp\left(
 H_N^\dagger-H_N(m_N)
 \ge N\bigl(E_0(\xi)-E_1(\xi)-\eps\bigr)
 \right)\longrightarrow1.
\end{equation}
Every pathwise minimal gate realizing $H_N^\dagger$ is almost surely an inclusion-minimal finite cut of index-one saddles.
\end{theorem}

\begin{proof}
Evenness \eqref{eq:even-symmetry} makes $-m_N$ a distinct ground state.  By \Cref{prop:spin-gates}, the first-merging level contains an index-one critical point.  For every $\delta>0$, \Cref{prop:layering} implies that, with probability tending to one, no index-one critical point lies below $-N(E_1(\xi)+\delta)$.  Therefore
\[
 H_N^\dagger\ge-N(E_1(\xi)+\delta)
\]
on that event.  Combine this with the ground-state convergence \eqref{eq:ground-limit} and choose $2\delta<\eps$.
\end{proof}

\begin{theorem}[Sequential spherical-spin spectral-gap bound]\label{thm:sequential-gap}
For every $\eps>0$,
\begin{equation}\label{eq:sequential-prob}
 \Pp\left(
 \liminf_{\beta\to\infty}
 -\frac1{\beta N}\log\gamma_{N,\beta}
 \ge E_0(\xi)-E_1(\xi)-\eps
 \right)\longrightarrow1.
\end{equation}
Equivalently,
\begin{equation}\label{eq:sequential-pliminf}
 \pLiminf_{N\to\infty}\;
 \liminf_{\beta\to\infty}
 \left(-\frac1{\beta N}\log\gamma_{N,\beta}\right)
 \ge E_0(\xi)-E_1(\xi).
\end{equation}
This theorem uses neither an Eyring--Kramers prefactor nor a Hessian determinant.
\end{theorem}

\begin{proof}
Fix $N$ in an almost-sure Morse realization and abbreviate
$U_*=H_N(m_N)$ and $H=H_N^\dagger$.  Choose regular values
$U_*<c_0<c_1<H$.  Let $C_1$ be the component of
$\{H_N<c_1\}$ containing $m_N$.  Since $c_1<H$, the point $-m_N$ is outside $C_1$.  There is a smooth function equal to one near $m_N$, equal to zero outside $C_1$ and near $-m_N$, and whose gradient is supported in
$C_1\cap\{c_0\le H_N\le c_1\}$.

Its unnormalized Dirichlet energy is at most $C e^{-\beta c_0}$.  A quadratic upper bound near either isolated ground state and the elementary volume bound give the two-sided estimate
\[
 c\beta^{-(N-1)/2}e^{-\beta U_*}
 \ \le\ Z^{\vol}_{N,\beta}
 \ \le\ \vol(\SphereN)e^{-\beta U_*}
\]
for the partition function with surface volume, which shows that fixed $\beta^{-1/2}$-balls around the two minima have only polynomially small normalized mass.  The Rayleigh quotient is therefore at most a polynomial in $\beta$ times
$e^{-\beta(c_0-U_*)}$.  Sending $\beta\to\infty$ and then $c_0\uparrow H$ gives
\[
 \liminf_{\beta\to\infty}-\frac1\beta\log\gamma_{N,\beta}
 \ge H_N^\dagger-H_N(m_N).
\]
Now apply \Cref{thm:ground-state-barrier}.
\end{proof}

Theorem~\ref{thm:arrhenius} of the main text is exactly the statement
\eqref{eq:sequential-pliminf}, so its proof is complete.

\section{An unconditional fixed-temperature bound from deep sublevels: proof of Theorem~\ref{thm:fixed}}\label{sec:deep-cutoff}

For $a>0$, write
\begin{equation}\label{eq:deep-sublevel}
 \cL_N(a)\defeq\{\sigma\in\SphereN:H_N(\sigma)\le-Na\}.
\end{equation}
The next construction uses three facts that are special to the present model: the index-one threshold $E_1(\xi)$, antipodal symmetry, and the limiting spherical free energy.

\begin{proposition}[Antipodal pairing of deep spherical-spin wells]\label{prop:deep-components}
Fix $a>E_1(\xi)$.  With probability tending to one, every connected component of $\cL_N(a)$ contains exactly one critical point, which is a local minimum, and the components occur in distinct antipodal pairs
\[
 C,-C,\qquad C\ne-C.
\]
In particular, one can choose a sign map
\begin{equation}\label{eq:component-sign}
 s_{N,a}:\cL_N(a)\longrightarrow\{-1,1\}
\end{equation}
that is constant on each component and satisfies
$s_{N,a}(-\sigma)=-s_{N,a}(\sigma)$.
\end{proposition}

\begin{proof}
Choose $\delta>0$ with $E_1(\xi)+\delta<a$.  By \Cref{prop:layering}, with probability tending to one there is no critical point of index at least one below $-N(E_1(\xi)+\delta)$, hence none in $\cL_N(a)$.  The deterministic level $-Na$ is almost surely regular.  Starting below the minimum and increasing the sublevel to $-Na$, the only possible handle attachments are therefore index-zero handles.  Each creates one component and one local minimum, and no merger can occur.  Thus every component contains exactly one minimum.

The map $\sigma\mapsto-\sigma$ preserves $\cL_N(a)$ by \eqref{eq:even-symmetry} and permutes its components.  If a component $C$ were self-antipodal, the unique minimum $m\in C$ would imply that $-m\in C$ is a second, distinct minimum, a contradiction.  Choose one component from each antipodal pair and assign it sign $1$, assigning sign $-1$ to its antipode.  A measurable deterministic choice can be made by orienting the unique minimum according to the sign of its first nonzero coordinate.
\end{proof}

\begin{proposition}[Gibbs mass of a deep sublevel]\label{prop:deep-mass}
For every fixed $\beta>0$ and $a>0$,
\begin{equation}\label{eq:deep-mass-bound}
 \pi_{N,\beta}(\cL_N(a)^c)
 \le \exp\{\beta Na-NF_{N,\xi}(\beta)\}.
\end{equation}
Consequently, if $F_\xi(\beta)>\beta a$, then
\begin{equation}\label{eq:deep-mass-one}
 \pi_{N,\beta}(\cL_N(a))\longrightarrow1
 \qquad\text{in probability},
\end{equation}
and the convergence is exponentially fast on the $N$ scale.
\end{proposition}

\begin{proof}
On $\cL_N(a)^c$ one has $H_N>-Na$, and therefore
\[
 \int_{\cL_N(a)^c}e^{-\beta H_N}\dd\nu_N
 \le e^{\beta Na}.
\]
Division by $Z_{N,\beta}$ gives \eqref{eq:deep-mass-bound}; now use \eqref{eq:free-energy-limit}.
\end{proof}

\begin{theorem}[Unconditional fixed-temperature spherical-spin gap bound]\label{thm:fixed-temperature}
Fix $\beta>0$ such that
\begin{equation}\label{eq:fixed-beta-condition}
 F_\xi(\beta)>\beta E_1(\xi).
\end{equation}
Then
\begin{equation}\label{eq:fixed-temperature-bound}
 \pLiminf_{N\to\infty}
 \left(-\frac1N\log\gamma_{N,\beta}\right)
 \ge F_\xi(\beta)-\beta E_1(\xi).
\end{equation}
The proof uses no selected minimum, no selected gate, no well-Hessian determinant, and no Eyring--Kramers approximation.
\end{theorem}

\begin{proof}
Choose levels
\begin{equation}\label{eq:a0-a1-choice}
 E_1(\xi)<a_0<a_1<\frac{F_\xi(\beta)}\beta.
\end{equation}
On the event in \Cref{prop:deep-components}, let $s_{N,a_0}$ be the component sign.  Choose a smooth function
$\chi:\R\to[0,1]$ satisfying
\begin{equation}\label{eq:chi}
 \chi(u)=1\quad(u\le-a_1),
 \qquad
 \chi(u)=0\quad(u\ge-a_0),
 \qquad
 \|\chi'\|_\infty\le\frac{C}{a_1-a_0},
\end{equation}
and flat at the endpoints.  Define
\begin{equation}\label{eq:g-cutoff}
 g_N(\sigma)\defeq
 \begin{cases}
  s_{N,a_0}(\sigma)\,\chi(H_N(\sigma)/N),
   &\sigma\in\cL_N(a_0),\\
  0,&\sigma\notin\cL_N(a_0),
 \end{cases}
 \qquad
 f_N\defeq\frac{1+g_N}{2}.
\end{equation}
Because $\chi$ is flat where it vanishes, $g_N$ is smooth across the boundaries of the finitely many components.  Equations \eqref{eq:even-symmetry} and \eqref{eq:component-sign} give
$g_N(-\sigma)=-g_N(\sigma)$.  The Gibbs measure is antipodally invariant, hence
\begin{equation}\label{eq:mean-variance-cutoff}
 \int f_N\dd\pi_{N,\beta}=\frac12,
 \qquad
 \Var_{\pi_{N,\beta}}(f_N)
 =\frac14\int g_N^2\dd\pi_{N,\beta}
 \ge\frac14\pi_{N,\beta}(\cL_N(a_1)).
\end{equation}
By \Cref{prop:deep-mass} and \eqref{eq:a0-a1-choice}, the last probability converges to one.

The gradient is supported in the shell
\begin{equation}\label{eq:energy-shell}
 \cS_N(a_0,a_1)
 =\{-Na_1<H_N<-Na_0\},
\end{equation}
and
\begin{equation}\label{eq:gradient-cutoff}
 |\grad f_N|
 \le\frac{C}{N(a_1-a_0)}|\grad H_N|.
\end{equation}
On the derivative event \eqref{eq:derivative-event},
$|\grad f_N|^2\le C_{\xi,a_0,a_1}/N$.  Moreover,
$e^{-\beta H_N}\le e^{\beta Na_1}$ throughout the shell.  Therefore
\begin{align}
 \int_{\SphereN}|\grad f_N|^2\dd\pi_{N,\beta}
 &\le
 \frac{C_{\xi,a_0,a_1}}{N Z_{N,\beta}}
 \int_{\cS_N(a_0,a_1)}e^{-\beta H_N}\dd\nu_N
 \notag\\
 &\le
 \frac{C_{\xi,a_0,a_1}}N
 \exp\{\beta Na_1-NF_{N,\xi}(\beta)\}.
 \label{eq:cutoff-dirichlet}
\end{align}
Insert \eqref{eq:mean-variance-cutoff} and \eqref{eq:cutoff-dirichlet} into the Rayleigh quotient.  Using \eqref{eq:free-energy-limit} gives
\begin{equation}\label{eq:fixed-a1}
 \pLiminf_{N\to\infty}
 \left(-\frac1N\log\gamma_{N,\beta}\right)
 \ge F_\xi(\beta)-\beta a_1.
\end{equation}
Since $a_1$ can be chosen arbitrarily close to $E_1(\xi)$ from above, \eqref{eq:fixed-temperature-bound} follows.
\end{proof}

The inequality \eqref{eq:fixed-temperature-bound} is exactly the
statement of Theorem~\ref{thm:fixed} in the main text, whose proof is
therefore complete.

\begin{corollary}[The stronger order of limits]\label{cor:stronger-order}
The fixed-temperature result implies
\begin{equation}\label{eq:stronger-order}
 \liminf_{\beta\to\infty}
 \frac1\beta
 \pLiminf_{N\to\infty}
 \left(-\frac1N\log\gamma_{N,\beta}\right)
 \ge E_0(\xi)-E_1(\xi).
\end{equation}
Thus the same Arrhenius constant is obtained with $N\to\infty$ before $\beta\to\infty$.
\end{corollary}

\begin{proof}
For all sufficiently large $\beta$, \Cref{prop:free-energy-ground} verifies the condition of \Cref{thm:fixed-temperature}.  Divide \eqref{eq:fixed-temperature-bound} by $\beta$ and use \eqref{eq:F-over-beta}.
\end{proof}

\begin{remark}[What the unconditional fixed-temperature theorem measures]\label{rem:coarse-vs-EK}
The exponent in \eqref{eq:fixed-temperature-bound} is a free-energy cost for leaving the full collection of deep antipodal wells.  It already uses more spin-glass structure than a selected-gate formula and requires no well-mass or target-mass assumption.  It is intentionally coarse at the saddle scale: it bounds the entire transition shell by its deepest energy and therefore does not yet use the transverse widths of the index-one channels.  Those widths enter the aggregate Eyring--Kramers pressure in the next section.
\end{remark}

\section{Complexity-weighted Eyring--Kramers conductance}\label{sec:weighted-KR}

This section uses the Hessian and critical-point structure of the mixed spherical model to replace a determinant assumption at one selected saddle by an aggregate estimate over all index-one saddles.  The single structural difference from the pure case is that the conditional Hessian is a \emph{randomly} shifted GOE matrix: the scalar shift fluctuation, of variance $v_\xi^2/N$, participates in every exponential rate through the penalty $\mathfrak p_\xi$ of \eqref{eq:penalty}, and the rate functions of \Cref{subsec:complexity-thresholds} are the resulting one-dimensional suprema over the shift.

\subsection{The index-one saddle weight}

Let $z$ be an index-one critical point and write
\begin{align}
 Q_z&=\Hess H_N(z),\notag\\
 \spec(Q_z)
 &=\{-\alpha_z,\lambda_{z,2},\ldots,\lambda_{z,N-1}\},
 \qquad \alpha_z,\lambda_{z,j}>0.
 \label{eq:saddle-spectrum}
\end{align}
The local Eyring--Kramers weight is
\begin{equation}\label{eq:EK-weight}
 w_{N,\beta}(z)
 \defeq e^{-\beta H_N(z)}
 \frac{\alpha_z}{\sqrt{|\det Q_z|}}.
\end{equation}
The name records the functional form of the classical
Eyring--Kramers prefactor \citep{bovier2004metastability};
\eqref{eq:EK-weight} is a \emph{definition}, and no theorem of that
theory---which is developed for diffusions in Euclidean space---is
invoked anywhere in our proofs.  All spectral-gap bounds below come
from the variational principle \eqref{eq:spectral-gap} with explicit
test functions, which is intrinsic to the sphere.

For an energy interval $I\subset\R$, define the total index-one saddle weight
\begin{equation}\label{eq:W-def}
 \cW_{N,\beta}(I)
 \defeq\sum_{\substack{z:\ \grad H_N(z)=0,\ \ind Q_z=1\\
                   H_N(z)/N\in I}}
 e^{-\beta H_N(z)}
 \frac{\alpha_z}{\sqrt{|\det Q_z|}},
\end{equation}
and its geometrically normalized conductance
\begin{equation}\label{eq:K-def}
 \cK_{N,\beta}(I)
 \defeq\frac1{\vol(\SphereN)}
 \left(\frac{2\pi}{\beta}\right)^{(N-3)/2}
 \cW_{N,\beta}(I).
\end{equation}
The point of \eqref{eq:W-def} is that it sums every possible low-index channel; it does not require choosing a dominant saddle or knowing its basin adjacency.

\subsection{Shifted GOE and the determinant rate}

Let $G_{N-1}$ be a normalized GOE matrix with
\begin{equation}\label{eq:GOE-normalization}
 \E G_{ij}^2=\frac1{N-1}\quad(i\ne j),
 \qquad
 \E G_{ii}^2=\frac2{N-1}.
\end{equation}
Its empirical law converges to the semicircle density
\begin{equation}\label{eq:semicircle}
 \rho_{\sclaw}(x)=\frac1{2\pi}\sqrt{4-x^2}\,\mathbf1_{[-2,2]}(x).
\end{equation}

\begin{proposition}[Randomly shifted GOE Hessian of the mixed model]\label{prop:shifted-GOE}
Fix $\sigma\in\SphereN$ and condition on
\begin{equation}\label{eq:condition-H-grad}
 H_N(\sigma)=Nu,
 \qquad
 \grad H_N(\sigma)=0.
\end{equation}
Then, as an operator on $T_\sigma\SphereN$,
\begin{equation}\label{eq:shifted-GOE}
 \Hess H_N(\sigma)
 \stackrel{\mathrm d}=c_NG_{N-1}
 -\bigl(\xi'(1)u+\omega_N\bigr)I_{N-1},
 \qquad
 c_N=\sqrt{\frac{N-1}{N}\,\xi''(1)},
\end{equation}
where $\omega_N=v_\xi Z/\sqrt N$ with $Z$ a standard Gaussian
variable independent of the GOE matrix $G_{N-1}$, and both are
independent of the conditioned tangent gradient.  In the pure case
$v_\xi=0$ this is the deterministically shifted GOE law of
\citet[Lemma~3.2(b)]{auffinger2013random}, with
$c_N=\sqrt{\tfrac{N-1}N\,p(p-1)}$ and shift $-pu$.
\end{proposition}

\begin{proof}
Write $H_N=\sum_{p\in P}\gamma_pH_{N,p}$ with independent pure
components as in \eqref{eq:hamiltonian}.  For each component,
\citet[Lemma~3.2(b)]{auffinger2013random}, translated to the
radius-$\sqrt N$ sphere and the GOE convention
\eqref{eq:GOE-normalization}, gives
$\Hess H_{N,p}(\sigma)=c_N^{(p)}G^{(p)}_{N-1}-p\,(H_{N,p}(\sigma)/N)\,I$
with $c_N^{(p)}=\sqrt{\tfrac{N-1}N\,p(p-1)}$ and $G^{(p)}_{N-1}$ a
GOE matrix independent of $H_{N,p}(\sigma)$ and of the tangent
gradient; since the conditional law does not depend on the
conditioning value, this is an exact decomposition in law.  Summing
the independent components, the GOE parts add in variance to
$c_N^2=\tfrac{N-1}N\sum_p\gamma_p^2\,p(p-1)=\tfrac{N-1}N\,\xi''(1)$,
while the scalar part is
$-W I$ with $W=\sum_p p\,\gamma_p H_{N,p}(\sigma)/N$.  Conditionally
on $H_N(\sigma)=Nu$, the Gaussian vector
$(H_{N,p}(\sigma))_{p\in P}$ gives
$W\sim N(\xi'(1)u,\;v_\xi^2/N)$ with
$v_\xi^2=\sum_pp^2\gamma_p^2-(\sum_pp\gamma_p^2)^2$ as in
\eqref{eq:v-xi}; the conditional-variance computation, together with
the identification of $W$ as the radial derivative
$\langle\sigma,\nabla H_N(\sigma)\rangle/N$, is recorded in
\Cref{app:geometry}.  The residual fluctuation of $W$ is a function
of the energies $(H_{N,p}(\sigma))$ and is therefore independent of
the GOE parts and of the gradient.
\end{proof}

Define the semicircle logarithmic potential
\begin{equation}\label{eq:Omega}
 \Omega(t)\defeq\int_{-2}^{2}\log|x-t|\rho_{\sclaw}(x)\dd x,
\end{equation}
so that the determinant rate of \eqref{eq:Dxi} is
\begin{equation}\label{eq:Dp}
 D_\xi(u,\omega)
 =\frac12\log\xi''(1)
 +\Omega\bigl(t_\xi(u,\omega)\bigr),
 \qquad
 D_\xi(u)=D_\xi(u,0),
\end{equation}
with $t_\xi(u,\omega)=(\xi'(1)u+\omega)/\sqrt{\xi''(1)}$ the shift
location of \eqref{eq:t-shift}.  An explicit expression is
\begin{equation}\label{eq:Omega-explicit}
 \Omega(t)=
 \begin{cases}
  \dfrac{t^2}{4}-\dfrac12,&|t|\le2,\\[0.7em]
  \dfrac{t^2}{4}-\dfrac12
  -\dfrac{|t|}{4}\sqrt{t^2-4}
  +\log\!\left(\dfrac{|t|+\sqrt{t^2-4}}2\right),&|t|>2.
 \end{cases}
\end{equation}

\begin{proposition}[Fixed-location spherical determinant]\label{prop:fixed-location-det}
Let $u_N\to u$ and condition at a deterministic $\sigma_N\in\SphereN$ on
$H_N(\sigma_N)=Nu_N$ and $\grad H_N(\sigma_N)=0$.  Then
\begin{equation}\label{eq:fixed-location-det}
 \frac1N\log\left|\det\Hess H_N(\sigma_N)\right|
 \longrightarrow D_\xi(u)
 \qquad\text{in probability}.
\end{equation}
\end{proposition}

\begin{proof}
Under \eqref{eq:shifted-GOE}, the normalized log determinant is the empirical integral of
$\log|c_Nx-\xi'(1)u_N-\omega_N|$.  The random shift $\omega_N=v_\xi Z/\sqrt N$ tends to zero in probability, so it does not affect a limit in probability: truncate the logarithm at distance $\eta$ from the shift, use the semicircle law for the bounded truncation, and remove $\eta$ by GOE eigenvalue-counting and level-repulsion estimates.  The details are recorded in \Cref{app:weighted-KR}.  (The fluctuation $\omega_N$ matters only under exponential-scale expectations, where it is tilted to order-one values at the speed-$N$ cost $\mathfrak p_\xi(\omega)$; this is the content of the rate functionals \eqref{eq:Lambda-rk}.)  This proposition is not used by selecting one gate saddle; instead, its bulk rate reappears inside the aggregate Kac--Rice calculation below.
\end{proof}

\subsection{The Eyring--Kramers conductance complexity}

Define the Eyring--Kramers conductance upper complexity by the
half-determinant tilted rate of \eqref{eq:Lambda-rk}:
\begin{equation}\label{eq:Xi-EK}
 \Xi^{\mathrm{EK}}_{1,\xi}(u)
 \defeq\varrho_\xi(u)+\Lambda^{(1/2)}_{1,\xi}(u)
 =\varrho_\xi(u)
 +\sup_{\omega\in\R}
 \Bigl\{-\mathfrak p_\xi(\omega)
 +\tfrac12D_\xi(u,\omega)
 -J\bigl(t_\xi(u,\omega)\bigr)\Bigr\}.
\end{equation}
In the pure case $\xi(x)=x^p$, where $v_\xi=0$, the supremum collapses at $\omega=0$ and
\begin{equation}\label{eq:Xi-EK-explicit}
 \Xi^{\mathrm{EK}}_{1,\xi}(u)
 =\Theta_{1,\xi}(u)-\frac12D_\xi(u)
 =\frac12\log(p-1)
 -\frac{p-2}{4(p-1)}u^2
 -2I_{1,\xi}(u)-\frac12D_\xi(u),
 \quad u<-E_\infty(\xi);
\end{equation}
for genuine mixtures the shift optimization couples the complexity
and determinant terms, and only inequalities survive.  Let
$\omega^{\mathrm{crt}}_u$ and $\omega^{\mathrm{EK}}_u$ be maximizers
of the suprema defining $\Lambda^{(1)}_{1,\xi}(u)$ and
$\Lambda^{(1/2)}_{1,\xi}(u)$ respectively; they exist by
\Cref{lem:B-Lambda-regularity}.  Evaluating each objective at the
other's maximizer gives
\begin{equation}\label{eq:Xi-vs-Theta}
 \Theta_{1,\xi}(u)-\tfrac12D_\xi\bigl(u,\omega^{\mathrm{crt}}_u\bigr)
 \ \le\
 \Xi^{\mathrm{EK}}_{1,\xi}(u)
 \ \le\
 \Theta_{1,\xi}(u)-\tfrac12D_\xi\bigl(u,\omega^{\mathrm{EK}}_u\bigr)
 \ \le\
 \Theta_{1,\xi}(u)-\tfrac12\Bigl(\tfrac12\log\xi''(1)-\tfrac12\Bigr),
\end{equation}
where the last step uses the global bound
$\Omega(t)\ge\Omega(0)=-\tfrac12$, read off from
\eqref{eq:Omega-explicit}, so that
$D_\xi(u,\omega)\ge\tfrac12\log\xi''(1)-\tfrac12$ for \emph{every}
shift $\omega$.  The sharper value $\Omega(-2)=\tfrac12$, which would
give the constant $\tfrac12(\tfrac12\log\xi''(1)+\tfrac12)$, is
available only when the half-determinant maximizer places the
Hessian shift outside the semicircle support,
$t_\xi(u,\omega^{\mathrm{EK}}_u)\le-2$.  Although the index-one
event forces one eigenvalue below the shift, the \emph{optimizing}
shift of the half-determinant problem can lie inside the bulk even
at $u=-E_1(\xi)$, and then the sharper constant fails
(\Cref{rem:Xi-constant-example}).  Under \Cref{ass:mixture},
$\xi''(1)\ge12$, so the certified constant in
\eqref{eq:Xi-vs-Theta} is
$\tfrac14\log\xi''(1)-\tfrac14\ge\tfrac14\log12-\tfrac14>0.37$.
The subtraction of one half of the
determinant rate has a direct Kac--Rice origin.  The critical-point
Jacobian contributes $|\det Q_z|$, whereas the local conductance
contributes $\alpha_z/\sqrt{|\det Q_z|}$.  Their product is
\begin{equation}\label{eq:KR-cancellation}
 |\det Q_z|\,
 \frac{\alpha_z}{\sqrt{|\det Q_z|}}
 =\alpha_z\sqrt{|\det Q_z|}.
\end{equation}
Thus no inverse determinant remains in the one-point expectation.  Soft stable eigenvalues suppress rather than enlarge the weighted Kac--Rice integrand.

\begin{theorem}[Weighted Kac--Rice pressure for index-one saddles]\label{thm:weighted-KR}
Fix $\beta>0$ and a compact interval
$I\subset(-\infty,0)$.  Then
\begin{equation}\label{eq:weighted-KR-expectation}
 \limsup_{N\to\infty}\frac1N
 \log\E\cW_{N,\beta}(I)
 \le
 \sup_{u\in I}
 \left\{-\beta u+\Xi^{\mathrm{EK}}_{1,\xi}(u)\right\}.
\end{equation}
Consequently,
\begin{equation}\label{eq:weighted-KR-prob}
 \pLimsup_{N\to\infty}\frac1N
 \log\cW_{N,\beta}(I)
 \le
 \sup_{u\in I}
 \left\{-\beta u+\Xi^{\mathrm{EK}}_{1,\xi}(u)\right\}.
\end{equation}
Here and below $\log0=-\infty$.
\end{theorem}

\begin{proof}
Apply the marked one-point Kac--Rice formula to \eqref{eq:W-def}.
Two points of rigor are recorded at the outset.  First, the mark
$\alpha(Q)\sqrt{|\det Q|}\,\mathbf1_{\{\ind Q=1\}}$ is a nonnegative
Borel function of the Hessian that is unbounded; the Kac--Rice
identity for it is obtained by applying the bounded-mark identity to
the truncations $\min\{\cdot,m\}$ and letting $m\to\infty$ by
monotone convergence on both sides (the same argument is used, and
spelled out, in the proof of \Cref{prop:soft-marked-KR}).  Second,
the prose estimates in this proof and in \Cref{lem:half-det} are
conveniences for readability; the quantitative statement this
theorem ultimately consumes is \Cref{lem:B-matrix-rates}(b), which
bounds the conditional matrix expectation uniformly on compact
energy sets.  Conditional on energy $Nu$ and zero tangent gradient, the Hessian has the randomly shifted GOE law \eqref{eq:shifted-GOE}.  By \eqref{eq:KR-cancellation}, the matrix expectation is
\[
 \E
 \left[
  \alpha(Q)\sqrt{|\det Q|}\,
  \mathbf1_{\{\ind Q=1\}}
 \right],
 \qquad Q=c_NG_{N-1}-\bigl(\xi'(1)u+\omega_N\bigr)I,
\]
the expectation running over the GOE matrix and the shift
fluctuation $\omega_N$.  Conditionally on
$\omega_N=\omega$, the empirical spectral measure has speed-$N^2$
concentration at the semicircle law, the single outlier producing
index one has the speed-$N$ cost $J(t_\xi(u,\omega))$, the bulk half
determinant contributes at most $D_\xi(u,\omega)/2$, and the
unstable eigenvalue $\alpha(Q)$ contributes no positive $N$-scale
logarithm; integrating over $\omega$ at the Gaussian cost
$\mathfrak p_\xi(\omega)$ (\Cref{lem:B-tilt-laplace}) yields the
rate $\Lambda^{(1/2)}_{1,\xi}(u)$.  Compared with the index-one
count, whose mark carries the full determinant power and produces
$\Lambda^{(1)}_{1,\xi}(u)$, the half-determinant mark lowers the
matrix exponent exactly as in \eqref{eq:Xi-vs-Theta}.  Adding the
Kac--Rice prefactor rate $\varrho_\xi(u)$ gives the local upper rate
\eqref{eq:Xi-EK}.

Partition $I$ into finitely many bins, bound $e^{-\beta H}$ by its maximum in each bin, apply the marked Kac--Rice identity \eqref{eq:B-KR-identity} together with the uniform rate \Cref{lem:B-matrix-rates}(b) on each bin, and let the mesh tend to zero, using the continuity of $\Xi^{\mathrm{EK}}_{1,\xi}$ (\Cref{lem:B-Lambda-regularity}).  This proves \eqref{eq:weighted-KR-expectation}.  The GOE truncation and the shift-Laplace step are given in detail in \Cref{app:weighted-KR}.  Finally, Markov's inequality upgrades the annealed upper bound to \eqref{eq:weighted-KR-prob}.
\end{proof}

\begin{corollary}[Geometrically normalized saddle pressure]\label{cor:K-pressure}
Under the assumptions of \Cref{thm:weighted-KR},
\begin{equation}\label{eq:K-pressure}
 \pLimsup_{N\to\infty}\frac1N
 \log\cK_{N,\beta}(I)
 \le
 -\frac12\log(\beta e)
 +\sup_{u\in I}
 \left\{-\beta u+\Xi^{\mathrm{EK}}_{1,\xi}(u)\right\}.
\end{equation}
\end{corollary}

\begin{proof}
Stirling's formula gives
\begin{equation}\label{eq:sphere-volume-asymptotic}
 \lim_{N\to\infty}\frac1N\log\vol(\SphereN)
 =\frac12\log(2\pi e).
\end{equation}
Also,
\[
 \lim_{N\to\infty}\frac1N
 \log\left(\frac{2\pi}{\beta}\right)^{(N-3)/2}
 =\frac12\log\frac{2\pi}{\beta}.
\]
Combine these limits with \eqref{eq:weighted-KR-prob} and \eqref{eq:K-def}.
\end{proof}

\begin{remark}[Why this is stronger than a selected determinant assumption]\label{rem:aggregate-vs-selected}
The weighted pressure does not assert that a saddle selected by connectivity has the fixed-location law \eqref{eq:fixed-location-det}.  It calculates the whole marked sum before any saddle is selected.  The Kac--Rice Jacobian removes the singular inverse determinant, and the complexity term $\Theta_{1,\xi}$ explicitly pays for the exponential number of index-one channels.  This is the principal advantage of the aggregate formulation over a selected-gate determinant premise.
\end{remark}

\section{The aggregate Eyring--Kramers refinement: proof of Theorem~\ref{thm:ek}}\label{sec:network}

The unconditional theorem in \Cref{sec:deep-cutoff} requires no well-mass or target-mass assumption, while the weighted Kac--Rice theorem removes the need to transfer a determinant law to a connectivity-selected saddle.  Passing from the aggregate saddle pressure to a continuum trial potential has four logically distinct parts.  The low-index handlebody must admit a finite separating co-core cut; curves that leave the low sublevel must be charged at energy cost $\exp\{\beta Nb+o(N)\}$; the local profiles must be glued without paying for overlap; and the aggregate energy of the co-core profiles must be comparable to the Eyring--Kramers sum.  The first three statements are deterministic consequences of Morse theory, the derivative bound, and weighted $2$-modulus.  For $b$ sufficiently close to $E_1(\xi)$, the last comparison is formulated as \Cref{thm:cocore-localization}.  Appendix~\ref{app:near-threshold-cocore} proves its deterministic and good-saddle parts and supplies the marked Kac--Rice estimate that controls the soft-saddle contribution.

For $a>E_1(\xi)$, on the event in \Cref{prop:deep-components}, define the antipodal deep-well unions
\begin{equation}\label{eq:A-B-deep}
 \cA_N(a)\defeq\{\sigma\in\cL_N(a):s_{N,a}(\sigma)=1\},
 \qquad
 \cB_N(a)\defeq-\cA_N(a).
\end{equation}
They satisfy
\begin{equation}\label{eq:A-B-masses}
 \pi_{N,\beta}(\cA_N(a))
 =\pi_{N,\beta}(\cB_N(a))
 =\frac12\pi_{N,\beta}(\cL_N(a)).
\end{equation}
For a smooth $f:\SphereN\to\R$, write
\begin{equation}\label{eq:unnormalized-energy}
 \cE_{N,\beta}(f)
 =\int_{\SphereN}|\grad f|^2e^{-\beta H_N}\dd\nu_N.
\end{equation}
Throughout this section,
\begin{equation}\label{eq:b-range}
 E_2(\xi)<b<E_1(\xi),
\end{equation}
and, for $\delta>0$, we abbreviate
\begin{equation}\label{eq:saddle-interval}
 I_{\xi,b,\delta}=[-E_1(\xi)-\delta,-b].
\end{equation}

\begin{proposition}[Uniform stable directions in the low-index layer]\label{prop:stable-gap-layer}
Fix $b>E_2(\xi)$.  There is $\kappa_{\xi,b}>0$ such that
\begin{equation}\label{eq:stable-gap-layer}
 \Pp\left(
 \begin{array}{c}
  \exists z:\quad \grad H_N(z)=0,\quad H_N(z)\le-Nb,\\[-0.1em]
  \ind Q_z=1,\quad \lambda_2(Q_z)\le\kappa_{\xi,b}
 \end{array}
 \right)\longrightarrow0,
\end{equation}
where $\lambda_2(Q_z)$ is the smallest positive eigenvalue of the index-one Hessian.
\end{proposition}

\begin{proof}
By Markov's inequality, the probability in \eqref{eq:stable-gap-layer} is bounded by the expected number of index-one critical points below $-Nb$ whose smallest positive Hessian eigenvalue is at most $\kappa$.  \Cref{lem:B-soft-stable-count}, proved in \Cref{app:weighted-KR} by an explicit two-outlier Kac--Rice estimate, bounds this expectation by $e^{-c_{\xi,b}N}$ for a suitable $\kappa_{\xi,b}>0$: the shifted-GOE law \eqref{eq:shifted-GOE} converts the Hessian event into $\lambda_2(G_{N-1})\le t_N(u)+\kappa/c_N$, the uniform two-outlier large-deviation estimate of \Cref{lem:B-two-outlier} supplies, at each value $\omega$ of the shift fluctuation, the cost $2J(t_\xi(u,\omega))-O(\kappa)$, the scalar-tilt Laplace bound of \Cref{lem:B-tilt-laplace} integrates over $\omega$ at the penalty $\mathfrak p_\xi(\omega)$, and the resulting Kac--Rice exponent is at most $\Theta_{2,\xi}(-b)+O(\kappa)<0$ for $\kappa$ small, uniformly over the entire energy range because the Gaussian energy density controls $u\le-C_\xi$ directly.  This proves the claim.
\end{proof}

\subsection{The separating co-core cut}

Set
\[
 \overline\Omega_N(b)=\{\sigma\in\SphereN:H_N(\sigma)\le-Nb\},
 \qquad
 \Omega_N(b)=\operatorname{int}\overline\Omega_N(b)
 =\{H_N<-Nb\}.
\]
A co-core of a one-handle $[-1,1]\times\mathbb D^{N-2}$ is the disk
$\{0\}\times\mathbb D^{N-2}$.  Co-cores are taken neatly and properly embedded in the compact regular sublevel $\overline\Omega_N(b)$, with boundary on $\partial\overline\Omega_N(b)$.

\begin{lemma}[Topological co-core cut]\label{lem:topological-cocore-cut}
Fix $a>E_1(\xi)$ and $\delta>0$.  On the event that
\begin{enumerate}[label=\textup{(\roman*)},leftmargin=*]
 \item $H_N$ is Morse and the levels $-Na$ and $-Nb$ are regular;
 \item every critical point below $-Nb$ has index zero or one;
 \item no index-one critical point lies below $-N(E_1(\xi)+\delta)$;
 \item every component of $\cL_N(a)$ contains exactly one critical point, which is a minimum,
\end{enumerate}
there are a finite set $\cC_N=\cC_N(a,b,\delta)$ of index-one critical points, pairwise disjoint co-core disks $(D_z)_{z\in\cC_N}$, and pairwise disjoint bicollars
\[
 \Psi_z:D_z\times(-2,2)\longrightarrow\overline\Omega_N(b),
 \qquad \Psi_z(x,0)=x,
\]
whose closures are disjoint from $\cA_N(a)\cup\cB_N(a)$, such that
\begin{equation}\label{eq:cut-saddle-range}
 \frac{H_N(z)}N\in I_{\xi,b,\delta},
 \qquad z\in\cC_N.
\end{equation}
Moreover, every path in $\Omega_N(b)$ from $\cA_N(a)$ to $\cB_N(a)$ contains, for some $z\in\cC_N$, a subpath in the bicollar joining the two faces
$\Psi_z(D_z\times\{-1\})$ and $\Psi_z(D_z\times\{1\})$.  If the two boundary sets already lie in different components of $\Omega_N(b)$, one may take $\cC_N=\varnothing$.
\end{lemma}

\begin{proof}
Consider the handle decomposition of the regular sublevel
$\{H_N\le-Nb\}$.  By (ii), it consists only of zero-handles and one-handles.  Collapsing every zero-handle to a vertex and every one-handle to its core interval gives a deformation retraction onto a finite graph $\Gamma_N$.  By (iv), each component of $\cL_N(a)$ contains a unique minimum and therefore determines one vertex of this graph.  Let $V_A$ and $V_B$ be the disjoint vertex sets determined by the components forming $\cA_N(a)$ and $\cB_N(a)$, respectively.

Choose an inclusion-minimal edge cut $C_N$ separating $V_A$ from $V_B$ in $\Gamma_N$; if no graph path joins the two vertex sets, take the empty cut.  For every edge $e\in C_N$, let $z_e$ be the critical point whose one-handle generates $e$, set
$\cC_N=\{z_e:e\in C_N\}$, and choose the standard co-core disk $D_{z_e}$ in that handle.  The handle interiors, and hence the co-cores, may be chosen pairwise disjoint.  Removing these disks from the handlebody is the geometric counterpart of deleting the corresponding edge interiors from $\Gamma_N$: the complement deformation retracts onto $\Gamma_N\setminus C_N$.  Since $C_N$ separates $V_A$ from $V_B$, no path in the complement joins $\cA_N(a)$ to $\cB_N(a)$.  Thus every such path crosses one of the selected co-cores and, for at least one selected disk, changes from one component of its local complement to the other.  Choose pairwise disjoint bicollars of the disks, small enough that their closures avoid the two deep boundary sets.  The lateral boundary of a proper bicollar lies on the boundary of the handlebody.  Hence a path contained in the open sublevel and changing local side contains a subpath that enters through one collar face and leaves through the opposite face.  Rescaling the collar coordinate gives the faces at parameters $-1$ and $1$ in the statement.

Every selected handle occurs below $-Nb$, so $H_N(z_e)<-Nb$.  Assumption (iii) gives
$H_N(z_e)>-N(E_1(\xi)+\delta)$.  This proves \eqref{eq:cut-saddle-range}.
\end{proof}

The preceding lemma supplies the separating disks but no quantitative control of their metric geometry.  The next elementary lemma records that smooth charging profiles nevertheless always exist.

\begin{lemma}[Smooth co-core charging profiles]\label{lem:cocore-profiles}
Under the assumptions of \Cref{lem:topological-cocore-cut}, there are nonnegative smooth functions
$\rho_{N,z}:\SphereN\to[0,\infty)$, $z\in\cC_N$, supported in arbitrarily small ambient neighborhoods of the corresponding bicollars, with the following property: whenever a rectifiable path segment $\gamma$ in the $z$-bicollar joins
$\Psi_z(D_z\times\{-1\})$ to $\Psi_z(D_z\times\{1\})$,
\begin{equation}\label{eq:cocore-charge}
 \int_\gamma\rho_{N,z}\,\dd s\ge1.
\end{equation}
The bicollars, and hence the supports, may be chosen arbitrarily thin.  No assertion about the weighted energies of these profiles is made here.
\end{lemma}

\begin{proof}
A co-core is a compact, two-sided, neatly embedded hypersurface in the handlebody.  Extend $D_z$ slightly across $\partial\overline\Omega_N(b)$ to a smooth disk $\widetilde D_z\subset\SphereN$ whose interior contains $D_z$.  The tubular-neighborhood theorem supplies coordinates $(x,t)$ on a neighborhood of $\widetilde D_z$, agreeing with the stated bicollar on $D_z\times(-2,2)$.  Choose a smooth cutoff $\chi_z$ on $\widetilde D_z$ that equals one on $D_z$ and has compact support in the interior of $\widetilde D_z$.  Also choose a smooth nondecreasing function
$\eta:\R\to[0,1]$ with $\eta(t)=0$ for $t\le-1$,
$\eta(t)=1$ for $t\ge1$, and $\eta'$ supported in $(-3/2,3/2)$.  In the tubular neighborhood set
\[
 \rho_{N,z}(x,t)=\chi_z(x)\eta'(t)|\grad t|
\]
and extend it by zero.  The cutoff and the support condition on $\eta'$ make this extension smooth.  Along the original bicollar one has $\chi_z=1$, and therefore, if $\gamma$ joins the two specified faces, the chain rule and total variation give
\[
 \int_\gamma\rho_{N,z}\,\dd s
 \ge
 \left|\eta(t_z(\gamma(t_+)))-
       \eta(t_z(\gamma(t_-)))\right|=1.
\]
The extensions and tubular neighborhoods may be chosen arbitrarily small, which gives the final assertion.
\end{proof}

\subsection{The high-level extension and overlap-free gluing}

The contribution of paths reaching the level $-Nb$ is unconditional.

\begin{lemma}[High-level charging density]\label{lem:high-level-density}
Fix $a>b$ and a fixed $\beta>0$.  On the derivative event \eqref{eq:derivative-event}, there is a smooth nonnegative function $\rho_N^{\mathrm{hi}}$ such that every rectifiable path starting in $\cA_N(a)\cup\cB_N(a)$ and meeting $\{H_N\ge-Nb\}$ satisfies
\begin{equation}\label{eq:high-level-charge}
 \int_\gamma\rho_N^{\mathrm{hi}}\,\dd s\ge1,
\end{equation}
and
\begin{equation}\label{eq:high-level-energy}
 \int_{\SphereN}(\rho_N^{\mathrm{hi}})^2
 e^{-\beta H_N}\dd\nu_N
 \le \exp\{\beta Nb+q_N\},
 \qquad q_N=o(N).
\end{equation}
For example, one may take $q_N=\beta N^{2/3}+O(\log N)$.
\end{lemma}

\begin{proof}
Put $\ell_N=N^{2/3}$ and choose a smooth nondecreasing function
$\vartheta_N:\R\to[0,1]$ such that
\[
 \vartheta_N(t)=0\quad(t\le-Nb-\ell_N),
 \qquad
 \vartheta_N(t)=1\quad(t\ge-Nb),
 \qquad
 \|\vartheta_N'\|_\infty\le\frac2{\ell_N}.
\]
Set
\[
 \rho_N^{\mathrm{hi}}
 =|\grad(\vartheta_N\circ H_N)|.
\]
Since $a>b$, for all sufficiently large $N$ one has
$-Na<-Nb-\ell_N$.  If a path starts in either deep-well union and reaches the level $-Nb$, the fundamental theorem of calculus along the path yields \eqref{eq:high-level-charge}.

On \eqref{eq:derivative-event},
\[
 (\rho_N^{\mathrm{hi}})^2
 \le \frac{4}{\ell_N^2}|\grad H_N|^2
 \le C_\xi\frac{N}{\ell_N^2}=C_\xi N^{-1/3}.
\]
The density is supported in the shell
$\{-Nb-\ell_N\le H_N\le-Nb\}$.  Since $\nu_N$ is normalized,
\begin{align*}
 \int_{\SphereN}(\rho_N^{\mathrm{hi}})^2
 e^{-\beta H_N}\dd\nu_N
 &\le C_\xi N^{-1/3}
       \exp\{\beta Nb+\beta\ell_N\}\\
 &\le \exp\{\beta Nb+\beta N^{2/3}+O(\log N)\}.
\end{align*}
This is \eqref{eq:high-level-energy}.
\end{proof}

We next give the precise deterministic gluing statement.  It is the weighted capacity--modulus implication specialized to the form needed here.

\begin{lemma}[Weighted modulus produces a smooth trial potential]\label{lem:weighted-modulus}
Let $M$ be a compact connected Riemannian manifold, let $A,B\subset M$ be disjoint compact sets, let $w$ be a positive smooth weight, and let $\rho\ge0$ be continuous.  Suppose every rectifiable path $\gamma$ from $A$ to $B$ satisfies
\begin{equation}\label{eq:modulus-admissible}
 \int_\gamma\rho\,\dd s\ge1.
\end{equation}
Then, for every $\eps>0$, there is $f\in C^\infty(M;[0,1])$, equal to one on $A$ and zero on $B$, such that
\begin{equation}\label{eq:modulus-energy}
 \int_M|\grad f|^2w\,\dd\vol_M
 \le (1+\eps)\int_M\rho^2w\,\dd\vol_M+\eps.
\end{equation}
Moreover, $f$ may be chosen constant on neighborhoods of $A$ and $B$.
\end{lemma}

\begin{proof}
For $x\in M$, define the $\rho$-length distance from $B$ by
\[
 d_\rho(x,B)=\inf_{\gamma:B\to x}\int_\gamma\rho\,\dd s,
 \qquad
 u(x)=\min\{1,d_\rho(x,B)\}.
\]
Because $\rho$ is continuous on the compact manifold, $u$ is Lipschitz.  It is zero on $B$, and \eqref{eq:modulus-admissible} implies $u=1$ on $A$.  The standard upper-gradient argument gives
\begin{equation}\label{eq:upper-gradient}
 |\grad u|\le\rho
 \quad\text{almost everywhere.}
\end{equation}
Indeed, along every minimizing geodesic segment from $x$ to a nearby point $y$,
$u(y)-u(x)$ is bounded by the $\rho$-length of the segment; differentiation at almost every $x$ yields \eqref{eq:upper-gradient}.

Fix $0<\eta<1/4$ and choose a smooth nondecreasing map
$\chi_\eta:[0,1]\to[0,1]$ that is zero on $[0,\eta]$, one on $[1-\eta,1]$, and satisfies
$\|\chi_\eta'\|_\infty\le(1-2\eta)^{-1}(1+\eta)$.  Then
$v_\eta=\chi_\eta\circ u$ is constant on neighborhoods of $A$ and $B$ and
\[
 \int_M|\grad v_\eta|^2w\,\dd\vol_M
 \le \frac{(1+\eta)^2}{(1-2\eta)^2}
      \int_M\rho^2w\,\dd\vol_M.
\]
Smooth $v_\eta$ only on the complement of smaller neighborhoods on which it is already constant.  Standard local mollification and a partition of unity give smooth functions converging to $v_\eta$ in the weighted $W^{1,2}$ norm; truncation by a smooth map into $[0,1]$ preserves the boundary values and does not increase the limiting energy.  First choose $\eta$ sufficiently small and then the smoothing error sufficiently small to obtain \eqref{eq:modulus-energy}.
\end{proof}

The square-sum construction in the next proof is the point at which overlap is eliminated: no disjointness, bounded multiplicity, or coloring of the saddle neighborhoods is required.

For $b$ in \eqref{eq:b-range}, define the model-specific aggregate conductance pressure
\begin{equation}\label{eq:Lambda-EK}
 \Lambda^{\mathrm{EK}}_{\xi,\beta}(b)
 \defeq-\frac12\log(\beta e)
 +\sup_{u\in[-E_1(\xi),-b]}
 \left\{-\beta u+\Xi^{\mathrm{EK}}_{1,\xi}(u)\right\}.
\end{equation}

\begin{theorem}[Near-threshold quantitative co-core localization]\label{thm:cocore-localization}
There exists a constant
\begin{equation}\label{eq:b-star-range}
 b_*(\xi)\in(E_2(\xi),E_1(\xi))
\end{equation}
such that the following holds.  (The proof chooses $b_*(\xi)$ through
\eqref{eq:D-bstar}, which forces
$b_*(\xi)\ge E_1(\xi)-\eta_0/4$ with $\eta_0=\kappa c_0^2/16$ from
\Cref{lem:D-stable-manifold}; since $\kappa$ and $c_0$ are not
numerically explicit, the admissible window
$(b_*(\xi),E_1(\xi))$ is a left neighborhood of $E_1(\xi)$ of
uncontrolled, possibly very small, width.)  Fix
\begin{equation}\label{eq:near-threshold-b}
 b\in(b_*(\xi),E_1(\xi)).
\end{equation}
There is a constant $C_{\xi,b}<\infty$ with the following property; here and in \Cref{thm:network-realization,thm:refined-EK} the transverse width parameter $K\ge1$ of \Cref{lem:D-in-chart-cocore} is fixed once and for all at $K=1$, and $C_{\xi,b}$ depends on this choice.  For every fixed sufficiently large $\beta$, every
$a\in(E_1(\xi),F_\xi(\beta)/\beta)$ with $a-E_1(\xi)\ge(E_0(\xi)-E_1(\xi))/4$, and every $\delta>0$, one may take the cut $\cC_N$ to consist of \emph{all} index-one critical points of $H_N$ below $-Nb$, with the stable-disk co-cores, slabs, and smooth charging densities $(\rho_{N,z})_{z\in\cC_N}$ supplied by \Cref{lem:D-stable-manifold,lem:D-in-chart-cocore,lem:D-profile-energy}.  These densities charge every low crossing in the sense of \Cref{lem:D-in-chart-cocore}, every $z\in\cC_N$ satisfies \eqref{eq:cut-saddle-range} with $\delta_\beta$ in place of $\delta$, and, with probability tending to one,
\begin{equation}\label{eq:cocore-localization}
 \sum_{z\in\cC_N}
 \int_{\SphereN}\rho_{N,z}^2e^{-\beta H_N}\dd\nu_N
 \le
 \exp\left\{\frac{C_{\xi,b}N}{\sqrt\beta}+s_N\right\}
 \Bigl(\cK_{N,\beta}(I_{\xi,b,\delta})
 +e^{N\Lambda^{\mathrm{EK}}_{\xi,\beta}(b)}\Bigr),
\end{equation}
where $s_N=o_{\Pp}(N)$.
If the cut is empty, the left-hand side is understood to be zero.  The
deterministic additive term has exactly the pressure that already
appears in \Cref{thm:refined-EK} and is therefore harmless there.
\end{theorem}

\begin{proof}
The proof occupies \Cref{app:near-threshold-cocore}: the quantitative
stable-manifold geometry is \Cref{lem:D-stable-manifold}, the
separation and slab crossing is \Cref{lem:D-in-chart-cocore}, the
profile construction and energy bound is \Cref{lem:D-profile-energy},
the good-saddle summation is \eqref{eq:D-good-sum}, and the
soft-saddle contribution is controlled through the annealed marked
Kac--Rice estimate \Cref{prop:soft-marked-KR}, with the final assembly
in the last subsection of that appendix.
\end{proof}

\begin{theorem}[Global trial function from co-core localization]\label{thm:network-realization}
Fix $b\in(b_*(\xi),E_1(\xi))$, with $b_*(\xi)$ as in \Cref{thm:cocore-localization}.  For every fixed sufficiently large $\beta$, every
$a\in(E_1(\xi),F_\xi(\beta)/\beta)$ with $a-E_1(\xi)\ge(E_0(\xi)-E_1(\xi))/4$, and every $\delta>0$, there are random smooth functions $f_N:\SphereN\to[0,1]$ and random errors $r_N=o_{\Pp}(N)$ such that, with probability tending to one,
\begin{equation}\label{eq:network-boundary}
 f_N=1\text{ on }\cA_N(a),
 \qquad
 f_N=0\text{ on }\cB_N(a),
\end{equation}
and
\begin{align}
 \cE_{N,\beta}(f_N)
 \le{}&
 \exp\left\{\frac{C_{\xi,b}N}{\sqrt\beta}+r_N\right\}
 \Bigl(\cK_{N,\beta}(I_{\xi,b,\delta})
 +e^{N\Lambda^{\mathrm{EK}}_{\xi,\beta}(b)}\Bigr)\notag\\
 &+\exp\{\beta Nb+r_N\}.
 \label{eq:network-realization}
\end{align}
\end{theorem}

\begin{proof}
Work on the intersection of the almost-sure Morse and regular-level event with the events in
\Cref{prop:deep-components,prop:layering,prop:derivative-scale} and
\Cref{thm:cocore-localization}.  Their probability tends to one.  Let
$\cC_N$ and $(\rho_{N,z})_{z\in\cC_N}$ be as in \Cref{thm:cocore-localization}, and let
$\rho_N^{\mathrm{hi}}$ be the density from \Cref{lem:high-level-density}.  Define
\begin{equation}\label{eq:square-sum-density}
 \rho_N\defeq
 \left[(\rho_N^{\mathrm{hi}})^2+
       \sum_{z\in\cC_N}\rho_{N,z}^2\right]^{1/2}.
\end{equation}
This density is continuous even when the supports overlap.

Let $\gamma$ be a rectifiable path from $\cA_N(a)$ to $\cB_N(a)$.  If it meets
$\{H_N\ge-Nb\}$, then \Cref{lem:high-level-density} gives
$\int_\gamma\rho_N\,\dd s\ge1$.  Otherwise the whole path lies in
$\Omega_N(b)$; the stable-disk separation of \Cref{lem:D-in-chart-cocore} then supplies a $z\in\cC_N$ whose truncated stable disk $\gamma$ meets, together with a subpath of $\gamma$ running inside the slab of $z$ from one transverse face to the disk and back to a transverse face, and the half-mass charging property of \Cref{lem:D-profile-energy} gives
\[
 \int_\gamma\rho_N\,\dd s
 \ge\int_\gamma\rho_{N,z}\,\dd s\ge\tfrac12+\tfrac12=1 .
\]
Thus $\rho_N$ is admissible in the sense of \eqref{eq:modulus-admissible}.

Apply \Cref{lem:weighted-modulus} with
$w=e^{-\beta H_N}$ and with a smoothing error, say, $e^{-N^2}$.  This gives a smooth $f_N$ satisfying \eqref{eq:network-boundary} and
\begin{align*}
 \cE_{N,\beta}(f_N)
 &\le (1+e^{-N^2})
 \left[
  \int_{\SphereN}(\rho_N^{\mathrm{hi}})^2e^{-\beta H_N}\dd\nu_N
  +\sum_{z\in\cC_N}
   \int_{\SphereN}\rho_{N,z}^2e^{-\beta H_N}\dd\nu_N
 \right]\\
 &\quad+e^{-N^2}.
\end{align*}
The equality of the square with the sum in \eqref{eq:square-sum-density} is exact, so there is no overlap multiplicity.  Insert \eqref{eq:high-level-energy} and \eqref{eq:cocore-localization}; all remaining terms have logarithm $o_{\Pp}(N)$ and may be absorbed into one random $r_N=o_{\Pp}(N)$.  This proves \eqref{eq:network-realization}.
\end{proof}

\begin{proposition}[Inputs to the localization theorem and the global trial function]\label{prop:network-inputs}
For every $b$ in \eqref{eq:b-range}, the following inputs are unconditional.  Items \textup{(i)}--\textup{(viii)} and \textup{(x)} hold with probability tending to one, while item \textup{(ix)} holds in the probability-limit sense of \eqref{eq:K-pressure}:
\begin{enumerate}[label=\textup{(\roman*)},leftmargin=*]
 \item below $-Nb$, all critical points have index zero or one;
 \item the stable eigenvalues of every index-one point below $-Nb$ are bounded below by $\kappa_{\xi,b}$;
 \item the Hamiltonian and its first three derivatives satisfy \eqref{eq:derivative-event};
 \item no index-one point lies below $-N(E_1(\xi)+\delta)$;
 \item a finite separating co-core cut exists in the interval $I_{\xi,b,\delta}$;
 \item smooth densities charging every low-energy crossing exist;
 \item all paths reaching $\{H_N\ge-Nb\}$ are charged at energy cost $\exp\{\beta Nb+o(N)\}$;
 \item arbitrary overlap among the low-energy profiles produces no multiplicity loss in the global trial energy;
 \item the total Eyring--Kramers saddle conductance satisfies \eqref{eq:K-pressure};
 \item the two boundary sets in \eqref{eq:network-boundary} have Gibbs masses $1/2+o_{\Pp}(1)$.
\end{enumerate}
For $b\in(b_*(\xi),E_1(\xi))$, the additional quantitative comparison is the content of \Cref{thm:cocore-localization}, proved in Appendix~\ref{app:near-threshold-cocore}.
\end{proposition}

\begin{proof}
Items (i) and (iv) follow from \Cref{prop:layering}; item (ii) is
\Cref{prop:stable-gap-layer}; item (iii) is \Cref{prop:derivative-scale}; items (v) and (vi) are
\Cref{lem:topological-cocore-cut,lem:cocore-profiles}; item (vii) is
\Cref{lem:high-level-density}; item (viii) is the square-sum and modulus argument in the proof of \Cref{thm:network-realization}; item (ix) is \Cref{cor:K-pressure}.  For item (x), use \eqref{eq:A-B-masses} and \Cref{prop:deep-mass}, since $a<F_\xi(\beta)/\beta$.
\end{proof}

\begin{remark}[Why the remaining estimate is genuinely additional]\label{rem:soft-unstable-mode}
The stable-gap estimate controls the positive eigenvalues of an index-one Hessian, but it does not control the magnitude
$\alpha_z=|\lambda_-(Q_z)|$ of its unique negative eigenvalue.  This distinction matters at fixed $\beta$.  Consider the deterministic local model
\begin{equation}\label{eq:soft-mode-model}
 U_N(s,y)\defeq h_N-\frac{e^{-N}}2s^2-\frac{s^4}{4N}
          +\frac\kappa2\|y\|^2.
\end{equation}
At the origin its Hessian is
$\operatorname{diag}(-e^{-N},\kappa,\ldots,\kappa)$.  On a ball of radius $O(\sqrt N)$ it satisfies the same derivative scales as \eqref{eq:derivative-event}:
$|\grad U_N|=O(\sqrt N)$,
$\|\nabla^2U_N\|_{\op}=O(1)$, and
$\|\nabla^3U_N\|_{\op}=O(N^{-1/2})$.
Nevertheless, the exact one-dimensional resistance factor in the unstable direction is
\[
 J_N=\int_{\R}
 \exp\left\{-\beta\left(\frac{e^{-N}}2s^2+\frac{s^4}{4N}\right)\right\}\dd s
 \asymp_\beta N^{1/4},
\]
where the estimate follows from the substitution $s=N^{1/4}t$.  The corresponding finite-width conductance factor is $J_N^{-1}\asymp_\beta N^{-1/4}$, whereas the quadratic Eyring--Kramers conductance factor is proportional to
$\sqrt{e^{-N}}=e^{-N/2}$.  The ratio between the finite-width conductance and its quadratic prediction is therefore of order
$e^{N/2}N^{-1/4}$.  For sufficiently large fixed $\beta$, this cannot be absorbed by
$\exp\{C_{\xi,b}N/\sqrt\beta\}$ with $C_{\xi,b}$ independent of $\beta$.

This example is not asserted to occur in the mixed spherical landscape.
It shows only that the derivative scales and the stable spectral gap
do not logically imply \eqref{eq:cocore-localization}.
\Cref{app:near-threshold-cocore} separates the saddles according to
the size of $\alpha_z$: the nonsoft contribution is controlled
deterministically, while the aggregate treatment of exponentially soft
modes is carried out in \Cref{prop:soft-marked-KR}.
\end{remark}

Recall the aggregate conductance pressure $\Lambda^{\mathrm{EK}}_{\xi,\beta}(b)$ defined in \eqref{eq:Lambda-EK}.

\begin{theorem}[Complexity-weighted spherical-spin Eyring--Kramers bound]\label{thm:refined-EK}
Choose $b\in(b_*(\xi),E_1(\xi))$, with $b_*(\xi)$ as in \Cref{thm:cocore-localization}.  There is a threshold $\beta_0(\xi,b)<\infty$ such that, for every fixed $\beta\ge\beta_0(\xi,b)$,
\begin{equation}\label{eq:refined-EK}
 \pLiminf_{N\to\infty}
 \left(-\frac1N\log\gamma_{N,\beta}\right)
 \ge
 F_\xi(\beta)
 -\max\left\{
  \beta b,
  \Lambda^{\mathrm{EK}}_{\xi,\beta}(b)
  +\frac{C_{\xi,b}}{\sqrt\beta}
 \right\}.
\end{equation}
The threshold $\beta_0(\xi,b)$ is chosen so that, by \eqref{eq:F-over-beta}, $F_\xi(\beta)>\beta E_1(\xi)$ and $F_\xi(\beta)/\beta-E_1(\xi)>(E_0(\xi)-E_1(\xi))/4$ hold, and so that $\beta$ exceeds the localization thresholds of \Cref{thm:cocore-localization,thm:network-realization}.  The weaker hypothesis $F_\xi(\beta)>\beta E_1(\xi)$ alone would not suffice: the proof requires a level $a$ with $a-E_1(\xi)\ge(E_0(\xi)-E_1(\xi))/4$ below $F_\xi(\beta)/\beta$, and it invokes the near-threshold theorems, both of which are available only for sufficiently large fixed $\beta$.  Every term on the right-hand side is defined by the mixture $\xi$ alone.  There is no selected-well mass, target-mass penalty, selected-saddle energy, selected determinant, or determinant-transfer premise.
\end{theorem}

\begin{proof}
Since $\beta\ge\beta_0(\xi,b)$, the definition of the threshold permits a choice of $a\in(E_1(\xi),F_\xi(\beta)/\beta)$ with $a-E_1(\xi)\ge(E_0(\xi)-E_1(\xi))/4$, and \Cref{thm:cocore-localization,thm:network-realization} apply at this $\beta$.  By \eqref{eq:A-B-masses} and \Cref{prop:deep-mass},
\begin{equation}\label{eq:balanced-masses-refined}
 \pi_{N,\beta}(\cA_N(a))
 =\pi_{N,\beta}(\cB_N(a))
 =\frac12+o_{\Pp}(1).
\end{equation}
Use the trial function supplied by \Cref{thm:network-realization}.  The double-integral representation of the variance gives
\begin{equation}\label{eq:variance-network}
 \Var_{\pi_{N,\beta}}(f_N)
 \ge
 \pi_{N,\beta}(\cA_N(a))
 \pi_{N,\beta}(\cB_N(a))
 =\frac14+o_{\Pp}(1).
\end{equation}
Since the numerator in \eqref{eq:spectral-gap} is the unnormalized energy in \eqref{eq:network-realization} divided by $Z_{N,\beta}$,
\begin{equation}\label{eq:gap-network-pre}
 -\frac1N\log\gamma_{N,\beta}
 \ge F_{N,\xi}(\beta)
 -\frac1N\log
 \left(\int|\grad f_N|^2e^{-\beta H_N}\dd\nu_N\right)
 -o_{\Pp}(1).
\end{equation}

For fixed $\delta>0$, \Cref{cor:K-pressure} and \eqref{eq:network-realization} imply the following bound; the deterministic term $e^{N\Lambda^{\mathrm{EK}}_{\xi,\beta}(b)}$ of \eqref{eq:network-realization} is absorbed into the second entry of the maximum, because
\[
 \Lambda^{\mathrm{EK}}_{\xi,\beta}(b)
 \le-\frac12\log(\beta e)
 +\sup_{u\in[-E_1(\xi)-\delta,-b]}
 \bigl\{-\beta u+\Xi^{\mathrm{EK}}_{1,\xi}(u)\bigr\},
\]
the supremum being over the larger interval:
\begin{align}
 &\pLimsup_{N\to\infty}\frac1N\log
 \left(\int|\grad f_N|^2e^{-\beta H_N}\dd\nu_N\right)
 \notag\\
 &\qquad\le
 \max\left\{
 \begin{array}{l}
  \beta b,\\[0.2em]
  \displaystyle
  \frac{C_{\xi,b}}{\sqrt\beta}-\frac12\log(\beta e)
  +\sup_{u\in[-E_1(\xi)-\delta,-b]}
   \{-\beta u+\Xi^{\mathrm{EK}}_{1,\xi}(u)\}
 \end{array}
 \right\}.
 \label{eq:energy-pressure-delta}
\end{align}
Let $\delta\downarrow0$ and use continuity of the GOE and complexity rate functions.  Combining \eqref{eq:free-energy-limit}, \eqref{eq:gap-network-pre}, and \eqref{eq:energy-pressure-delta} proves \eqref{eq:refined-EK}.
\end{proof}

Theorem~\ref{thm:ek} of the main text is the statement
\eqref{eq:refined-EK}: the fixed-$b$ display of the main text is
\eqref{eq:refined-EK} verbatim, and the optimized form there follows
by taking, at the given fixed $\beta$, the infimum of the right-hand
side of \eqref{eq:refined-EK} over the admissible set
$B(\beta)=\{b\in(b_*(\xi),E_1(\xi)):\beta_0(\xi,b)\le\beta\}$---each such
$b$ yields a valid bound at that $\beta$, so the infimum does too;
the temperature is never an optimization variable.  The aggregate
conductance pressure $\Lambda^{\mathrm{EK}}_{\xi,\beta}(b)$ appearing
there is \eqref{eq:Lambda-EK}, the statistic
$\Xi^{\mathrm{EK}}_{1,\xi}=\Theta_{1,\xi}-\tfrac12D_\xi$ is \eqref{eq:Xi-EK}
with the half-determinant rate $D_\xi$ of \eqref{eq:Dp}, the phrase
``$b$ sufficiently close to $E_1(\xi)$'' means $b\in(b_*(\xi),E_1(\xi))$
with $b_*(\xi)$ as in \eqref{eq:b-star-range} and \eqref{eq:D-bstar},
and ``all sufficiently large fixed $\beta$'' means
$\beta\ge\beta_0(\xi,b)$.  This completes the proof of
Theorem~\ref{thm:ek}.

\begin{corollary}[Low-temperature form of the refined exponent]\label{cor:refined-low-temp}
Fix one $b\in(b_*(\xi),E_1(\xi))$, and let $L_{\Xi}(b)$ be a
Lipschitz constant for $\Xi^{\mathrm{EK}}_{1,\xi}$ on
$[-E_1(\xi),-b]$ (\Cref{lem:B-Lambda-regularity}).  For all
$\beta\ge\max\bigl\{\beta_0(\xi,b),\,L_{\Xi}(b),\,
\beta_b\bigr\}$, where $\beta_b$ is the finite threshold defined in
the proof below,
\begin{align}
 \pLiminf_{N\to\infty}
 \left(-\frac1N\log\gamma_{N,\beta}\right)
 \ge{}&F_\xi(\beta)-\beta E_1(\xi)
 -\Xi^{\mathrm{EK}}_{1,\xi}(-E_1(\xi))
 +\frac12\log(\beta e)
 \notag\\
 &-\frac{C_{\xi,b}}{\sqrt\beta},
 \label{eq:refined-low-temp}
\end{align}
where, by \eqref{eq:Xi-vs-Theta} and
$\Theta_{1,\xi}(-E_1(\xi))=0$,
\begin{equation}\label{eq:Xi-endpoint-negative}
 -\Xi^{\mathrm{EK}}_{1,\xi}(-E_1(\xi))
 \ \ge\
 \frac12D_\xi\bigl(-E_1(\xi),\omega^{\mathrm{EK}}_{-E_1(\xi)}\bigr)
 \ \ge\
 \frac12\Bigl(\frac12\log\xi''(1)-\frac12\Bigr)
 \ >\ 0;
\end{equation}
in the pure case
$-\Xi^{\mathrm{EK}}_{1,\xi}(-E_1(\xi))=\tfrac12D_\xi(-E_1(\xi))$.
In particular,
\begin{equation}\label{eq:refined-arrhenius}
 \liminf_{\beta\to\infty}
 \frac1\beta
 \pLiminf_{N\to\infty}
 \left(-\frac1N\log\gamma_{N,\beta}\right)
 \ge E_0(\xi)-E_1(\xi).
\end{equation}
\end{corollary}

\begin{proof}
By \Cref{lem:B-Lambda-regularity}, $\Xi^{\mathrm{EK}}_{1,\xi}$ is
Lipschitz on the compact interval $[-E_1(\xi),-b]$ with constant
$L_\Xi(b)$.  For $\beta>L_\Xi(b)$, the map
$u\mapsto-\beta u+\Xi^{\mathrm{EK}}_{1,\xi}(u)$ is strictly
decreasing on the interval, so its supremum is attained
\emph{exactly} at the left endpoint $u=-E_1(\xi)$:
\begin{align}
 \sup_{u\in[-E_1(\xi),-b]}
 \{-\beta u+\Xi^{\mathrm{EK}}_{1,\xi}(u)\}
 =\beta E_1(\xi)+\Xi^{\mathrm{EK}}_{1,\xi}(-E_1(\xi)).
 \label{eq:endpoint-pressure}
\end{align}
No $o_\beta(1)$ error is incurred.  Since $b<E_1(\xi)$, the
saddle-pressure term in \eqref{eq:refined-EK} dominates the
high-level extension term $\beta b$ once
$\beta(E_1(\xi)-b)\ge\frac12\log(\beta e)+|\Xi^{\mathrm{EK}}_{1,\xi}(-E_1(\xi))|
+C_{\xi,b}\beta^{-1/2}$, which holds for all $\beta$ beyond a finite
threshold $\beta_b$.  Substitute \eqref{eq:endpoint-pressure} into
\eqref{eq:refined-EK} to obtain \eqref{eq:refined-low-temp}; the
sign bound \eqref{eq:Xi-endpoint-negative} is
\eqref{eq:Xi-vs-Theta} at $u=-E_1(\xi)$ together with
$\Theta_{1,\xi}(-E_1(\xi))=0$ from \eqref{eq:Ek}.  For
\eqref{eq:refined-arrhenius}, divide by $\beta$ and use
\eqref{eq:F-over-beta}.
\end{proof}

\begin{remark}[The constant in \eqref{eq:Xi-endpoint-negative} is not $\tfrac14\log\xi''(1)+\tfrac14$]\label{rem:Xi-constant-example}
A previous version of this paper asserted
$-\Xi^{\mathrm{EK}}_{1,\xi}(-E_1(\xi))\ge\tfrac14\log\xi''(1)+\tfrac14$,
using $\Omega(-2)=\tfrac12$ as if every shift contributing to the
half-determinant rate satisfied $t_\xi\le-2$.  That is false for
genuine mixtures.  For $\xi(q)=0.9q^4+0.1q^{12}$ one has
$\xi'(1)=4.8$, $\xi''(1)=24$, $v_\xi^2=5.76$, and a direct
evaluation of the definitions \eqref{eq:Theta-mixed},
\eqref{eq:Ek}, and \eqref{eq:Xi-EK} (bisection for $E_1$,
one-dimensional maximization over $\omega$, high-precision
arithmetic) gives
\begin{center}
\begin{tabular}{@{}lr@{}}
\toprule
$E_1(\xi)$ & $1.8693502$\\
$E_2(\xi)$ & $1.8686770$\\
$-\Xi^{\mathrm{EK}}_{1,\xi}(-E_1(\xi))$ & $1.0409270$\\
$\tfrac12D_\xi(-E_1(\xi),\omega^{\mathrm{crt}})$ & $1.0513485$\\
$\tfrac12D_\xi(-E_1(\xi),\omega^{\mathrm{EK}})$ & $1.0190907$\\
$\tfrac14\log\xi''(1)+\tfrac14$ (previous claim) & $1.0445135$\\
$\tfrac14\log\xi''(1)-\tfrac14$ (certified) & $0.5445135$\\
\bottomrule
\end{tabular}
\end{center}
The counting maximizer sits at $t_\xi\approx-2.0148$, just outside
the semicircle support, whereas the half-determinant maximizer sits
at $t_\xi\approx-1.9485$, inside it; the previous claim fails by
$0.0036$, while the two-sided bound \eqref{eq:Xi-vs-Theta} holds
with room to spare.  Positivity of the determinant correction, and
hence the low-temperature improvement of \Cref{cor:refined-low-temp},
survives with the weaker certified constant.  We thank an anonymous
reader for the counterexample.  (The same computation with
$\xi(q)=q^4$ returns $E_1=1.7833003$ and
$-\Xi^{\mathrm{EK}}_{1,\xi}(-E_1)=\tfrac12D_\xi(-E_1)=0.8964335$,
the pure-case values quoted in \Cref{rem:improvement}.)
\end{remark}

\begin{remark}[Relation to the unconditional bound]\label{rem:refinement-relation}
The determinant correction in \eqref{eq:refined-low-temp} is not
obtained by plugging the fixed-location determinant into the
Eyring--Kramers formula for one critical point.  It comes from the
half-determinant statistic in the aggregate marked Kac--Rice sum.  The
unconditional theorem \eqref{eq:fixed-temperature-bound} is
independent of \Cref{thm:cocore-localization}.  The determinant-level
refinement rests on the annealed marked Kac--Rice estimate
\Cref{prop:soft-marked-KR} proved in \Cref{app:near-threshold-cocore}.
\end{remark}

\section{Numerical experiments}\label{sec:experiments}

All simulations use \texttt{numpy}; the scripts are publicly
available (see the code and data availability statement).
All experiments probe the \emph{pure} case $\xi(x)=x^p$, covered by
\Cref{cor:pure}, where every rate function has a closed form; a
numerical exploration of genuinely mixed $\xi$, for which the rate
functions require the additional one-dimensional shift optimization
of \eqref{eq:Lambda-rk}, is left to future work.  At the even value
$p=4$, numerical root finding on the pure-case
complexity functions
$\Theta_{k,\xi}(-z)=\tfrac12\log(p-1)-\tfrac{(p-2)z^2}{4(p-1)}
-(k+1)I_{1,\xi}(z)$ of \citet{auffinger2013random} gives
\begin{equation}\label{eq:thresholds}
  E_\infty(4)=1.73205,\quad
  E_2(4)=1.77648,\quad
  E_1(4)=1.78330,\quad
  E_0(4)=1.79409.
\end{equation}
The dynamics experiments are run at this value of $p$; only the
dynamical-arrest sweep of \Cref{sec:arrest-exp} uses $p=3$, where
$T_d=1/E_\infty$ exactly makes the comparison with the
\citet{cugliandolo1993analytical} transition parameter-free (it probes
the dynamical transition, not the theorems).  The geometric picture
behind the theorems is summarized in \Cref{fig:gate-morse-geometry}
(\Cref{sec:gate}); \Cref{fig:landscape} renders an actual small
instance of the landscape.

\begin{figure}[t]
\centering
\includegraphics[width=0.82\linewidth]{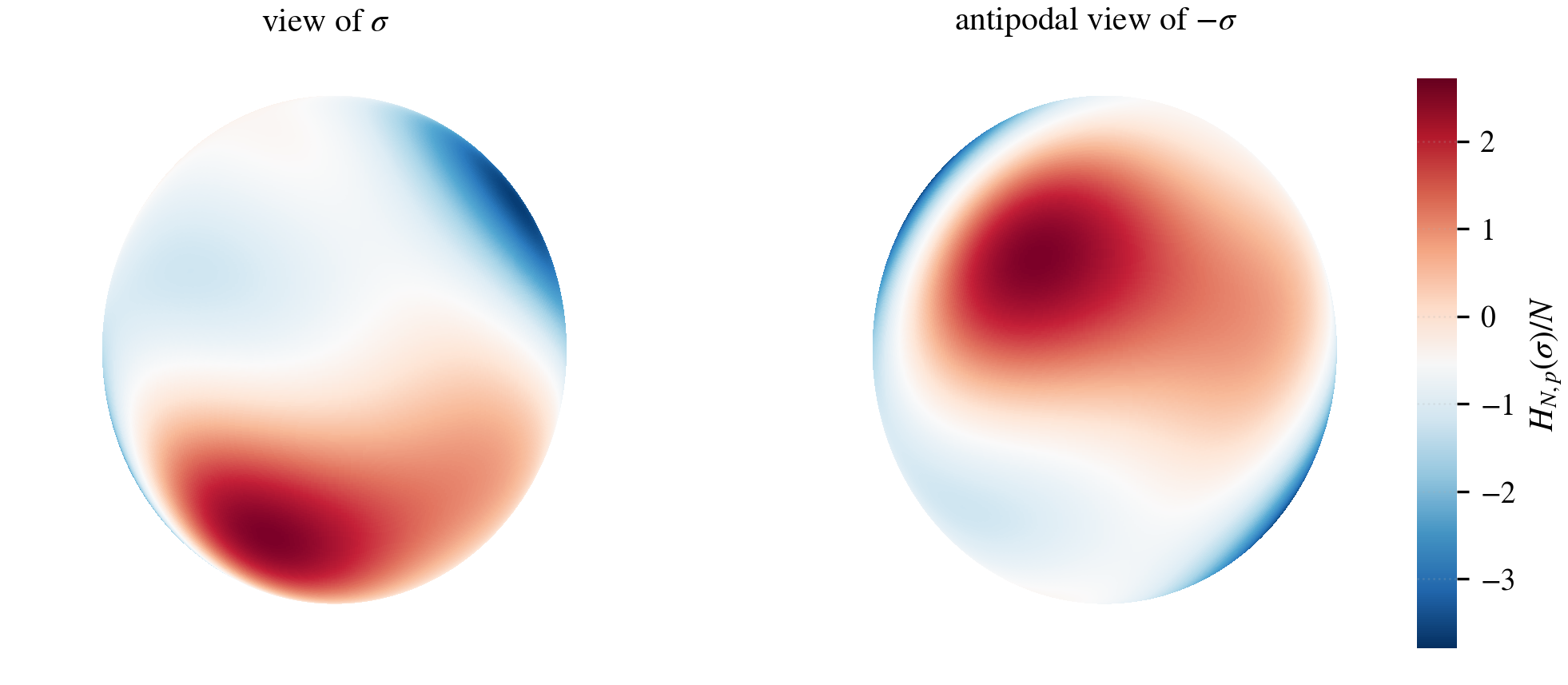}
\caption{An actual pure $p$-spin landscape ($p=4$, $N=3$); blue
regions are wells, red are barriers.  Because $p$ is even,
$H_N(-\sigma)=H_N(\sigma)$: every well has an exact antipodal
twin at the same depth (right panel), the mechanism behind the mass
balance in \Cref{thm:fixed}; white corridors between basins are
crossed at index-one saddles, whose lowest possible level
\Cref{thm:arrhenius} constrains from below.}
\label{fig:landscape}
\end{figure}

\subsection{Langevin dynamics of the $p$-spin model at even $p=4$}
\label{sec:pspin-exp}

We simulate the Euler discretization of spherical Langevin dynamics
$\mathrm d\sigma=-\beta\,\grad_{\mathrm{sph}}H_N\,\mathrm dt
+\sqrt2\,\mathrm dB_t$ for $p=4$, $N=60$, with i.i.d.\ couplings, step
size $\mathrm dt=5\times10^{-3}$, and three disorder samples per
temperature.

\begin{figure}[t]
\centering
\includegraphics[width=0.44\linewidth]{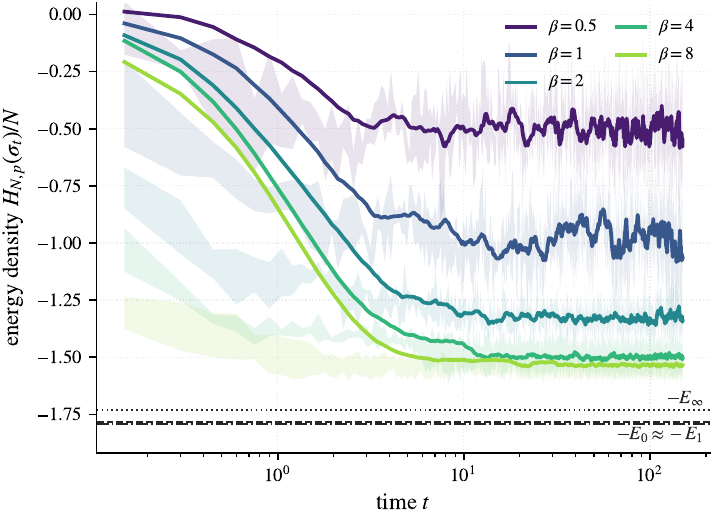}\hfill
\includegraphics[width=0.44\linewidth]{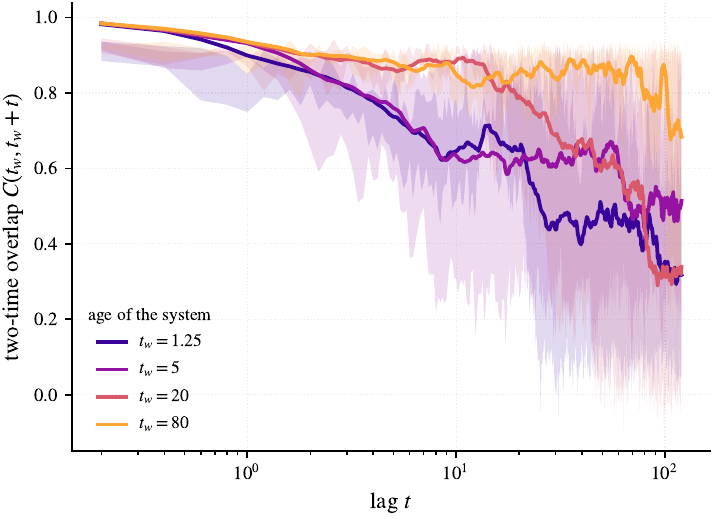}
\caption{\textbf{Left:} energy density along Langevin trajectories
from a random start ($p=4$, $N=60$; line = smoothed median over three
disorder samples, band = min--max).  At low temperature the dynamics
relaxes toward a plateau on the threshold side of the landscape,
above $-E_\infty$ (dotted) and far from the nearly
coincident $-E_1$, $-E_0$ (dashed/dash-dotted); the plateau offset
from $-E_\infty$ is a finite-$N$, finite-time, finite-step effect,
larger than at $p=3$ because barriers grow with $p$.
\textbf{Right:} two-time overlap
$C(t_w,t_w+t)=\langle\sigma_{t_w},\sigma_{t_w+t}\rangle/N$ at
$\beta=2.5$, moderately below the dynamical transition
($\beta_d=1.84$), for increasing $t_w$ (mean over three thermal
histories, band = range): older systems decorrelate more slowly
(aging); deeper in the frozen phase $C$ barely decays over accessible
windows.}
\label{fig:pspin}
\end{figure}

\paragraph{Threshold-side relaxation and aging (\Cref{fig:pspin}).}
For the low-temperature runs the energy density relaxes toward a
plateau on the threshold side of the landscape, above $-E_\infty$,
and stalls there on simulation
time scales, in agreement with the Cugliandolo--Kurchan picture
\citep{cugliandolo1993analytical}: the Gibbs-relevant wells below
$-NE_1$ are not reached on the simulated time scales, consistent with
(though not implied by; see \Cref{sec:rsb-aging}) the spectral-gap
bounds of \Cref{sec:theory}.  The two-time
overlap exhibits aging: the relaxation time grows with the age $t_w$ of
the system.  These are qualitative signatures of
the glassy phase; they do not measure the spectral gap itself.

\begin{figure}[t]
\centering
\includegraphics[width=0.52\linewidth]{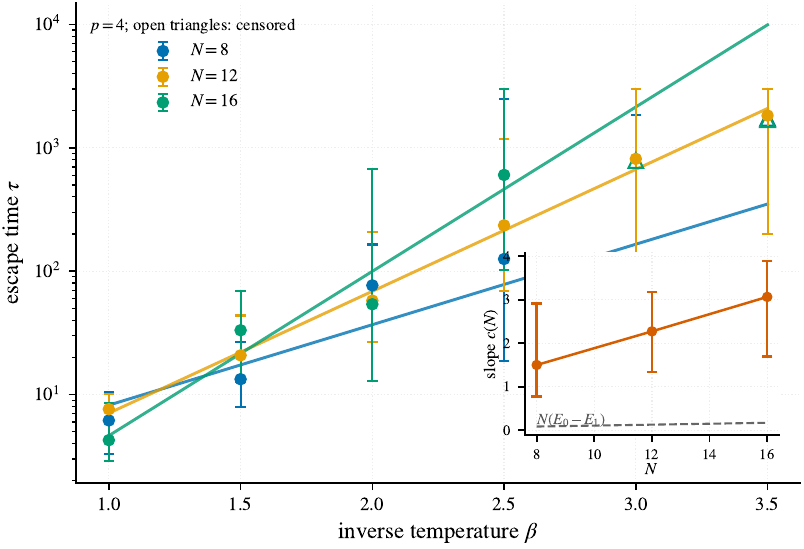}
\caption{Escape from the deepest of ten gradient-descent minima at
$p=4$, $N=8,12,16$ (fifteen disorder samples per point; escape =
overlap with the initial minimum first drops below $0.3$).  Markers:
median with IQR bars; open triangles: censored medians ($>40\%$ of
runs did not escape; lower bounds, excluded from fits).  Lines:
Arrhenius fits $\tau\propto e^{c(N)\beta}$.  \textbf{Inset:} slope
$c(N)$ with bootstrap $95\%$ CIs and the asymptotic constant
$N(E_0-E_1)$ of \Cref{thm:arrhenius} (dashed); the measured
barriers grow with $N$ but exceed $N(E_0-E_1)$ by a large factor at
these sizes---see text.}
\label{fig:arrhenius}
\end{figure}

\paragraph{Activated escape and its $N$-dependence
(\Cref{fig:arrhenius}).}
At small $N$, where activated events are observable, the median escape
time from a deep minimum grows exponentially in $\beta$, and the
fitted effective barrier grows with $N$: $c(8)=1.5$ (bootstrap $95\%$
CI $[0.8,2.9]$), $c(12)=2.3$ ($[1.3,3.2]$), $c(16)=3.1$
($[1.7,3.9]$)---the two qualitative signatures of activated dynamics
with extensive barriers.  This experiment does \emph{not}
quantitatively validate \Cref{thm:arrhenius}: the theorem's
constant $E_0-E_1$ is the asymptotic separation between the
ground-state band and the lowest index-one level, and at $N\le16$ that
separation ($N(E_0-E_1)\approx0.09$--$0.17$) is far smaller than the
$O(1)$ sample-to-sample fluctuations of individual well depths, so the
measured slope reflects the local barrier of the particular
well---exactly as the inset shows.  Single-well escape is also a
different observable from the two-well relaxation the gap measures: in
the same runs, first passage to the antipodal hemisphere (overlap
$<-0.5$) was censored at the horizon for most low-temperature runs,
consistent with---but not a measurement of---the much longer gap time
scale.  These three small sizes support an illustrative claim of
extensive barrier growth, not a quantitative validation of the
constant.

\subsection{Measuring the aggregate saddle weight directly}
\label{sec:saddle-exp}

The distinctive object of \Cref{thm:ek} is the weighted index-one
saddle sum
\begin{equation}\label{eq:W-main}
  \cW_{N,\beta}(I)
  =\sum_{\substack{z:\ \grad_{\mathrm{sph}} H_N(z)=0,\ \ind z=1\\
                 H_N(z)/N\in I}}
  e^{-\beta H_N(z)}
  \frac{\alpha_z}{\sqrt{|\det Q_z|}},
\end{equation}
where $Q_z$ is the Riemannian Hessian and $\alpha_z$ the magnitude of
its unique negative eigenvalue, and the theory bounds the annealed
growth rate by $\sup_{u\in I}\{-\beta
u+\Xi^{\mathrm{EK}}_{1,\xi}(u)\}$ with
$\Xi^{\mathrm{EK}}_{1,\xi}=\Theta_{1,\xi}-\tfrac12D_p$
(\Cref{thm:weighted-KR}).  Rather than test only qualitative
phenomenology, we measure \eqref{eq:W-main} directly at small sizes:
for $p=4$, $N\in\{8,10,12\}$ and eight disorder samples per size, we
enumerate critical points by damped projected Newton iteration from
$4{,}000$ starts per sample (random, gradient-descent- and
ascent-preconditioned), deduplicate up to the exact antipodal
symmetry, classify each point by the inertia of its Riemannian
Hessian, and record energy, $\alpha_z$, and $\log|\det Q_z|$ for
every index-one saddle.

\begin{figure}[t]
\centering
\includegraphics[width=0.92\linewidth]{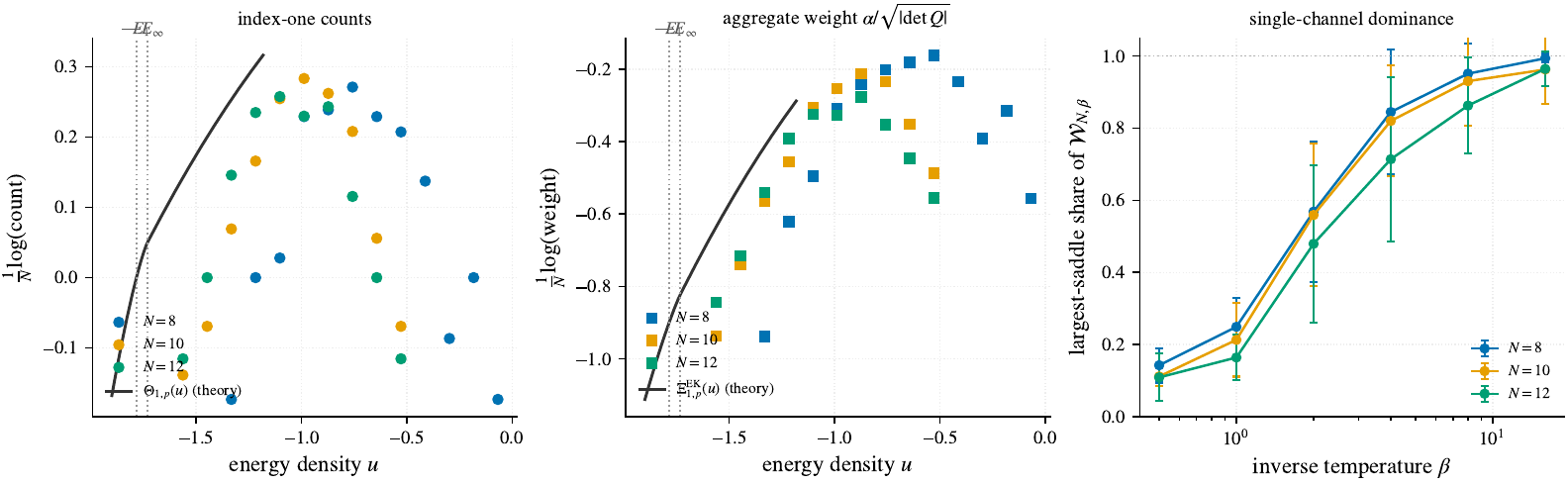}
\caption{Direct measurement of the aggregate index-one weight
($p=4$; $N=8,10,12$; eight disorder samples per size; energy bins of
width $0.115$; disorder-averaged).  \textbf{Left:} binned index-one
counts, $\frac1N\log(\text{count})$, against the complexity
$\Theta_{1,\xi}(u)$ (solid).  \textbf{Middle:} binned
determinant-weighted sums
$\frac1N\log\sum\alpha_z/\sqrt{|\det Q_z|}$ against the conductance
complexity $\Xi^{\mathrm{EK}}_{1,\xi}(u)$ (solid).  \textbf{Right:}
share of the largest single saddle in the full weighted sum
$\cW_{N,\beta}$ versus $\beta$ (mean $\pm$ s.d.\ over disorder
samples).  Vertical dotted lines mark $-E_1$ and $-E_\infty$.}
\label{fig:saddle-sum}
\end{figure}

\paragraph{Results (\Cref{fig:saddle-sum}).}
Three observations.  \emph{First}, the energy-resolved counts and
weighted sums follow the shape of the theoretical rate functions, but
at these sizes the asymptotic band structure is strongly smeared: the
deepest minima fall below $-NE_0$ (down to $u=-1.87$ at $N=8$), and
no index-one saddle is found near $-NE_1$---the maximal complexity
$\Theta_{1,4}\le0.049$ makes the expected near-threshold population
$O(1)$ at these $N$, so the observed index-one saddles concentrate at
higher energies.  The comparison is therefore a consistency check of
the rate functions' shape and trend, not a verification of their
values.  \emph{Second}, the determinant
weighting matters: the weighted profile is visibly displaced from the
raw counts, in the direction predicted by the half-determinant
statistic $-\tfrac12D_p$.  \emph{Third}, the dominance structure is
informative: at moderate $\beta$ the largest single saddle carries
only part of $\cW_{N,\beta}$---the aggregate is genuinely
shared among channels---while as $\beta\to\infty$ the lowest saddle
takes over, the finite-$N$ image of the crossover between the two
branches of \Cref{thm:ek}; the share of the largest channel at fixed
moderate $\beta$ decreases with $N$, consistent with a channel
entropy that grows with dimension.  Discovery diagnostics (new
saddles found in the last quarter of Newton starts) and all
tolerances are reported in \Cref{app:sim}.

\subsection{Dynamical arrest at $p=3$}\label{sec:arrest-exp}

\begin{figure}[t]
\centering
\includegraphics[width=0.85\linewidth]{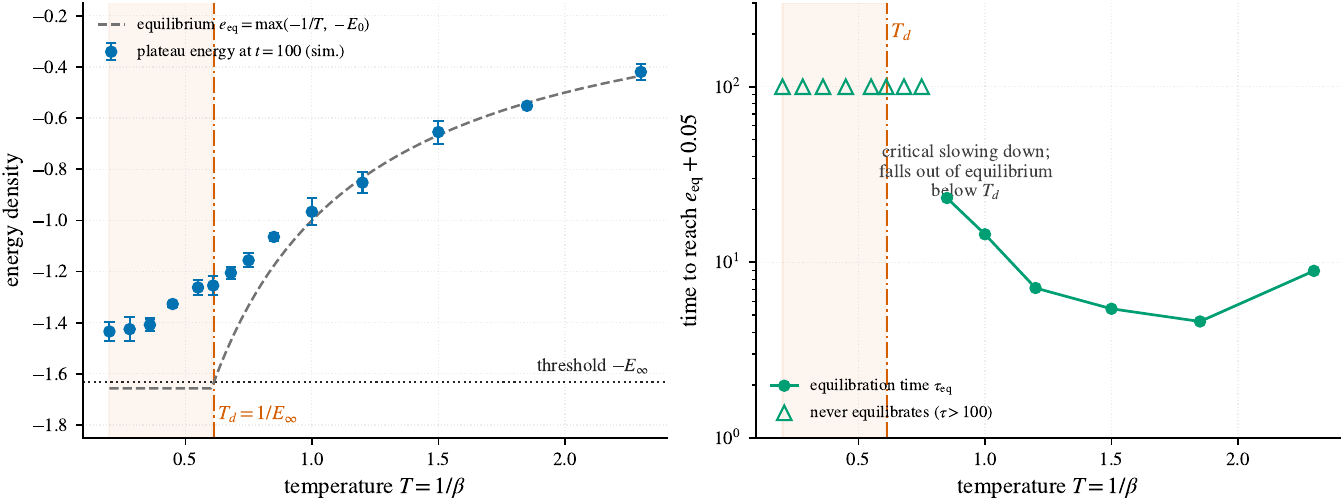}
\caption{Dynamical arrest of $p$-spin Langevin dynamics ($p=3$,
$N=64$, three disorder samples).  \textbf{Left:} plateau energy after
a quench versus temperature: above $T_d$ the data sit on the
equilibrium curve $e_{\mathrm{eq}}(T)$; below $T_d$ the dynamics
arrests near $-E_\infty$ (residual offset: finite-$N$/step-size
effect).  For $p=3$, $T_d=1/E_\infty$ exactly, so the equilibrium
curve reaches the threshold precisely at the transition.
\textbf{Right:} time to reach the equilibrium energy: it grows
steeply as $T\downarrow T_d$ and is unreachable on this horizon
slightly above $T_d$ and for every $T<T_d$ (open triangles are lower
bounds).}
\label{fig:arrest}
\end{figure}

The spectral-gap bounds of \Cref{sec:theory} concern the
low-temperature phase; \Cref{fig:arrest} locates the phase boundary
itself in simulation, at $p=3$, where the comparison with
\citet{cugliandolo1993analytical} is parameter-free (this probes the
dynamical transition, not the theorems, which require even $p\ge4$).
Sweeping the temperature across the dynamical transition
$T_d^2=p(p-2)^{p-2}/\bigl(2(p-1)^{p-1}\bigr)$
\citep{cugliandolo1993analytical}, the measured plateau energy
follows the equilibrium law $e_{\mathrm{eq}}=-1/T$ at high
temperature and detaches at $T_d$, arresting near $-E_\infty$; the
time to reach equilibrium diverges at the same point.  This is the
dynamical counterpart of the positivity condition
$F_\xi(\beta)>\beta E_1(\xi)$ in \Cref{thm:fixed}: below the
transition, equilibrium lives in wells whose worst-case relaxation
time is exponential in $N$, and which the simulated dynamics does not
reach on accessible time scales.

\subsection{The formal branch comparison}\label{sec:transitions}

\begin{figure}[t]
\centering
\includegraphics[width=0.58\linewidth]{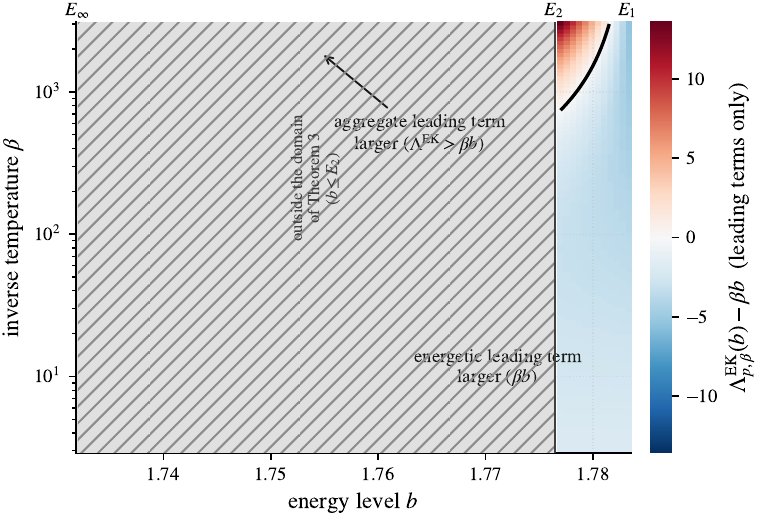}
\caption{\emph{Formal comparison of the two explicit leading terms}
of \Cref{thm:ek} at $p=4$: the sign of
$\Lambda^{\mathrm{EK}}_{\xi,\beta}(b)-\beta b$ in the $(b,\beta)$
plane.  This is \emph{not} the branch boundary of the theorem: the
proved bound compares $\beta b$ with
$\Lambda^{\mathrm{EK}}_{\xi,\beta}(b)+C_{\xi,b}/\sqrt\beta$, and
$C_{\xi,b}$ is not explicit, so the displayed curve compares leading
terms only.  Moreover the theorem holds only for
$b>b_*(\xi)$ with $b_*$ not numerically certified and forced into a
left neighborhood of $E_1(\xi)$ (\Cref{thm:ek}): the hatched region
$b\le E_2(4)=1.77648$ violates even the necessary condition and is
masked, and in the remaining strip the comparison is formal wherever
$b<b_*$.}
\label{fig:mechanism}
\end{figure}

\Cref{fig:mechanism} visualizes the temperature structure of the
\emph{bound} in \Cref{thm:ek} as a formal leading-term comparison:
where $\Lambda^{\mathrm{EK}}_{\xi,\beta}(b)>\beta b$ the aggregate
term controls the displayed maximum, and at very low temperature the
energetic term takes over.  The caption states the two caveats
precisely (unknown $C_{\xi,b}$, uncertified $b_*$); whether the true
escape rate exchanges mechanisms along a similar boundary is exactly
the content of the open matching upper bound (\Cref{rem:scope}).

\section{Discussion and limitations}\label{sec:discussion}

\paragraph{What the theory does and does not say.}  The bounds of
\Cref{sec:theory} are one-sided (\Cref{rem:scope}).  A
matching upper bound would certify the candidate gates as the true
bottleneck; it reduces to two open landscape hypotheses (quenched gate
attainment and local mixing within wells); until then, statements
about the realized escape mechanism remain hypotheses.  The experiments
probe the saddle sum at sizes far from the asymptotic regime: at
$N\le12$ the index thresholds are strongly smeared, and comparisons
with the theory are in shape and trend only.

\paragraph{The evenness restriction.}  Evenness of the mixture is
load-bearing: every theorem uses the exact antipodal symmetry
$H_N(-\sigma)=H_N(\sigma)$---for the mass balance between
the two deep-well unions, for the antisymmetric test function of
\Cref{thm:fixed}, and for the gate analysis between $\pm$ pairs.
For mixtures with odd components (including the pure $p=3$ model)
the deepest wells are not antipodal
images of each other, so the two-well construction does not apply as
stated; a version built on the two deepest wells would require
control of their individual Gibbs masses and separation, which we
do not have; extending to odd components is open.  The exclusion of
quadratic components ($p=2$) is of a different nature: a strong
quadratic part pushes the mixture out of the pure-like class toward
full RSB and topological trivialization, where assumption \textup{(A2)}
fails and the bounds are correctly vacuous.

\paragraph{Open problems.}  First, computable bounds on $b_*(\xi)$ and
$C_{\xi,b}$ would upgrade \Cref{rem:improvement} to a fully certified
finite-$\beta$ statement and \Cref{fig:mechanism} to a true phase
diagram of the bound.  Second, the weighted Kac--Rice estimate
behind \Cref{thm:ek} is an annealed \emph{upper} rate plus Markov's
inequality; a quenched convergence result for
$\frac1N\log\cW_{N,\beta}$---via a second-moment argument in the
spirit of \citet{subag2017complexity}, extended to the
determinant-weighted marks in the window $[-E_1(\xi)-\delta,-b]$ the
proof uses---would establish $\Lambda^{\mathrm{EK}}_{\xi,\beta}$ as
the actual exponential scale; our enumeration experiment is its
finite-$N$ template.  Third, this
paper covers even pure-like mixtures under the strict separation
\textup{(A2)}; verifying \textup{(A2)} for explicit families of
mixtures (beyond the pure case, where it is a theorem), and
understanding the dynamics in full-RSB landscapes---where marginal
connectivity replaces barriers and polynomial-time optimization is
possible \citep{subag2021following}---remain open.  Fourth, the
spectral-gap bounds say nothing about relaxation from a uniform
start, because the certified slow mode is odd under the antipodal
map while a symmetric initial law lives in the even sector
(\Cref{sec:rsb-aging}); proving that aging from random starts
persists to exponential times requires a lower bound on the descent
time of an even observable such as the energy, and would connect the
present barrier estimates with the exact asymptotic energies of the
limiting dynamics \citep{lang2026aging}.

\subsection*{Code and data availability}
All figures are generated by self-contained \texttt{numpy} scripts
with fixed random seeds; the theoretical thresholds in
equation~\eqref{eq:thresholds} are computed in the same scripts by
quadrature and root finding.  The scripts, together with the
per-sample data and diagnostics of the saddle enumeration, are
publicly available at
\url{https://github.com/mbadieik/spectral_gap_for_mixed_spherical_spin_glasses}.

\appendix

\section{GOE determinant powers and weighted Kac--Rice}\label{app:weighted-KR}

This appendix supplies the random-matrix details behind
\Cref{prop:complexity-upper,prop:fixed-location-det} and
\Cref{thm:weighted-KR}.  Put $n=N-1$.  By
\Cref{prop:shifted-GOE}, the conditional Hessian at energy $Nu$ is
\begin{equation}\label{eq:B-Q}
 Q_N(u)\defeq c_NG_n-s_N(u)I_n,
 \qquad
 s_N(u)\defeq\xi'(1)u+\omega_N,
 \qquad
 \omega_N\defeq\frac{v_\xi Z}{\sqrt N},
\end{equation}
with $Z$ standard Gaussian independent of $G_n$ and
$c_N=\sqrt{\tfrac{N-1}N\,\xi''(1)}$.  For a deterministic shift
value $\omega$, write
\begin{equation}\label{eq:B-Q-omega}
 Q_N(u,\omega)\defeq c_NG_n-\bigl(\xi'(1)u+\omega\bigr)I_n,
 \qquad
 t_N(u,\omega)\defeq\frac{\xi'(1)u+\omega}{c_N},
\end{equation}
so that $Q_N(u)=Q_N(u,\omega_N)$ and
$t_N(u,\omega)\to t_\xi(u,\omega)$ of \eqref{eq:t-shift}, uniformly
on compacts.  Every estimate below is proved by conditioning on
$\{\omega_N=\omega\}$, applying a fixed-shift GOE estimate uniformly
on a compact rectangle of $(u,\omega)$, and integrating over the
Gaussian shift by the following elementary Laplace bound; the
penalty $\mathfrak p_\xi$ of \eqref{eq:penalty} is exactly the
speed-$N$ rate of $\omega_N$.  Let
\begin{equation}\label{eq:B-empirical}
 L_n\defeq\frac1n\sum_{i=1}^n\delta_{\lambda_i(G_n)},
\end{equation}
so that
\begin{equation}\label{eq:B-logdet}
 \frac1n\log|\det Q_N(u,\omega)|
 =\log c_N+\int\log|x-t_N(u,\omega)|L_n(\dd x).
\end{equation}

\begin{lemma}[Scalar-tilt Laplace bound]\label{lem:B-tilt-laplace}
Let $\Phi_N:\R\to[0,\infty)$ be measurable functions and let
$\varphi:\R\to\R$ be continuous.  Assume
\begin{enumerate}[label=\textup{(\roman*)},leftmargin=*]
\item \textup{(local bound)} for every compact $W\subset\R$ and
$\eps>0$ there is $N_0$ with
$\Phi_N(\omega)\le e^{N(\varphi(\omega)+\eps)}$ for all
$\omega\in W$ and $N\ge N_0$;
\item \textup{(global growth)} there are $C_0<\infty$ and $N_1$
such that, for all $\omega\in\R$ and $N\ge N_1$,
\[
 \varphi(\omega)\le C_0\bigl(1+\log(1+|\omega|)\bigr)
 \qquad\text{and}\qquad
 \Phi_N(\omega)\le e^{NC_0(1+\log(1+|\omega|))}.
\]
\end{enumerate}
Then
\begin{equation}\label{eq:B-tilt-laplace}
 \limsup_{N\to\infty}\frac1N\log\E\bigl[\Phi_N(\omega_N)\bigr]
 \le\sup_{\omega\in\R}
 \bigl\{\varphi(\omega)-\mathfrak p_\xi(\omega)\bigr\}.
\end{equation}
If $v_\xi=0$ the statement is trivial, the right-hand side being
$\varphi(0)$.
\end{lemma}

\begin{proof}
Assume $v_\xi>0$ and abbreviate the right-hand side of
\eqref{eq:B-tilt-laplace} by $S$; note $S\ge\varphi(0)>-\infty$.
Since $\omega_N=v_\xi Z/\sqrt N$, for every interval
$[\omega,\omega+\theta]$ one has
$\Pp(\omega_N\in[\omega,\omega+\theta])
\le\exp\{-N\min_{[\omega,\omega+\theta]}\mathfrak p_\xi\}$.
Fix $\eps>0$ and $\Omega_0\ge1$ with
$\Omega_0^2/(4v_\xi^2)\ge C_0\bigl(1+\log(1+\Omega_0)\bigr)+|S|+1$;
the map $\omega\mapsto C_0(1+\log(1+|\omega|))-\omega^2/(4v_\xi^2)$
is decreasing on $[\Omega_0,\infty)$, so on
$\{|\omega_N|\ge\Omega_0\}$ the global growth bound (ii) on
$\Phi_N$ and the Gaussian tail give
\[
 \E\bigl[\Phi_N(\omega_N)\mathbf1_{\{|\omega_N|\ge\Omega_0\}}\bigr]
 \le2\sum_{j\ge0}
 e^{NC_0(1+\log(2+\Omega_0+j))}
 e^{-N(\Omega_0+j)^2/(2v_\xi^2)}
 \le e^{N(S-1)}
\]
for all large $N$.  On $[-\Omega_0,\Omega_0]$, choose a finite
partition into intervals of mesh $\theta$ so small that both
$\varphi$ and $\mathfrak p_\xi$ oscillate by at most $\eps$ on each
interval (uniform continuity); on each interval the local bound (i)
and the probability bound above give a contribution at most
$e^{N(\varphi(\omega_j)-\mathfrak p_\xi(\omega_j)+3\eps)}
\le e^{N(S+3\eps)}$.  Summing the finitely many intervals and
letting $\eps\downarrow0$ proves \eqref{eq:B-tilt-laplace}.
\end{proof}

\begin{remark}[The global growth hypothesis cannot be dropped]\label{rem:B-tilt-laplace-global}
Hypothesis (i) controls $\Phi_N$ only on fixed compact sets and says
nothing about mass escaping to infinity with $N$.  For $v_\xi>0$,
$\varphi\equiv0$, and
$\Phi_N(\omega)=1+e^{N^4}\mathbf1_{[N,N+1]}(\omega)$, hypothesis (i)
holds (on any fixed compact, $\Phi_N=1$ for large $N$), yet
$\frac1N\log\E[\Phi_N(\omega_N)]\ge N^3-\frac{(N+1)^2}{2v_\xi^2}-O(\log N/N)\to+\infty$, whereas the
right-hand side of \eqref{eq:B-tilt-laplace} is $0$.  A previous
version of this lemma omitted (ii); its proof used the bound
silently.  In every application below, (ii) is supplied either
trivially, because $\Phi_N$ is a probability, or by the
operator-norm moment bound \eqref{eq:B-det-moment-omega}: every
Hessian mark used is at most $\|Q_N(u,\omega)\|_{\op}^{rn+1}$ with
$r\le3$, and
$e^{C_0n}(1+|\omega|)^{rn+1}\le e^{N(C_0+r+1)(1+\log(1+|\omega|))}$
is of the required form.
\end{remark}

\begin{lemma}[Regularity of the tilted rates]\label{lem:B-Lambda-regularity}
For every $r\in[0,3]$ and $k\in\{1,2\}$, the supremum in
\eqref{eq:Lambda-rk} is attained, the maximizers lie in a compact
set that is uniform for $u$ in compacts, and
$\Lambda^{(r)}_{k,\xi}$---hence also $\Theta_{k,\xi}$,
$\Xi^{\mathrm{EK}}_{1,\xi}$, and $I_{1,\xi}$---is locally Lipschitz
on $\R$.
\end{lemma}

\begin{proof}
The potential $\Omega$ is differentiable with
$\Omega'(t)=t/2$ for $|t|\le2$ and
$\Omega'(t)=t/2-\operatorname{sgn}(t)\sqrt{t^2-4}/2$ for $|t|>2$, a
function bounded on compacts, and $J$ of \eqref{eq:Jgoe} is
differentiable with $J'(t)=-\tfrac12\sqrt{t^2-4}$ on $(-\infty,-2]$
and $J'=0$ on $(-2,\infty)$; both are therefore locally Lipschitz
in $t$, and $t_\xi(u,\omega)$ is affine in $(u,\omega)$.  If
$v_\xi=0$ the claim follows at once.  If $v_\xi>0$, the objective in
\eqref{eq:Lambda-rk} tends to $-\infty$ as $|\omega|\to\infty$,
uniformly for $u$ in a compact $K$ (the penalty is quadratic, $J\ge0$,
and $D_\xi$ grows logarithmically), so the maximizers lie in a
compact $W_K$; on $K\times W_K$ the objective is Lipschitz in $u$
uniformly in $\omega$, and the upper envelope of a uniformly
Lipschitz family is Lipschitz.  The infimum defining $I_{1,\xi}$ is
handled identically.
\end{proof}

\subsection{Fixed-location logarithmic determinant}

For $\eta>0$, let
\begin{equation}\label{eq:B-trunc-log}
 \ell_{\eta,t}(x)\defeq\log(|x-t|\vee\eta).
\end{equation}
Since $\omega_N\to0$ in probability,
$t_N(u_N,\omega_N)\to t_\xi(u,0)$ in probability whenever
$u_N\to u$.  For fixed $\eta$, the empirical semicircle law gives
\begin{equation}\label{eq:B-trunc-convergence}
 \int\ell_{\eta,t_N(u_N,\omega_N)}\dd L_n
 \longrightarrow
 \int\ell_{\eta,t_\xi(u,0)}(x)\rho_{\sclaw}(x)\dd x
\end{equation}
in probability.  To remove the truncation, let $\bar\rho_n$ be the normalized mean GOE eigenvalue density.  The Hermite-function representation and uniform Plancherel--Rotach estimates give, for each compact $K\subset\R$,
\begin{equation}\label{eq:B-density}
 \sup_n\sup_{x\in K}\bar\rho_n(x)<\infty;
\end{equation}
see \citet[Chapter~6]{mehta2004random}.  Consequently, conditionally
on any value of the independent shift,
\begin{align}
 \E\left[
 \int_{|x-t_N|<\eta}\log\frac\eta{|x-t_N|}L_n(\dd x)
 \right]
 \le C\int_{-\eta}^{\eta}\log\frac\eta{|v|}\dd v
 \le C'\eta.
 \label{eq:B-uniform-integrability}
\end{align}
Markov's inequality, followed by $\eta\downarrow0$, proves uniform integrability of the logarithm.  Substitution into \eqref{eq:B-logdet} yields \eqref{eq:fixed-location-det} with the limit $D_\xi(u)=D_\xi(u,0)$.

\subsection{Changing the determinant power inside the index-one Kac--Rice integral}

Define
\begin{equation}\label{eq:B-index-event}
 \cI_{N,1}(u)\defeq\{\ind Q_N(u)=1\},
\end{equation}
and let $\alpha_N(u)$ be the absolute value of the unique negative eigenvalue on this event.  The next lemma is the matrix input specific to the Eyring--Kramers mark.

\begin{lemma}[Half-determinant upper comparison]\label{lem:half-det}
Let $K\subset(-\infty,0)$ be compact, and define
\begin{align}
 M_N^{\mathrm{EK}}(u)
 &\defeq\E\left[
  \alpha_N(u)|\det Q_N(u)|^{1/2}
  \mathbf1_{\cI_{N,1}(u)}
 \right],\notag\\
 M_N^{\mathrm{crt}}(u)
 &\defeq\E\left[
  |\det Q_N(u)|\mathbf1_{\cI_{N,1}(u)}
 \right].
 \label{eq:B-two-moments}
\end{align}
Uniformly for $u\in K$,
\begin{equation}\label{eq:B-half-vs-full}
 \frac1N\log M_N^{\mathrm{EK}}(u)
 \le\Lambda^{(1/2)}_{1,\xi}(u)+o(1),
 \qquad
 \frac1N\log M_N^{\mathrm{crt}}(u)
 \le\Lambda^{(1)}_{1,\xi}(u)+o(1);
\end{equation}
by \eqref{eq:Xi-vs-Theta}, the half-determinant mark thus lowers the
matrix exponent below the counting exponent by at least
$\tfrac12(\tfrac12\log\xi''(1)-\tfrac12)$.  Only these upper bounds
are needed for the conductance-pressure theorem.
\end{lemma}

\begin{proof}
Condition on $\{\omega_N=\omega\}$, so that the matrix is the
deterministically shifted GOE $Q_N(u,\omega)$, and write the ordered
GOE eigenvalues as $\lambda_1\le\cdots\le\lambda_n$.  On the
index-one event, exactly one eigenvalue lies below the threshold
$t_N(u,\omega)$; when $t_N(u,\omega)\le-2-c$ this is a speed-$N$
event of cost $J(t_\xi(u,\omega))-o(1)$
(\Cref{lem:B-matrix-rates}(a)), and when $t_N(u,\omega)>-2-c$ we use
only $J\ge0$, which is consistent with the convention
\eqref{eq:Jgoe}.  Meanwhile the empirical bulk remains near the
semicircle law: deviations of $L_n$ have speed $N^2$
\citep{benarous1997large}, so conditioning on one outlier, or
tilting by a determinant power $r\in[1/2,1]$, cannot move the
empirical bulk at the $N$ scale.  On the typical-bulk, bounded-norm
event, the bulk determinant power $r$ contributes at most
$e^{N(rD_\xi(u,\omega)+\eps)}$; the one outlier contributes a single
eigenvalue whose magnitude, like $\alpha_N(u)$, has logarithm $o(N)$
on every speed-$N$ relevant compact set, and eigenvalues
exponentially close to the shift only decrease the marked
integrand.  The bulk-deviation complement has probability
$e^{-cN^2}$ against a determinant growth of at most $e^{C_RN}$, and
the operator-norm tail is controlled by the determinant-moment
estimate of \Cref{lem:B-det-moments}.  This gives, uniformly for
$(u,\omega)$ in compacts and $r\in[\tfrac12,1]$,
\[
 \E\Bigl[
  \alpha_N|\det Q_N(u,\omega)|^{r}\,
  \mathbf1_{\{\ind Q_N(u,\omega)=1\}}
 \Bigr]
 \le\exp\Bigl\{N\Bigl(rD_\xi(u,\omega)
 -J\bigl(t_\xi(u,\omega)\bigr)+\eps\Bigr)\Bigr\},
\]
with the crude growth bound
$\E[\,\cdot\,]\le\E\|Q_N(u,\omega)\|_{\op}^{rn+1}
\le e^{C_0n}(1+|\omega|)^{rn+1}$ for large $|\omega|$ from
\Cref{lem:B-det-moments}.  \Cref{lem:B-tilt-laplace} then integrates
over the shift and yields \eqref{eq:B-half-vs-full}.  A fully
quantitative version of the truncation and outlier decomposition
used here is given in \Cref{lem:B-matrix-rates} below.
\end{proof}

\subsection{Marked one-point Kac--Rice formula}

Let $J\subset(-\infty,0)$ be a small compact interval.  Rotational invariance and the one-point Kac--Rice formula give
\begin{align}
 \E\cW_{N,\beta}(J)
 ={}&\vol(\SphereN)
 p_{\grad H_N(\mathbf n)}(0)
 \int_J N p_{H_N(\mathbf n)}(Nu)e^{-\beta Nu}
 \notag\\
 &\quad\times
 \E\left[
  \alpha_N(u)|\det Q_N(u)|^{1/2}
  \mathbf1_{\cI_{N,1}(u)}
 \right]\dd u,
 \label{eq:B-marked-KR}
\end{align}
where $\mathbf n$ is any fixed point of $\SphereN$ and the matrix
expectation runs over the GOE matrix \emph{and} the shift
fluctuation $\omega_N$ of \eqref{eq:B-Q}.  The standard index-one count has the identical prefactor and integral, with the matrix expectation replaced by
\begin{equation}\label{eq:B-standard-KR-matrix}
 \E\left[
 |\det Q_N(u)|\mathbf1_{\cI_{N,1}(u)}
 \right].
\end{equation}
\Cref{lem:half-det}, or equivalently \Cref{lem:B-matrix-rates}(b) with $r=1/2$ combined with the prefactor limit \Cref{lem:B-matrix-rates}(d), gives the local marked upper rate
\begin{equation}\label{eq:B-marked-rate}
 -\beta u+\varrho_\xi(u)+\Lambda^{(1/2)}_{1,\xi}(u)
 =-\beta u+\Xi^{\mathrm{EK}}_{1,\xi}(u).
\end{equation}
Cover a compact interval $I$ by finitely many bins of width $\delta$, use the maximum of $e^{-\beta Nu}$ in each bin, and apply the uniform version of \eqref{eq:B-marked-rate}.  Sending first $N\to\infty$ and then $\delta\downarrow0$, using the local Lipschitz continuity of $\Xi^{\mathrm{EK}}_{1,\xi}$ (\Cref{lem:B-Lambda-regularity}), proves \eqref{eq:weighted-KR-expectation}.

\subsection{One-point rate identities for marked index-one integrands}

The following lemma collects the shifted-GOE rates used by the soft-saddle estimates in \Cref{app:near-threshold-cocore}; it also makes the truncation argument of \Cref{lem:half-det} quantitative.  Recall $Q_N(u)=Q_N(u,\omega_N)$ and $t_N(u,\omega)$ from \eqref{eq:B-Q}--\eqref{eq:B-Q-omega}.  On the index-one event $\cI_{N,1}(u)$, write $\alpha_N(u)$ for the modulus of the unique negative eigenvalue and $B_N(u)$ for the restriction of $Q_N(u)$ to its stable eigenspace, so that
$\det B_N(u)=|\det Q_N(u)|/\alpha_N(u)$.  Define the Kac--Rice prefactor rate
\begin{equation}\label{eq:B-prefactor-rate}
 \varrho_N(u)
 \defeq\frac1N\log\Bigl[\vol(\SphereN)\,
 p_{\grad H_N(\mathbf n)}(0)\,
 N p_{H_N(\mathbf n)}(Nu)\Bigr],
\end{equation}
so that \eqref{eq:B-marked-KR} with the mark removed becomes the identity
\begin{equation}\label{eq:B-KR-identity}
 \E\Crt_{N,1}(B)
 =\int_B e^{N\varrho_N(u)}\,
 \E\bigl[\alpha_N(u)\det B_N(u)\,
 \mathbf1_{\cI_{N,1}(u)}\bigr]\dd u
\end{equation}
for every Borel set $B\subset\R$.

The next lemma supplies the uniform integrability off the operator-norm event that every determinant-power estimate below requires; without it, Cauchy--Schwarz against the norm tail alone is insufficient because the determinant itself grows with the operator norm.

\begin{lemma}[Determinant moments off the norm event]\label{lem:B-det-moments}
Let $K,W\subset\R$ be compact.  There are constants $c_0>0$, $C_0<\infty$, and $R_0\ge3$, depending only on $K$, $W$, and $\xi$, such that for every $r\in[\tfrac12,3]$, every $u\in K$, every $R\ge R_0$, and every $n$:
\begin{equation}\label{eq:B-det-moment}
 \sup_{\omega\in W}
 \E\bigl[\|Q_N(u,\omega)\|_{\op}^{rn+1}\,
 \mathbf1_{\{\|G_n\|_{\op}\ge R\}}\bigr]
 \le e^{-c_0nR^2},
 \qquad
 \sup_{\omega\in W}
 \E\bigl[\|Q_N(u,\omega)\|_{\op}^{rn+1}\bigr]\le e^{C_0n};
\end{equation}
moreover, for every $\omega\in\R$,
\begin{equation}\label{eq:B-det-moment-omega}
 \E\bigl[\|Q_N(u,\omega)\|_{\op}^{rn+1}\bigr]
 \le e^{C_0n}(1+|\omega|)^{rn+1},
\end{equation}
and the same three bounds hold for the randomly shifted matrix
$Q_N(u)$, using
$\E_Z\bigl[(1+|\omega_N|)^{rn+1}\bigr]\le e^{C_1n}$.
In particular, since every Hessian mark used below is bounded by $\|Q_N(u)\|_{\op}^{rn+1}$ for some $r\le3$, its expectation restricted to $\{\|G_n\|_{\op}\ge R\}$ is at most $e^{-c_0nR^2}$.
\end{lemma}

\begin{proof}
Write $\|Q_N(u,\omega)\|_{\op}\le c_N\|G_n\|_{\op}+\xi'(1)|u|+|\omega|\le C_{K,W}(1+\|G_n\|_{\op})$ for $(u,\omega)\in K\times W$, with $C_{K,W}=\sqrt{\xi''(1)}+\xi'(1)\sup_K|u|+\sup_W|\omega|$; for general $\omega$ the same bound holds with $C_{K,W}(1+|\omega|)$ in place of $C_{K,W}$, which yields \eqref{eq:B-det-moment-omega} from \eqref{eq:B-det-moment}.  Gaussian concentration for the operator norm of the GOE matrix gives $c_1>0$ and $R_0\ge3$ with
\begin{equation}\label{eq:B-norm-tail}
 \Pp\bigl(\|G_n\|_{\op}\ge s\bigr)\le e^{-c_1ns^2}
 \qquad(s\ge R_0,\ n\ge1).
\end{equation}
Decompose $\{\|G_n\|_{\op}\ge R\}$ into the shells
$S_k=\{2^kR\le\|G_n\|_{\op}<2^{k+1}R\}$, $k\ge0$.  On $S_k$ the integrand is at most $[C_{K,W}(1+2^{k+1}R)]^{rn+1}\le[4C_{K,W}2^kR]^{rn+1}$, so
\[
 \E\bigl[\|Q_N(u,\omega)\|_{\op}^{rn+1}\mathbf1_{\{\|G_n\|_{\op}\ge R\}}\bigr]
 \le\sum_{k\ge0}
 \exp\bigl\{(rn+1)\log(4C_{K,W}2^kR)-c_1n4^kR^2\bigr\}.
\]
Enlarge $R_0$ so that, for $R\ge R_0$,
$c_1R^2\ge8\log(4C_{K,W}R)+4$.  Then for every $k\ge0$ and $r\le3$,
\[
 (rn+1)\log(4C_{K,W}2^kR)
 \le4n\bigl[\log(4C_{K,W}R)+k\log2\bigr]
 \le\tfrac{c_1}2n4^kR^2-2n-kn,
\]
using $rn+1\le4n$, the inequality $4^k\ge1+3k$, and the two consequences
$\tfrac{c_1}2R^2\ge4\log(4C_{K,W}R)+2$ and
$\tfrac32c_1R^2\ge4\log2+1$ of the choice of $R_0$; hence the $k$th summand is at most $e^{-c_1n4^kR^2/2}e^{-kn}$ and the series is bounded by $2e^{-c_1nR^2/2}$.  This proves the first bound with $c_0=c_1/4$.  The second bound follows by splitting at $R_0$: the contribution of $\{\|G_n\|_{\op}<R_0\}$ is at most $[C_{K,W}(1+R_0)]^{rn+1}\le e^{C_0n}$, and the tail is smaller than one.  Finally, for the randomly shifted matrix, independence of $Z$ and $G_n$ factorizes the expectation, and
$\E_Z[(1+|\omega_N|)^{rn+1}]\le2^{rn+1}\bigl(1+\E|v_\xi Z/\sqrt N|^{rn+1}\bigr)
\le2^{rn+1}\bigl(1+(v_\xi\sqrt{(rn+1)/N})^{rn+1}\bigr)\le e^{C_1n}$,
since $rn+1\le4N$.
\end{proof}

\begin{lemma}[Marked index-one rates]\label{lem:B-matrix-rates}
Let $K\subset(-\infty,0)$ be compact.  The following hold uniformly for $u\in K$.
\begin{enumerate}[label=\textup{(\alph*)},leftmargin=*]
 \item \textup{(Outlier upper rate.)}  For every compact $W\subset\R$ and $\eps>0$ there is $N_0(\eps,K,W)$ such that
 \[
  \Pp\bigl(\lambda_1(G_n)\le t_N(u,\omega)\bigr)
  \le\exp\bigl\{-N\bigl(J(t_\xi(u,\omega))-\eps\bigr)\bigr\},
  \qquad (u,\omega)\in K\times W,\ N\ge N_0,
 \]
 where $J$ is the one-sided outlier rate \eqref{eq:Jgoe}; consequently, integrating over the shift,
 \[
  \Pp\bigl(\ind Q_N(u)\ge1\bigr)
  \le\exp\bigl\{-N\bigl(I_{1,\xi}(u)-\eps\bigr)\bigr\}.
 \]
 The rate $J$ coincides on $(-\infty,-2]$ with the one-outlier GOE rate function
 \[
  J_{\mathrm{GOE}}(t)=\frac12\int_t^{-2}\sqrt{x^2-4}\,\dd x ,
 \]
 and in the pure case $J(t_\xi(u,0))=I_{1,\xi}(u)$ for $u\le-E_\infty(\xi)$.
 \item \textup{(Fractional determinant upper bound.)}  For every $r\in[\tfrac12,2]$ and $\eps>0$, for all sufficiently large $N$,
 \[
  \E\bigl[(\det B_N(u))^{r}\,
  \mathbf1_{\cI_{N,1}(u)}\bigr]
  \le\exp\bigl\{N\bigl(\Lambda^{(r)}_{1,\xi}(u)+\eps\bigr)\bigr\};
 \]
 moreover, conditionally on $\{\omega_N=\omega\}$ with $(u,\omega)\in K\times W$,
 \[
  \E\bigl[(\det B_N(u,\omega))^{r}\,
  \mathbf1_{\{\ind Q_N(u,\omega)=1\}}\bigr]
  \le\exp\bigl\{N\bigl(rD_\xi(u,\omega)-J(t_\xi(u,\omega))+\eps\bigr)\bigr\}.
 \]
 \item \textup{(Standard mark rate.)}  Uniformly for $u\in K$,
 \[
  \limsup_{N\to\infty}\frac1N\log
  \E\bigl[\alpha_N(u)\det B_N(u)\,
  \mathbf1_{\cI_{N,1}(u)}\bigr]
  \le \Lambda^{(1)}_{1,\xi}(u).
 \]
 \item \textup{(Prefactor limit.)}
 \[
  \lim_{N\to\infty}\varrho_N(u)
  =\varrho_\xi(u)
  =\frac12-\frac12\log\xi'(1)-\frac{u^2}2,
 \]
 uniformly on compacts; consequently, by \eqref{eq:Theta-mixed}, the annealed index-one rate at energy $u$ is $\Theta_{1,\xi}(u)=\varrho_\xi(u)+\Lambda^{(1)}_{1,\xi}(u)$, and in the pure case the identity
 $\varrho_\xi(u)+D_\xi(u)-J(t_\xi(u,0))=\Theta_{1,\xi}(u)$ recovers the explicit formula \eqref{eq:Theta-explicit} with $k=1$.
\end{enumerate}
\end{lemma}

\begin{proof}
Fix a compact $W\subset\R$; by \Cref{lem:B-Lambda-regularity} and \Cref{lem:B-tilt-laplace} it suffices, for (b) and (c), to prove the conditional estimates uniformly on $K\times W$ and to control large $|\omega|$ by the growth bound \eqref{eq:B-det-moment-omega}.  Abbreviate $t=t_N(u,\omega)$.  Uniformity on $K\times W$ is obtained as follows.  The map $(u,\omega)\mapsto t_N(u,\omega)$ is affine with coefficients converging to those of $t_\xi$, the potential $\Omega$ and the rate $J$ are Lipschitz on compacts (\Cref{lem:B-Lambda-regularity}), and hence there is $L_{K,W}<\infty$ with
\[
 |\Omega(t_N(u,\omega))-\Omega(t_N(u',\omega'))|
 +|D_\xi(u,\omega)-D_\xi(u',\omega')|
 +|J(t_\xi(u,\omega))-J(t_\xi(u',\omega'))|
 \le L_{K,W}\bigl(|u-u'|+|\omega-\omega'|\bigr)
\]
for all $(u,\omega),(u',\omega')\in K\times W$ and all large $N$.  Three transfer mechanisms are needed, because the events used below depend on $(u,\omega)$ through more than the threshold value.

\emph{First (bulk events).}  The truncated logarithm $\phi_{\eta,t}(x)=\log\bigl(((x-t)\vee\eta)\wedge2R\bigr)$ used in (b) depends on $t$ through its entire graph, not only through the value $\Omega(t)$; it satisfies the sup-norm Lipschitz bound
\[
 \sup_{x\in\R}\bigl|\phi_{\eta,t}(x)-\phi_{\eta,t'}(x)\bigr|
 \le\frac{|t-t'|}{\eta},
\]
because $x\mapsto((x-t)\vee\eta)\wedge2R$ is $1$-Lipschitz in $t$ and takes values in $[\eta,2R]$, where the logarithm is $\eta^{-1}$-Lipschitz.  Given $\eps>0$, fix the mesh
$\theta=\min\{\eps\eta/16,\ \eps/(16L_{K,W})\}$ and a finite $\theta$-net $(u_1,\omega_1),\ldots,(u_m,\omega_m)$ of $K\times W$.  For $(u,\omega)$ and its nearest net point, every realization satisfies, for all large $N$,
$|\int\phi_{\eta,t_N(u,\omega)}\dd L_n-\int\phi_{\eta,t_\xi(u_j,\omega_j)}\dd L_n|\le\eps/8$ and $|\Omega(t_N(u,\omega))-\Omega(t_\xi(u_j,\omega_j))|\le\eps/8$; hence the complement of the $(u,\omega)$-dependent bulk event
$\{\int\phi_{\eta,t_N(u,\omega)}\dd L_n\le\Omega(t_N(u,\omega))+\eps/2\}$
is contained in the \emph{fixed} closed event
$\{\int\phi_{\eta,t_\xi(u_j,\omega_j)}\dd L_n\ge\Omega(t_\xi(u_j,\omega_j))+\eps/4\}$.

\emph{Second (uniform positive $N^2$-rate).}  Each of the $m$ fixed closed events above excludes the semicircle law: $\nu\mapsto\int\phi_{\eta,t}\dd\nu$ is bounded and continuous, and $\int\phi_{\eta,t}\rho_{\sclaw}\dd x\le\Omega(t)+C\eta$ for every $t$, with equality up to $\eps/8$ after decreasing $\eta$.  The large-deviation upper bound at speed $N^2$ of \citet{benarous1997large}, whose good rate function vanishes only at the semicircle law, therefore supplies constants $c_j(\eps,\eta)>0$ with probability at most $e^{-c_jN^2}$, and $c(\eps,\eta)=\min_{1\le j\le m}c_j>0$ serves simultaneously for every $(u,\omega)\in K\times W$.  Every speed-$N^2$ bulk estimate below is uniform on $K\times W$ in this sense.

\emph{Third (outlier estimates).}  The events $\{\lambda_1(G_n)\le t\}$ are monotone in $t$; the same finite net, the Lipschitz bound on $J$, and monotonicity transfer the pointwise outlier estimates to all intermediate $(u,\omega)$ at the cost of replacing $\eps$ by $2\eps$.

Together, these three mechanisms convert every pointwise estimate below into a uniform one, with a single deterministic error sequence valid simultaneously on $K\times W$.  We use this upgrade silently in what follows.

\emph{(a)}  For a fixed threshold $s\le-2-c$, the large-deviation upper bound for the smallest GOE eigenvalue \citep[Theorem~A.1]{auffinger2013random} (stated there for the $k$th largest eigenvalue above the right edge; it applies to the smallest eigenvalue below the left edge after the reflection $G\mapsto-G$, under which the GOE law is invariant) gives $\Pp(\lambda_1(G_n)\le s)\le e^{-n(J_{\mathrm{GOE}}(s)-\eps)}$ for large $n$.  If $t_\xi(u,\omega)>-2$ the claimed bound is trivial since $J=0$ there; thresholds in $[-2-c,-2]$ are handled by monotonicity and the vanishing of $J_{\mathrm{GOE}}$ at $-2$.  The moving threshold and the uniformity are handled by the third transfer mechanism.  For the integrated statement, apply \Cref{lem:B-tilt-laplace} to $\Phi_N(\omega)=\Pp(\lambda_1(G_n)\le t_N(u,\omega))\le1$ with $\varphi(\omega)=-J(t_\xi(u,\omega))$, and note
$\sup_\omega\{-J-\mathfrak p_\xi\}=-I_{1,\xi}(u)$ by \eqref{eq:I1}; the event $\{\ind Q_N(u)\ge1\}$ is contained in $\{\lambda_1(G_n)\le t_N(u,\omega_N)\}$.  The antiderivative identity
\begin{equation}\label{eq:B-antiderivative}
 \frac12\int_2^{s}\sqrt{x^2-4}\,\dd x
 =\frac{s\sqrt{s^2-4}}4
 -\log\frac{s+\sqrt{s^2-4}}2,
 \qquad s\ge2,
\end{equation}
identifies $J_{\mathrm{GOE}}$ with the explicit expression used in \eqref{eq:Omega-explicit}; in the pure case $v_\xi=0$ the infimum in \eqref{eq:I1} is evaluation at $\omega=0$.

\emph{(b)}  Condition on $\{\omega_N=\omega\}$, $(u,\omega)\in K\times W$.  Fix $\eta\in(0,1)$ and $R=R(K,W)$ large, and let
$\phi_\eta(x)=\log\bigl(((x-t)\vee\eta)\wedge2R\bigr)$, a bounded Lipschitz function.  Introduce the events
\[
 \mathcal N_R=\{\|G_n\|_{\op}\le R\},
 \qquad
 \mathcal B_\eps
 =\Bigl\{\textstyle\int\phi_\eta\dd L_n
 \le\Omega(t)+\tfrac\eps2\Bigr\}.
\]
On $\mathcal N_R\cap\mathcal B_\eps\cap\{\lambda_1\le t\}$, using $\log(\lambda_i-t)\le\phi_\eta(\lambda_i)$ for $i\ge2$,
\[
 \log\det B_N(u,\omega)
 \le n\log c_N+n\int\phi_\eta\dd L_n
 +\log\tfrac1\eta
 \le N\bigl(D_\xi(u,\omega)+\eps\bigr)
\]
for all large $N$, because $c_N\to\sqrt{\xi''(1)}$ and $\Omega(t_N(u,\omega))\to\Omega(t_\xi(u,\omega))$ uniformly on $K\times W$; the contribution of this event is at most
$e^{rN(D_\xi(u,\omega)+\eps)}\Pp(\lambda_1\le t)$, and (a) applies (if $t>-2$ the index-one event costs nothing at speed $N$ and $J=0$).  On $\mathcal N_R\setminus\mathcal B_\eps$, the empirical measure lies in a closed set not containing $\rho_{\sclaw}$; by the first and second transfer mechanisms, this event is contained in one of finitely many fixed closed events, so the large-deviation principle at speed $N^2$ of \citet{benarous1997large} bounds the probability by $e^{-c(\eps,\eta)N^2}$ uniformly, while $(\det B_N)^r\le(c_N(R+|t|))^{rn}\le e^{C_{K,W}N}$ there.  On $\mathcal N_R^c$, one has $(\det B_N)^r\le\|Q_N(u,\omega)\|_{\op}^{rn}$ with $r\le2$, so \Cref{lem:B-det-moments} bounds the contribution by $e^{-c_0nR^2}$; this is negligible against the first term once $R=R(K,W)$ is chosen with $c_0R^2\ge1+\sup_{K\times W}(2|D_\xi|+J)$, because $D_\xi$ and $J$ are bounded on $K\times W$.  This proves the conditional bound; the unconditional bound follows from \Cref{lem:B-tilt-laplace} with $\varphi(\omega)=rD_\xi(u,\omega)-J(t_\xi(u,\omega))$, whose logarithmic growth in $|\omega|$ is supplied by \eqref{eq:B-det-moment-omega}, and the supremum is $\Lambda^{(r)}_{1,\xi}(u)$ by \eqref{eq:Lambda-rk}.

\emph{(c)}  This is (b) with $r=1$ together with $\alpha_N(u)\le c_N(R+|t|)\le C_{K,W}$ on $\mathcal N_R$; the complement is controlled as in (b), and the shift is integrated in the same way.  We do not state a matching lower bound: all uses of the marked rates in this paper are upper bounds, and the corresponding lower bounds---known in the pure case from the complexity asymptotics of \citet{auffinger2013random} and the second-moment analysis of \citet{subag2017complexity}---are not needed.

\emph{(d)}  The densities in \eqref{eq:B-prefactor-rate} are explicit Gaussians.  Differentiating the covariance \eqref{eq:covariance} shows that the tangent gradient at a fixed point has i.i.d.\ $N(0,\xi'(1))$ coordinates and that $H_N(\mathbf n)\sim N(0,N)$ by the normalization $\xi(1)=1$, so
$p_{\grad H_N(\mathbf n)}(0)=(2\pi \xi'(1))^{-(N-1)/2}$ and
$Np_{H_N(\mathbf n)}(Nu)=N(2\pi N)^{-1/2}e^{-Nu^2/2}$.  With \eqref{eq:sphere-volume-asymptotic},
\begin{equation}\label{eq:B-prefactor-limit}
 \varrho_N(u)\longrightarrow
 \varrho_\xi(u)=\frac12\log(2\pi e)-\frac12\log(2\pi \xi'(1))-\frac{u^2}2
 =\frac12-\frac12\log \xi'(1)-\frac{u^2}2,
\end{equation}
uniformly on compacts.  In the pure case, \eqref{eq:Dxi}, \eqref{eq:Omega-explicit}, and \eqref{eq:B-antiderivative} give
\[
 \varrho_\xi(u)+D_\xi(u)-J(t_\xi(u,0))
 =\frac12\log(p-1)
 -\frac{p-2}{4(p-1)}u^2
 -2J(t_\xi(u,0))
 =\Theta_{1,\xi}(u),
 \qquad u\le-E_\infty(\xi),
\]
which is \eqref{eq:Theta-explicit} with $k=1$.
\end{proof}

The same machinery yields the annealed complexity upper bound quoted
as \Cref{prop:complexity-upper}.

\begin{lemma}[Annealed index-$k$ upper bound]\label{lem:B-index-count}
Let $k\in\{1,2\}$ and $a>0$.  Then
\begin{equation}\label{eq:B-index-count}
 \limsup_{N\to\infty}\frac1N
 \log\E\Crt_{N,\ge k}((-\infty,-a])
 \le\sup_{u\le-a}\Theta_{k,\xi}(u).
\end{equation}
\end{lemma}

\begin{proof}
The Kac--Rice identity \eqref{eq:B-KR-identity}, extended to the nonnegative Borel mark $|\det Q_N(u)|\mathbf1_{\{\ind Q_N(u)\ge k\}}$ by the monotone truncation argument of \Cref{prop:soft-marked-KR}, gives
\[
 \E\Crt_{N,\ge k}(B)
 =\int_B e^{N\varrho_N(u)}\,
 \E\bigl[|\det Q_N(u)|\,
 \mathbf1_{\{\ind Q_N(u)\ge k\}}\bigr]\dd u .
\]
\emph{Compact energy range.}  Fix $C\ge a+1$, to be chosen in the tail step, and let $K_C=[-C,-a]$.  Conditionally on $\{\omega_N=\omega\}$ with $\omega$ in a compact $W$, the event $\{\ind Q_N(u,\omega)\ge k\}$ forces $\lambda_k(G_n)\le t_N(u,\omega)$, whose probability is at most $e^{-n(kJ(t_\xi(u,\omega))-\eps)}$: for $k=1$ this is \Cref{lem:B-matrix-rates}(a), and for $k=2$ it is the two-outlier estimate \Cref{lem:B-two-outlier} below (whose proof is independent of this lemma) when $t_\xi(u,\omega)\le-2-c$, and trivial otherwise since $J\ge0$ vanishes near the edge and $J$ is continuous.  On the events of the proof of \Cref{lem:B-matrix-rates}(b), $|\det Q_N(u,\omega)|\le C_{K,W}\,e^{N(D_\xi(u,\omega)+\eps)}$ (the at most $k\le2$ outlier eigenvalues contribute a bounded factor on $\mathcal N_R$), and the complementary events are controlled exactly as there.  Hence the conditional matrix expectation is at most $e^{N(D_\xi(u,\omega)-kJ(t_\xi(u,\omega))+2\eps)}$, uniformly on $K_C\times W$; the global growth hypothesis of \Cref{lem:B-tilt-laplace} is supplied by \eqref{eq:B-det-moment-omega}, since $|\det Q_N(u,\omega)|\le\|Q_N(u,\omega)\|_{\op}^{n}$, and the limiting function $\varphi(\omega)=D_\xi(u,\omega)-kJ(t_\xi(u,\omega))$ grows at most logarithmically in $|\omega|$.  \Cref{lem:B-tilt-laplace} then integrates the shift to give
\[
 \E\bigl[|\det Q_N(u)|\,\mathbf1_{\{\ind Q_N(u)\ge k\}}\bigr]
 \le e^{N(\Lambda^{(1)}_{k,\xi}(u)+3\eps)},
 \qquad u\in K_C,
\]
uniformly.  Multiplying by $e^{N\varrho_N(u)}$, using \Cref{lem:B-matrix-rates}(d), covering $K_C$ by finitely many bins, and using the continuity of $\Theta_{k,\xi}$ (\Cref{lem:B-Lambda-regularity}) bounds the contribution of $K_C$ by $e^{N(\sup_{u\in K_C}\Theta_{k,\xi}(u)+4\eps)}$.

\emph{Energy tail.}  For $u\le-C$, drop the index indicator.  The Gaussian densities give the deterministic bound $\varrho_N(u)\le C_\xi'-u^2/2$ for all $u$ and $N$.  For the determinant, factor the $u$- and $\omega$-dependence deterministically:
\[
 |\det Q_N(u)|
 \le\bigl(c_N\|G_n\|_{\op}+\xi'(1)|u|+|\omega_N|\bigr)^n
 \le(1+|u|)^n\,3^n\bigl[(c_N\|G_n\|_{\op})^n+\xi'(1)^n+|\omega_N|^n\bigr],
\]
and $\E[(c_N\|G_n\|_{\op})^n]\le e^{C_0n}$ together with $\E|\omega_N|^n\le(v_\xi\sqrt{n/N})^n\le(2v_\xi)^n$ give
$\E|\det Q_N(u)|\le e^{n\log(C_\xi''(1+|u|))}$ with $C_\xi''$ independent of $u$.  The integrand is therefore at most
$\exp\{N(C_\xi'''+\log(1+|u|)-u^2/2)\}\le e^{-Nu^2/4}$
once $C=C(\xi)$ is large enough, and
$\int_{-\infty}^{-C}e^{-Nu^2/4}\dd u\le e^{-NC^2/8}$ for all large $N$.  Combining the two contributions and letting $\eps\downarrow0$ proves \eqref{eq:B-index-count}, since $\sup_{u\le-a}\Theta_{k,\xi}(u)\ge\Theta_{k,\xi}(-a)>-C^2/8$ after enlarging $C$ if necessary.
\end{proof}

\subsection{Uniform stable gap below the \texorpdfstring{$E_2$}{E2} layer}\label{subsec:B-stable-gap}

This subsection proves \Cref{prop:stable-gap-layer} through an explicit, quantitative two-outlier estimate.  The first lemma is the required uniform large-deviation input for the second smallest GOE eigenvalue.

\begin{lemma}[Uniform two-outlier estimate]\label{lem:B-two-outlier}
Let $S\subset(-\infty,-2)$ be compact and $\eps>0$.  For all sufficiently large $n$,
\begin{equation}\label{eq:B-two-outlier-LDP}
 \Pp\bigl(\lambda_2(G_n)\le s\bigr)
 \le\exp\bigl\{-n\bigl(2J_{\mathrm{GOE}}(s)-\eps\bigr)\bigr\}
 \qquad\text{for every }s\in S,
\end{equation}
where $J_{\mathrm{GOE}}$ is the one-outlier rate of \Cref{lem:B-matrix-rates}\textup{(a)}.
\end{lemma}

\begin{proof}
Theorem~A.1 of \citet{auffinger2013random} states that, for each fixed $k\ge1$, the $k$th largest GOE eigenvalue satisfies a large-deviation principle at speed $n$ with good rate function $kI_1(x)$ for $x$ at or above the right edge and $+\infty$ below it (an eigenvalue of fixed order inside the bulk is a speed-$n^2$ event).  Applying this with $k=2$ after the reflection $G\mapsto-G$, under which the GOE law is invariant, the second smallest eigenvalue $\lambda_2(G_n)$ satisfies a marginal large-deviation principle at speed $n$ whose good rate function equals $2J_{\mathrm{GOE}}(s)$ for $s\le-2$ and $+\infty$ for $s>-2$: pushing the second eigenvalue below the edge forces two eigenvalues below $s$, at cost $J_{\mathrm{GOE}}$ each.  Only the large-deviation \emph{upper bound} for the closed set $\{\lambda_2\le s\}$ is used; its rate infimum is $2J_{\mathrm{GOE}}(s)$ since $J_{\mathrm{GOE}}$ is continuous and strictly decreasing on $(-\infty,-2]$.  This is a marginal statement about a single order statistic; no joint principle for the pair $(\lambda_1,\lambda_2)$ is invoked.  For the uniformity, enlarge $S$ to a compact interval $[s_-,s_+]\subset(-\infty,-2)$, choose a finite partition $s_-=s_0<s_1<\dots<s_m=s_+$ with $J_{\mathrm{GOE}}(s_{j})-J_{\mathrm{GOE}}(s_{j+1})\le\eps/4$ for every $j$ (possible by uniform continuity), and apply the pointwise bound with $\eps/4$ at each $s_{j+1}$.  For $s\in[s_j,s_{j+1}]$, monotonicity gives
\[
 \Pp(\lambda_2\le s)\le\Pp(\lambda_2\le s_{j+1})
 \le e^{-n(2J_{\mathrm{GOE}}(s_{j+1})-\eps/4)}
 \le e^{-n(2J_{\mathrm{GOE}}(s)-\eps)}
\]
for all $n$ large, simultaneously for all $j$ because the net is finite.
\end{proof}

\begin{lemma}[Expected count of index-one points with a soft stable eigenvalue]\label{lem:B-soft-stable-count}
Fix $b>E_2(\xi)$.  There are $\kappa_{\xi,b}>0$ and $c_{\xi,b}>0$ such that, for all sufficiently large $N$,
\begin{equation}\label{eq:B-soft-stable-count}
 \E\#\bigl\{z:\ \grad H_N(z)=0,\ H_N(z)\le-Nb,\
 \ind Q_z=1,\ \lambda_2(Q_z)\le\kappa_{\xi,b}\bigr\}
 \le e^{-c_{\xi,b}N}.
\end{equation}
\end{lemma}

\begin{proof}
Write $\cM_N(\kappa)$ for the count on the left of \eqref{eq:B-soft-stable-count}.  By the identity \eqref{eq:B-KR-identity} extended by the Hessian indicator, which is a nonnegative Borel mark,
\begin{equation}\label{eq:B-M-KR}
 \E\cM_N(\kappa)
 =\int_{-\infty}^{-b}e^{N\varrho_N(u)}\,
 \E\bigl[|\det Q_N(u)|\,
 \mathbf1_{\{\ind Q_N(u)=1,\ \lambda_2(Q_N(u))\le\kappa\}}\bigr]\dd u .
\end{equation}

\emph{Spectral mapping.}  Conditionally on $\{\omega_N=\omega\}$, one has $\lambda_i(Q_N(u,\omega))=c_N\lambda_i(G_n)-(\xi'(1)u+\omega)$, so the Hessian event in \eqref{eq:B-M-KR} implies
\begin{equation}\label{eq:B-second-outlier}
 \lambda_2(G_n)\le t_N(u,\omega)+\frac\kappa{c_N}.
\end{equation}

\emph{Compact energy range.}  Fix a constant $C=C(\xi)\ge b+1$ chosen in the tail step below, let $K_C=[-C,-b]$, and fix a compact $W\subset\R$ for the shift.  Split the conditional expectation in \eqref{eq:B-M-KR} on the events of the proof of \Cref{lem:B-matrix-rates}(b).  On $\mathcal N_R\cap\mathcal B_\eps$ one has $|\det Q_N(u,\omega)|\le e^{N(D_\xi(u,\omega)+\eps)}$, so \eqref{eq:B-second-outlier} and the two-outlier estimate \Cref{lem:B-two-outlier}, applied where the displaced threshold $t_N(u,\omega)+\kappa/c_N$ lies below $-2$ (and trivially where it does not, since $J\ge0$ vanishes there and is continuous at the edge), give
\[
 \E\bigl[|\det Q_N(u,\omega)|\,\mathbf1_{\{\cdots\}}\,
 \mathbf1_{\mathcal N_R\cap\mathcal B_\eps}\bigr]
 \le\exp\Bigl\{N\Bigl(D_\xi(u,\omega)
 -2J\bigl(t_\xi(u,\omega)+\tfrac\kappa{c_N}\bigr)+2\eps\Bigr)\Bigr\},
 \qquad (u,\omega)\in K_C\times W .
\]
The rate $J$ is Lipschitz on compacts with a constant $L_{\xi,b}$ (\Cref{lem:B-Lambda-regularity}), so
\[
 2J\bigl(t_\xi(u,\omega)+\tfrac\kappa{c_N}\bigr)
 \ge2J\bigl(t_\xi(u,\omega)\bigr)-L_{\xi,b}\,\kappa-o(1),
\]
uniformly on $K_C\times W$.  Integrating the shift with \Cref{lem:B-tilt-laplace} (the large-$|\omega|$ growth being supplied by \eqref{eq:B-det-moment-omega}), and using \eqref{eq:B-prefactor-limit} together with the definition
$\Theta_{2,\xi}=\varrho_\xi+\Lambda^{(1)}_{2,\xi}$ of \eqref{eq:Theta-mixed}, the contribution of $K_C$ to \eqref{eq:B-M-KR} has $\frac1N\log$ at most
\[
 \sup_{u\in K_C}\Theta_{2,\xi}(u)+L_{\xi,b}\,\kappa+3\eps+o(1).
\]
Since $b>E_2(\xi)$, one has $\Theta_{2,\xi}(u)<0$ for every $u\le-b$ and $\Theta_{2,\xi}(u)\to-\infty$ as $u\to-\infty$, so $\vartheta_b\defeq-\sup_{u\le-b}\Theta_{2,\xi}(u)>0$ by continuity (\Cref{lem:B-Lambda-regularity}); choose
\[
 \eps=\frac{\vartheta_b}{12},
 \qquad
 \kappa_{\xi,b}=\frac{\vartheta_b}{4L_{\xi,b}},
\]
so that this exponent is at most $-\vartheta_b/2<0$ for all large $N$.  The complementary events on $K_C$ are handled exactly as in \Cref{lem:B-matrix-rates}(b): on $\mathcal N_R\setminus\mathcal B_\eps$ the integrand is at most $e^{C_{K,W}N}$ while the probability is at most $e^{-cN^2}$; off $\mathcal N_R$, \Cref{lem:B-det-moments} bounds the contribution by $e^{-c_0nR^2}$, and $R$ is chosen with $c_0R^2\ge1+\sup_{K_C}|\varrho_\xi|$.

\emph{Energy tail.}  For $u\le-C$, drop the Hessian indicator.  The defining Gaussian densities give the deterministic bound
$\varrho_N(u)\le C_\xi'-u^2/2$ for all $u$ and all $N$.  For the determinant, note that \Cref{lem:B-det-moments} is stated only for $(u,\omega)$ in compacts and cannot be applied at unbounded $u$; instead, factor out the $u$- and $\omega$-dependence deterministically as in the proof of \Cref{lem:B-index-count}:
\[
 |\det Q_N(u)|
 \le\bigl(c_N\|G_n\|_{\op}+\xi'(1)|u|+|\omega_N|\bigr)^n
 \le(1+|u|)^n\,3^n\bigl[(c_N\|G_n\|_{\op})^n+\xi'(1)^n+|\omega_N|^n\bigr]
 \qquad\text{for every }u\in\R,
\]
and $\E[(c_N\|G_n\|_{\op})^n]\le1+\E\|Q_N(0,0)\|_{\op}^{n+1}\le e^{C_0n}$ (\Cref{lem:B-det-moments} applied at the single point $(0,0)$ with $r=1$), while $\E|\omega_N|^n\le(v_\xi\sqrt{n/N})^n\le(2v_\xi)^n$.  Hence
$\E|\det Q_N(u)|\le e^{n\log(C_\xi''(1+|u|))}$ with $C_\xi''$ independent of $u$, so the integrand of \eqref{eq:B-M-KR} is at most
$\exp\{N(C_\xi'''+\log(1+|u|)-u^2/2)\}\le e^{-Nu^2/4}$
once $C=C(\xi)$ is large enough, and
\[
 \int_{-\infty}^{-C}e^{-Nu^2/4}\dd u\le e^{-NC^2/8}
\]
for all large $N$.  Combining the three contributions proves \eqref{eq:B-soft-stable-count} with
\[
 c_{\xi,b}=\tfrac12\min\bigl\{\vartheta_b/2,\,1,\,C^2/8\bigr\}.
\]
\end{proof}

By Markov's inequality, \Cref{lem:B-soft-stable-count} proves \Cref{prop:stable-gap-layer} with the stated $\kappa_{\xi,b}$; no intersection with the derivative event is needed, since the energy tail is controlled directly inside the Kac--Rice integral.

\section{Spherical geometry and tensor derivative bounds}\label{app:geometry}

The radius-$\sqrt N$ geometry is used both in the shifted-GOE law and in the $N/\sqrt\beta$ localization scale.  Let
\begin{equation}\label{eq:C-projection}
 P_\sigma\defeq I-\frac1N\sigma\sigma^\top.
\end{equation}
For an ambient $C^2$ extension $F$, the spherical gradient and Hessian are
\begin{align}
 \grad_{\SphereN}F(\sigma)
 &=P_\sigma\nabla F(\sigma),
 \label{eq:C-gradient}\\
 \Hess_{\SphereN}F(\sigma)
 &=P_\sigma\nabla^2F(\sigma)P_\sigma
 -\frac{\langle\sigma,\nabla F(\sigma)\rangle}{N}
 I_{T_\sigma\SphereN}.
 \label{eq:C-Hessian}
\end{align}
For each pure component, Euler's identity for the degree-$p$
homogeneous polynomial $H_{N,p}$ gives
$\langle\sigma,\nabla H_{N,p}(\sigma)\rangle=pH_{N,p}(\sigma)$, so
for the mixture
\begin{equation}\label{eq:C-Euler}
 \langle\sigma,\nabla H_N(\sigma)\rangle
 =\sum_{p\in P}p\,\gamma_p\,H_{N,p}(\sigma)
 \ \defeq\ NW(\sigma).
\end{equation}
The vector $(H_{N,p}(\sigma))_{p\in P}$ is centered Gaussian with
independent coordinates of variance $N$, and
$H_N(\sigma)=\sum_p\gamma_pH_{N,p}(\sigma)$, so conditioning on
$H_N(\sigma)=Nu$ (with $\xi(1)=\sum_p\gamma_p^2=1$) gives the
conditional law
\begin{equation}\label{eq:C-radial-law}
 W\;\Big|\;\{H_N(\sigma)=Nu\}
 \ \sim\
 N\Bigl(\xi'(1)u,\ \frac{v_\xi^2}{N}\Bigr),
 \qquad
 v_\xi^2=\sum_pp^2\gamma_p^2-\Bigl(\sum_pp\gamma_p^2\Bigr)^2,
\end{equation}
by the standard Gaussian conditioning formulas
($\E[W\mid H_N=Nu]=\sum_pp\gamma_p^2\,u=\xi'(1)u$ and
$\Var(W\mid H_N)=\Var W-\Cov(W,H_N/N)^2/\Var(H_N/N)$ with
$\Var W=(\xi'(1)+\xi''(1))/N$ and $\Cov(W,H_N/N)=\xi'(1)/N$,
$\Var(H_N/N)=1/N$).  Through \eqref{eq:C-Hessian}, the radial term
$-WI$ explains the randomly shifted diagonal
$-(\xi'(1)u+\omega_N)I$ in \eqref{eq:shifted-GOE}: in the pure case
$W=pH_N/N$ is a deterministic function of the energy and
$v_\xi=0$, which is precisely Euler's identity degenerating the
shift fluctuation.

To prove \Cref{prop:derivative-scale}, it suffices by the triangle
inequality over the finitely many components to treat one pure
component $H_{N,p}$.  Write $\sigma=\sqrt N x$ with $\|x\|=1$ and let $\mathcal J_N$ denote the Gaussian $p$-linear form in \eqref{eq:hamiltonian}.  A fixed $1/4$-net of the unit sphere has at most $9^N$ points.  Gaussian tails on the $p$-fold product net, followed by the standard multilinear extension inequality, yield
\begin{equation}\label{eq:C-tensor}
 \Pp\left(
 \sup_{\|x_1\|=\cdots=\|x_p\|=1}
 |\mathcal J_N(x_1,\ldots,x_p)|
 \le C_p\sqrt N
 \right)\ge1-e^{-c_pN}.
\end{equation}
The $r$th ambient derivative, evaluated on unit directions, is a $p$-dependent constant times
\begin{equation}\label{eq:C-r-derivative}
 N^{(1-r)/2}
 \mathcal J_N(v_1,\ldots,v_r,x,\ldots,x),
 \qquad0\le r\le3.
\end{equation}
Equations \eqref{eq:C-tensor}--\eqref{eq:C-r-derivative} give the four scales in \eqref{eq:derivative-event} for each component, hence for $H_N$.  Covariant differentiation of \eqref{eq:C-gradient}--\eqref{eq:C-Hessian} introduces only powers of $N^{-1/2}$ and preserves them.

Finally, in geodesic normal coordinates $x$ with $r=\|x\|$ around any point of $\SphereN$, the volume Jacobian is
\begin{equation}\label{eq:C-Jacobian}
 J_N(r)\defeq\left[\frac{\sin(r/\sqrt N)}{r/\sqrt N}\right]^{N-2}.
\end{equation}
On the natural Gaussian scale $r\le R\sqrt{N/\beta}$,
\begin{equation}\label{eq:C-geometry-error}
 e^{-C_RN/\beta}\le J_N(r)\le1,
 \qquad
 g^{ij}(x)=\delta_{ij}+O_R(\beta^{-1}).
\end{equation}
At a critical point, \Cref{prop:derivative-scale} gives
\begin{equation}\label{eq:C-Taylor-error}
 \beta\left|H_N(\exp_\sigma x)-H_N(\sigma)
 -\frac12\langle x,Q_\sigma x\rangle\right|
 \le C_{\xi,R}\frac{N}{\sqrt\beta}.
\end{equation}
This is the origin of the localization loss in \eqref{eq:cocore-localization} and \eqref{eq:network-realization}.

\section{Near-threshold co-core localization: proof}\label{app:near-threshold-cocore}

This appendix proves \Cref{thm:cocore-localization}.  The quantitative
stable-manifold geometry, the stable-disk separation and slab
crossing, the finite-width profile calculation, and the good-saddle
summation are developed first.  The soft-saddle contribution is
controlled by the annealed marked Kac--Rice estimate
\Cref{prop:soft-marked-KR}, proved using the determinant-moment
estimate of \Cref{lem:B-det-moments} and the marked one-point rates of
\Cref{lem:B-matrix-rates}; the deterministic additive term in
\eqref{eq:cocore-localization} makes any lower bound on the random
conductance sum unnecessary.

\subsection{Uniform stable co-cores near the first saddle threshold}

Fix once and for all
\begin{equation}\label{eq:D-bbar}
 \bar b\in(E_2(\xi),E_1(\xi)).
\end{equation}
On the event in \Cref{prop:stable-gap-layer}, every index-one critical point below $-N\bar b$ has a stable Hessian block bounded below by
\begin{equation}\label{eq:D-stable-block}
 B_z\ge\kappa I,
 \qquad \kappa\defeq \kappa_{\xi,\bar b}>0.
\end{equation}
Intersect this event with the derivative event \eqref{eq:derivative-event}, and let $L=L_\xi$ be a deterministic bound for the quantities in \eqref{eq:derivative-functional}.

\begin{lemma}[Uniform quantitative local stable-manifold co-core]\label{lem:D-stable-manifold}
There are constants $c_0,C_0,\eta_0>0$, depending only on $\xi$, $\bar b$, $L$, and $\kappa$, such that the following holds on the preceding event.  Let $z$ be any index-one critical point with $H_N(z)\le-N\bar b$.  In geodesic normal coordinates $(s,y)\in\R\times\R^{N-2}$ aligned with the negative and positive eigenspaces of $Q_z$, the local stable manifold of the vector field $-\grad H_N$ is a graph
\begin{equation}\label{eq:D-stable-graph}
 W^s_{\mathrm{loc}}(z)
 =\{(\psi_z(y),y):\|y\|\le c_0\sqrt N\},
\end{equation}
where $\psi_z$ is $C^2$ with
\begin{equation}\label{eq:D-psi-bounds}
\begin{gathered}
 \psi_z(0)=0,
 \qquad D\psi_z(0)=0,
 \qquad
 |\psi_z(y)|\le \frac{C_0}{\sqrt N}\|y\|^2,\\
 \|D\psi_z(y)\|\le \frac{C_0}{\sqrt N}\|y\|,
 \qquad
 \|D^2\psi_z(y)\|\le \frac{C_0}{\sqrt N}.
\end{gathered}
\end{equation}
Moreover, the restriction $h_z(y)=H_N(\psi_z(y),y)$ is strongly convex:
\begin{equation}\label{eq:D-restricted-hessian}
 D^2h_z(y)\ge\frac\kappa2\,I,
 \qquad \|y\|\le c_0\sqrt N,
\end{equation}
and consequently
\begin{equation}\label{eq:D-stable-energy}
 H_N(\psi_z(y),y)
 \ge H_N(z)+\frac\kappa4\|y\|^2,
 \qquad \|y\|\le c_0\sqrt N;
\end{equation}
one may take
\begin{equation}\label{eq:D-eta0}
 \eta_0\defeq\frac{\kappa c_0^2}{16}.
\end{equation}
The constants are uniform in the unstable eigenvalue $\alpha_z=|\lambda_-(Q_z)|$.
\end{lemma}

\begin{proof}
Write $x=(s,y)$, $\alpha=\alpha_z$, $B=B_z$, and $V=-\grad H_N$ read in the normal chart, so that $V(0)=0$ and $DV(0)=-Q_z=\operatorname{diag}(\alpha,-B)=:A$, with $\alpha\in(0,L]$ and $B\ge\kappa I$ by \eqref{eq:D-stable-block} and the derivative event.

\emph{Step 1: rescaled vector field.}  Set $\hat x=x/\sqrt N$ and
$\hat V(\hat x)=N^{-1/2}V(\sqrt N\hat x)=A\hat x+\hat F(\hat x)$
on the ball $\hat B=\{\|\hat x\|\le2c_0\}$, $c_0\le1$.  In the geodesic chart of the radius-$\sqrt N$ sphere at scale $\|x\|\le2c_0\sqrt N$, the metric satisfies $g^{ij}=\delta_{ij}+O(c_0^2)$ with $\partial g=O(c_0/\sqrt N)$ and $\partial^2g=O(1/N)$, as in \Cref{app:geometry}, while the derivative event gives $\|\partial H_N(x)\|\le2L\|x\|$ near $z$, $\|\partial^2H_N\|\le2L$, and $\|\partial^3H_N\|\le2L/\sqrt N$ in the chart.  Hence
\begin{align*}
 \|D^2V(x)\|
 &\le\|\partial^3H_N\|\,O(1)
 +\|\partial^2H_N\|\,O(c_0/\sqrt N)
 +\|\partial H_N(x)\|\,O(1/N)\\
 &\le\frac{C_V}{\sqrt N},
 \qquad \|x\|\le2c_0\sqrt N,
\end{align*}
with $C_V=C_V(\xi,L)$.  Since $D^2\hat V(\hat x)=\sqrt N\,D^2V(\sqrt N\hat x)$ and $D\hat F(\hat x)=DV(\sqrt N\hat x)-DV(0)$, the mean value theorem gives, on $\hat B$,
\begin{equation}\label{eq:D-Fhat-bounds}
 \hat F(0)=0,\qquad D\hat F(0)=0,\qquad
 \|D\hat F(\hat x)\|\le C_V\|\hat x\|\le\delta\defeq 2C_Vc_0,\qquad
 \|D^2\hat F(\hat x)\|\le C_V .
\end{equation}
No fourth derivative of $H_N$ is used.

\emph{Step 2: Lyapunov--Perron fixed point.}  Fix the weight $\mu=\kappa/3$ and, for a stable datum $\xi\in\R^{N-2}$ with $\|\xi\|\le c_0$, consider on the Banach space
\[
 X=\Bigl\{x\in C([0,\infty);\R^{N-1}):\ \|x\|_\mu\defeq \sup_{t\ge0}e^{\mu t}\|x(t)\|<\infty\Bigr\}
\]
the map
\[
 (T_\xi x)(t)
 =\Bigl(
 -\int_t^\infty e^{-\alpha(\tau-t)}\hat F_s(x(\tau))\dd\tau,\;
 e^{-Bt}\xi+\int_0^te^{-B(t-\tau)}\hat F_y(x(\tau))\dd\tau
 \Bigr).
\]
Using $\|e^{-Bt}\|\le e^{-\kappa t}$ and $\alpha\ge0$, the two kernel estimates
\[
 \int_t^\infty e^{-\alpha(\tau-t)}e^{-\mu\tau}\dd\tau
 \le\frac{e^{-\mu t}}{\alpha+\mu}\le\frac3\kappa e^{-\mu t},
 \qquad
 \int_0^te^{-\kappa(t-\tau)}e^{-\mu\tau}\dd\tau
 \le\frac{e^{-\mu t}}{\kappa-\mu}=\frac3{2\kappa}e^{-\mu t},
\]
together with $\|\hat F(x)-\hat F(x')\|\le\delta\|x-x'\|$ on $\hat B$, show that $T_\xi$ is a contraction with constant $\tfrac{9\delta}{2\kappa}$ on the closed ball $\{\|x\|_\mu\le2\|\xi\|\}$ of $X$, and maps it into itself because $\|T_\xi(0)\|_\mu\le\|\xi\|$ and $\|\hat F(x)\|\le\delta\|x\|$.  Require
\begin{equation}\label{eq:D-c0-first}
 \delta=2C_Vc_0\le\frac\kappa9,
\end{equation}
so that the contraction constant is at most $\tfrac12$.  Let $x(\cdot\,;\xi)$ be the fixed point; it solves $\dot x=Ax+\hat F(x)$ with $y(0)=\xi$, stays in $\hat B$, and decays at rate $\mu$.  Define $\hat\psi(\xi)=x_s(0;\xi)$.  Every solution starting on the resulting graph remains on it by uniqueness in $X$, so the graph is invariant.

We record two identification statements for later use; both are allowed to use $\alpha$-dependent constants, while the quantitative bounds of the lemma are not.  We emphasize that the graph is \emph{not} claimed to coincide with the set of all bounded forward-trapped orbits: when $\alpha$ is small, a forward-trapped orbit need not lie on the stable graph (a second equilibrium at distance of order $\alpha$ from the origin already provides a counterexample), and the contraction above gives uniqueness only in the exponentially weighted space $X$.

\emph{(2a) Trapped orbits with exponential decay.}  The contraction estimate for $T_\xi$ used only that candidate trajectories take values in $\hat B$, so $T_\xi$ has at most one fixed point in the closed subset $\{x\in X:\ x(t)\in\hat B\ \text{for all }t\ge0\}$ of $X$.  Consequently, if a solution $x$ remains in $\hat B$ for all $t\ge0$ and satisfies $\sup_{t\ge0}e^{\mu t}\|x(t)\|<\infty$, then it lies on the graph.  Indeed, for $T'\ge t$, variation of constants gives
\[
 x_s(t)=e^{-\alpha(T'-t)}x_s(T')
 -\int_t^{T'}e^{-\alpha(\tau-t)}\hat F_s(x(\tau))\dd\tau;
\]
since $\alpha>0$ and $x_s(T')\to0$, letting $T'\to\infty$ yields the $s$-component equation of $T_\xi$ with $\xi=x_y(0)$, the $y$-component equation is the forward variation-of-constants formula, and uniqueness identifies $x=x(\cdot\,;\xi)$, so $x(0)=(\hat\psi(\xi),\xi)$.

\emph{(2b) Convergent orbits.}  Every solution that remains in $\hat B$ for all $t\ge0$ and converges to $0$ as $t\to\infty$ lies on the graph.  Set
\[
 \rho_\alpha
 =\min\Bigl\{c_0,\ \bigl[2C_V(\alpha^{-1}+\kappa^{-1})\bigr]^{-1}\Bigr\}>0,
\]
so that, by \eqref{eq:D-Fhat-bounds}, the Lipschitz constant of $\hat F$ on $\{\|\hat x\|\le\rho_\alpha\}$ is at most $C_V\rho_\alpha$ and $C_V\rho_\alpha(\alpha^{-1}+\kappa^{-1})\le\tfrac12$.  Choose $T$ with $\sup_{\tau\ge T}\|x(\tau)\|\le\rho_\alpha/2$ and put $\xi_T=x_y(T)$.  On the complete metric space
\[
 \cC_T=\bigl\{w\in C_b([0,\infty);\R^{N-1}):\ \|w(t)\|\le\rho_\alpha\ \text{for all }t\bigr\}
\]
with the supremum norm, the elementary kernel bounds
$\int_t^\infty e^{-\alpha(\tau-t)}\dd\tau=\alpha^{-1}$ and
$\int_0^te^{-\kappa(t-\tau)}\dd\tau\le\kappa^{-1}$
show that $T_{\xi_T}$ maps $\cC_T$ into itself, because
\[
 \|(T_{\xi_T}w)(t)\|
 \le\|\xi_T\|+C_V\rho_\alpha(\alpha^{-1}+\kappa^{-1})\,\rho_\alpha
 \le\frac{\rho_\alpha}2+\frac{\rho_\alpha}2,
\]
and is a contraction with constant at most $\tfrac12$.  The shifted solution $x(T+\cdot)$ belongs to $\cC_T$ and is a fixed point of $T_{\xi_T}$: the $s$-component equation follows exactly as in (2a), the boundary term vanishing because $x_s(T')\to0$, and the improper integral converges absolutely since $\alpha>0$ and the integrand is bounded.  The Lyapunov--Perron solution $x(\cdot\,;\xi_T)$ also belongs to $\cC_T$, since $\|x(t;\xi_T)\|\le2\|\xi_T\|e^{-\mu t}\le\rho_\alpha$.  Uniqueness of the fixed point gives $x(T+\cdot)=x(\cdot\,;\xi_T)$; in particular $x(T)=(\hat\psi(\xi_T),\xi_T)$ lies on the graph and $\|x(T+t)\|\le2\|\xi_T\|e^{-\mu t}$.  The full trajectory therefore satisfies $\sup_{t\ge0}e^{\mu t}\|x(t)\|\le\max\{e^{\mu T}\sup_{[0,T]}\|x\|,\ 2\|\xi_T\|e^{\mu T}\}<\infty$, so (2a) applies to $x$ itself and yields $x(0)=(\hat\psi(x_y(0)),x_y(0))$ on the graph, as claimed.

\emph{Step 3: first and second derivatives.}  By the fiber-contraction theorem \citep[Chapter~5]{hirsch1977invariant}, applied twice, the fixed point is $C^2$ in $\xi$ and its derivatives are the fixed points of the formally differentiated equations.  The first derivative $\Phi=D_\xi x$ satisfies the linear fixed-point equation obtained by replacing $\hat F(x)$ with $D\hat F(x)\Phi$ and $e^{-Bt}\xi$ with $e^{-Bt}$; the kernel estimates above give
$\|\Phi\|_\mu\le1+\tfrac{9\delta}{2\kappa}\|\Phi\|_\mu$, hence $\|\Phi\|_\mu\le2$.  The second derivative $\Psi=D^2_\xi x$ satisfies the same linear equation with the inhomogeneity built from $D^2\hat F(x)(\Phi,\Phi)$, whose norm at time $\tau$ is at most $4C_Ve^{-2\mu\tau}$ by \eqref{eq:D-Fhat-bounds}.  Since $2\mu=\tfrac{2\kappa}3<\kappa$, the analogous kernel estimates with weight $2\mu$,
\[
 \int_t^\infty e^{-\alpha(\tau-t)}e^{-2\mu\tau}\dd\tau
 \le\frac3{2\kappa}e^{-2\mu t},
 \qquad
 \int_0^te^{-\kappa(t-\tau)}e^{-2\mu\tau}\dd\tau
 \le\frac3{\kappa}e^{-2\mu t},
\]
give
$\|\Psi\|_{2\mu}\le\tfrac9{2\kappa}\bigl(4C_V+\delta\|\Psi\|_{2\mu}\bigr)$,
hence, by \eqref{eq:D-c0-first},
\begin{equation}\label{eq:D-Psi-bound}
 \|D^2\hat\psi(\xi)\|
 =\|\Psi_s(0)\|
 \le\|\Psi\|_{2\mu}
 \le\frac{36C_V}{\kappa}=:\hat C .
\end{equation}
At $\xi=0$ one has $x(\cdot\,;0)\equiv0$ and $\Phi(t)=(0,e^{-Bt})$, so
$\hat\psi(0)=0$ and $D\hat\psi(0)=0$; Taylor's theorem and \eqref{eq:D-Psi-bound} then give
$\|D\hat\psi(\xi)\|\le\hat C\|\xi\|$ and $|\hat\psi(\xi)|\le\tfrac{\hat C}2\|\xi\|^2$.  All kernel estimates used only $\alpha\ge0$, which is the uniformity in $\alpha_z$.

\emph{Step 4: descaling.}  Set $\psi_z(y)=\sqrt N\,\hat\psi(y/\sqrt N)$ for $\|y\|\le c_0\sqrt N$.  Then
$D\psi_z(y)=D\hat\psi(\hat y)$ and $D^2\psi_z(y)=N^{-1/2}D^2\hat\psi(\hat y)$, so \eqref{eq:D-psi-bounds} holds with $C_0=\hat C$; and \eqref{eq:D-stable-graph} holds because the graph of $\hat\psi$ is the local stable manifold in the rescaled chart.  Smoothness of $H_N$ gives $\psi_z\in C^\infty$ by the classical stable-manifold theorem, though only $C^2$ is used.

\emph{Step 5: restricted Hessian.}  Write $h_z(y)=H(\psi_z(y),y)$ with $H=H_N$ in the chart.  Then
\[
 D^2h_z
 =\partial^2_{yy}H
 +\partial^2_{sy}H\otimes D\psi_z
 +D\psi_z\otimes\partial^2_{ys}H
 +(\partial^2_{ss}H)\,D\psi_z\otimes D\psi_z
 +(\partial_sH)\,D^2\psi_z ,
\]
evaluated at $(\psi_z(y),y)$.  On the ball $\|y\|\le c_0\sqrt N$: the chart Hessian differs from $Q_z$ by at most $C_2c_0$ in operator norm (third-derivative and metric errors), so
$\partial^2_{yy}H\ge(\kappa-C_2c_0)I$ and, because the coordinates diagonalize $Q_z$, $\|\partial^2_{sy}H\|\le C_2c_0$; also $\|D\psi_z\|\le C_0c_0$ and $|\partial^2_{ss}H|\le2L$.  For the last term --- the one that requires the second-derivative bound --- the gradient satisfies $\|\partial H(x)\|\le2L\|x\|\le C_3c_0\sqrt N$ on the ball, so
\[
 \|(\partial_sH)\,D^2\psi_z\|
 \le C_3c_0\sqrt N\cdot\frac{C_0}{\sqrt N}
 =C_0C_3\,c_0 .
\]
Collecting terms, $D^2h_z\ge(\kappa-C_4c_0)I$ with $C_4=C_4(\xi,L,\kappa)$; decreasing $c_0$ so that $C_4c_0\le\kappa/2$, in addition to \eqref{eq:D-c0-first}, proves \eqref{eq:D-restricted-hessian}.  Since $z$ is a critical point, $Dh_z(0)=\partial_yH(0)+\partial_sH(0)D\psi_z(0)=0$, so \eqref{eq:D-restricted-hessian} and Taylor's theorem give
$h_z(y)\ge h_z(0)+\tfrac\kappa4\|y\|^2$, which is \eqref{eq:D-stable-energy}.  Equation \eqref{eq:D-eta0} is then a convenient fixed energy margin.
\end{proof}

\begin{remark}[Consequences for the co-core geometry]\label{rem:D-convexity-consequences}
Two facts used in \Cref{lem:D-in-chart-cocore} follow at once from \eqref{eq:D-restricted-hessian}.  First, $h_z$ is strongly convex on the coordinate ball, so its only critical point is $y=0$ and every value in $(h_z(0),\,h_z(0)+\eta_0N]$ is a regular value of $h_z$; equivalently, since $-\grad H_N$ is tangent to the invariant graph, no critical point of $H_N$ other than $z$ lies on $W^s_{\mathrm{loc}}(z)$.  Second, each sublevel set $\{h_z\le\lambda\}$ with $\lambda\le h_z(0)+\eta_0N$ is a compact convex set with smooth boundary, contained in the open ball $\{\|y\|<c_0\sqrt N\}$ by \eqref{eq:D-stable-energy} and \eqref{eq:D-eta0}, and is therefore a smoothly embedded $(N-2)$-disk.
\end{remark}

Choose
\begin{equation}\label{eq:D-bstar}
 b_*(\xi)>
 \max\left\{\bar b,E_1(\xi)-\frac{\eta_0}{4}\right\},
 \qquad b_*(\xi)<E_1(\xi).
\end{equation}
Fix $b\in(b_*(\xi),E_1(\xi))$, $a>E_1(\xi)$, a sufficiently large fixed $\beta$, and $\delta>0$.  Choose
\begin{equation}\label{eq:D-delta-beta}
 0<\delta_\beta<
 \min\left\{
 \delta,\frac{a-E_1(\xi)}4,\frac{\eta_0}4,\beta^{-2}
 \right\}.
\end{equation}
The threshold exclusion of \Cref{prop:layering} (via the annealed upper bound \Cref{lem:B-index-count}) implies that, with probability tending to one, there is no index-one critical point below $-N(E_1(\xi)+\delta_\beta)$.  Let $\cC_N$ be the set of \emph{all} index-one critical points of $H_N$ below $-Nb$; no minimal cut selection is needed, because the energy comparisons below only use that $\cC_N$ is contained in the set of all index-one points in the energy window.  A terminological caution: $\cC_N$ together with its stable disks is a \emph{separating co-core family}, not an inclusion-minimal edge cut in the graph sense of \Cref{lem:topological-cocore-cut}; throughout this appendix the word ``cut'' is used only in the separating sense, and no minimality is claimed or used.  Every $z\in\cC_N$ then satisfies
\begin{equation}\label{eq:D-saddle-window}
 -E_1(\xi)-\delta_\beta
 \le \frac{H_N(z)}N\le-b
\end{equation}
and hence
\begin{equation}\label{eq:D-height-to-b}
 -Nb-H_N(z)
 \le N\bigl(E_1(\xi)+\delta_\beta-b\bigr)
 <\frac{\eta_0N}{2}.
\end{equation}

In what follows we also assume, decreasing $c_0$ in \Cref{lem:D-stable-manifold} once and for all, that
\begin{equation}\label{eq:D-c0-extra}
 C_0c_0\le\tfrac18
 \qquad\text{and}\qquad
 \eta_0=\frac{\kappa c_0^2}{16}\le\frac{E_0(\xi)-E_1(\xi)}8 .
\end{equation}

\begin{lemma}[Stable-disk separation and slab crossing]\label{lem:D-in-chart-cocore}
Fix $K\ge1$.  For $b>b_*(\xi)$ there is $\beta_1=\beta_1(\xi,b,K)$ such that for all fixed $\beta\ge\beta_1$ the following holds on the intersection of the almost-sure Morse event with the events of \Cref{prop:layering,prop:deep-components,prop:stable-gap-layer,prop:derivative-scale} and the no-low-saddle event above.  For $z\in\cC_N$ define the truncated stable disk
\begin{equation}\label{eq:D-cocore-disk}
 D_z\defeq W^s_{\mathrm{loc}}(z)\cap\{H_N\le-Nb\}
 =\{(\psi_z(y),y):\ h_z(y)\le-Nb\},
\end{equation}
the slab domain $Y_z=\{y:\ h_z(y)\le-Nb+\eta_0N/4\}$, and the slab
\begin{equation}\label{eq:D-slab-map}
 \Sigma_z\defeq\Phi_z\bigl([-r_N,r_N]\times Y_z\bigr),
 \qquad
 \Phi_z(t,y)\defeq(\psi_z(y)+t,y),
 \qquad r_N\defeq K\sqrt{\frac N\beta},
\end{equation}
in the normal chart of \Cref{lem:D-stable-manifold}.  Then:
\begin{enumerate}[label=\textup{(\roman*)},leftmargin=*,itemsep=0.3em]
 \item \textup{(Confinement.)}  $D_z$ is a smoothly embedded $(N-2)$-disk with boundary in $\{H_N=-Nb\}$, and it equals the set of points $x$ with $H_N(x)\le-Nb$ whose forward orbit under $-\grad H_N$ converges to $z$.  In particular, the disks $(D_z)_{z\in\cC_N}$ are pairwise disjoint.
 \item \textup{(Slab structure.)}  $\Phi_z$ is a diffeomorphism onto $\Sigma_z$ with $\det D\Phi_z=1$ and $\|D\Phi_z^{\pm1}\|\le2$; the lateral boundary $\Phi_z([-r_N,r_N]\times\partial Y_z)$ lies in $\{H_N\ge-Nb+\eta_0N/8\}$; and $\Sigma_z\cap(\cA_N(a)\cup\cB_N(a))=\varnothing$.
 \item \textup{(Separation and crossing.)}  Every continuous path in $\Omega_N(b)$ from $\cA_N(a)$ to $\cB_N(a)$ meets $D_z$ for some $z\in\cC_N$; around any such meeting it contains a subpath of $\Sigma_z$ whose two endpoints lie on the transverse faces $\{t=\pm r_N\}$ (possibly the same face twice) and which meets the central sheet $\{t=0\}$ in between.
\end{enumerate}
\end{lemma}

\begin{proof}
\emph{(i): confinement.}  If $h_z(y)\le-Nb+\eta_0N/4$, then \eqref{eq:D-stable-energy} and \eqref{eq:D-height-to-b} give $\kappa\|y\|^2/4\le-Nb-H_N(z)+\eta_0N/4<3\eta_0N/4$, so $\|y\|^2<3c_0^2N/16$: the domains of $D_z$ and $Y_z$ lie in $\{\|y\|<c_0\sqrt N/2\}$, well inside the coordinate ball.  By \Cref{rem:D-convexity-consequences}, $D_z$ is a smoothly embedded disk with boundary in the level $-Nb$, and every point of $D_z$ converges to $z$ under the forward flow.  Conversely, let $x$ satisfy $H_N(x)\le-Nb$ with forward orbit converging to $z$, write $\bar B_z$ for the closed chart ball of radius $\tfrac78c_0\sqrt N$ around $z$, and let
\[
 s_*=\inf\bigl\{t\ge0:\ \text{the orbit of }x\text{ from time }t\text{ onward stays in }\bar B_z\bigr\},
\]
which is finite by convergence.  The orbit from time $s_*$ onward remains in that ball and converges to $z$, hence lies on the graph $W^s_{\mathrm{loc}}(z)$ by the convergent-orbit identification (2b) in Step~2 of the proof of \Cref{lem:D-stable-manifold}; forward trapping alone is not invoked, since it does not characterize the stable graph when $\alpha_z$ is small.  If $s_*>0$, the orbit point at time $s_*$ lies on the sphere of radius $\tfrac78c_0\sqrt N$ and on the graph; by \eqref{eq:D-c0-extra} its $y$-component satisfies $\|y\|\ge\tfrac89\cdot\tfrac78c_0\sqrt N\ge\tfrac34c_0\sqrt N$, so \eqref{eq:D-stable-energy} bounds $H_N$ at that point from below:
\[
 H_N
 \ \ge\ H_N(z)+\frac\kappa4\cdot\frac9{16}c_0^2N
 \ >\ H_N(z)+2\eta_0N
 \ >\ -Nb
\]
by \eqref{eq:D-height-to-b}; but $H_N$ is nonincreasing along the orbit and starts at $H_N(x)\le-Nb$, a contradiction.  Hence $s_*=0$: the point $x$ itself lies on the graph, and $h_z$ at its $y$-component equals $H_N(x)\le-Nb$, so $x\in D_z$.  Pairwise disjointness of the disks is immediate from this characterization: a common point of $D_z$ and $D_{z'}$ would have forward orbit converging to both $z$ and $z'$, and uniqueness of the gradient-flow limit forces $z=z'$.

\emph{(ii): slab structure.}  $\Phi_z$ has the explicit smooth inverse $(s,y)\mapsto(s-\psi_z(y),y)$, is triangular with unit determinant, and $\|D\Phi_z^{\pm1}\|\le1+\|D\psi_z\|\le1+C_0c_0\le2$ by \eqref{eq:D-psi-bounds} and \eqref{eq:D-c0-extra}.  On the slab, the gradient bound of Step~1 of the proof of \Cref{lem:D-stable-manifold} gives $\|\partial H_N\|\le2L(\|y\|+|\psi_z(y)|+r_N)\le2Lc_0\sqrt N$, so moving from the graph a transverse distance at most $r_N$ changes $H_N$ by at most
\begin{equation}\label{eq:D-transverse-change}
 r_N\cdot2Lc_0\sqrt N=2LKc_0\frac N{\sqrt\beta}
 \le\frac{\eta_0N}8
 \qquad(\beta\ge\beta_1).
\end{equation}
On the lateral boundary, $h_z=-Nb+\eta_0N/4$, so $H_N\ge-Nb+\eta_0N/8$ there.  On all of $\Sigma_z$,
\[
 H_N
 \ \ge\ H_N(z)-\frac{\eta_0N}8
 \ \ge\ -N(E_1(\xi)+\delta_\beta)-\frac{\eta_0N}8
 \ >\ -Na,
\]
because $a-E_1(\xi)\ge(E_0(\xi)-E_1(\xi))/4\ge2\eta_0$ by \eqref{eq:D-c0-extra} and $\delta_\beta$ is small; hence $\Sigma_z$ avoids the deep sets, with a threshold $\beta_1$ depending only on $\xi$, $b$, and $K$.

\emph{(iii): separation.}  Write $M=\overline\Omega_N(b)=\{H_N\le-Nb\}$.  On the stated events, the critical points of $H_N$ in $M$ are local minima and the elements of $\cC_N$.  Forward orbits of $-\grad H_N$ remain in $M$, and each converges to a single critical point: the $\omega$-limit set of a gradient orbit of a Morse function on a compact manifold is a connected set of critical points, and critical points are isolated.  For each minimum $m$, the basin $U_m=\{x\in M:\ \text{orbit of }x\text{ converges to }m\}$ is open in $M$, by the standard Lyapunov argument at an asymptotically stable equilibrium: $U_m$ is the set of points whose orbit enters a fixed attracting sublevel neighborhood of $m$, an open condition.  If $x\in M$ belongs to no $U_m$, its orbit converges to some $z\in\cC_N$, and part (i) gives $x\in D_z$.  Hence
\[
 M\setminus\bigcup_{z\in\cC_N}D_z
 =\bigsqcup_{m}\;U_m\setminus\bigcup_{z\in\cC_N}D_z
\]
is a disjoint union of open subsets of $M$, so every connected subset of $M\setminus\bigcup_zD_z$ lies in a single basin.  Let $\gamma$ be a path in $\Omega_N(b)\subset M$ from $\cA_N(a)$ to $\cB_N(a)$.  The orbit of $\gamma(0)$ stays in the deep component of $\cA_N(a)$ containing it ($H_N$ decreases and orbits are connected) and therefore converges to the unique critical point of that component, a minimum $m_+$ with $s_{N,a}(m_+)=1$ (\Cref{prop:deep-components}); similarly $\gamma(1)$ converges to a minimum $m_-$ with $s_{N,a}(m_-)=-1$, and $m_+\ne m_-$.  If $\gamma$ avoided every $D_z$ it would lie in one basin, forcing $m_+=m_-$; so $\gamma$ meets $D_z$ for some $z\in\cC_N$, necessarily in the relative interior of $D_z$ since $H_N<-Nb$ on $\gamma$.

\emph{Crossing.}  The meeting point lies in the interior of $\Sigma_z$ (it has $t=0$, and its $y$-component lies in the interior of $Y_z$).  Take the maximal subinterval of the parameter containing the meeting time on which $\gamma$ stays in $\Sigma_z$.  Its endpoints lie on $\partial\Sigma_z$: they cannot be endpoints of $\gamma$ (the deep sets avoid $\Sigma_z$ by (ii)), and they cannot lie on the lateral boundary (there $H_N\ge-Nb+\eta_0N/8$, while $H_N<-Nb$ on $\gamma$).  Hence both endpoints lie on the transverse faces $\{t=\pm r_N\}$, and the subpath meets $\{t=0\}$ at the meeting point.
\end{proof}

\subsection{Finite-width profiles and the good-saddle estimate}

Write $d=N-1$ and decompose the saddle Hessian as
\begin{equation}\label{eq:D-Hessian-split}
 Q_z=\begin{pmatrix}-\alpha_z&0\\0&B_z\end{pmatrix},
 \qquad B_z\ge\kappa I.
\end{equation}
After straightening the stable graph, the following holds on a Gaussian core $\{\|y\|\le R\sqrt{N/\beta},\ |t|\le r_N\}$.  Write $x=\Phi_z(t,y)=(t+\psi_z(y),y)$ in the normal chart at $z$.  By \eqref{eq:D-psi-bounds}, $|\psi_z(y)|\le C_0\|y\|^2/\sqrt N\le C_0R^2\sqrt N/\beta$, so $\|x\|\le(K+R+1)\sqrt{N/\beta}$ for all $\beta$ above a fixed threshold.  The critical-point Taylor estimate \eqref{eq:C-Taylor-error}, applied at radius $(K+R+1)\sqrt{N/\beta}$, gives
\[
 H_N(x)
 \ge H_N(z)+\tfrac12\langle x,Q_zx\rangle
 -C_{\xi,R,K}\frac N{\beta^{3/2}},
\]
and, in the coordinates diagonalizing $Q_z$,
\[
 \tfrac12\langle x,Q_zx\rangle
 =-\frac{\alpha_z}2\bigl(t+\psi_z(y)\bigr)^2
 +\frac12\langle y,B_zy\rangle
 =-\frac{\alpha_z}2t^2+\frac12\langle y,B_zy\rangle
 -\alpha_zt\,\psi_z(y)-\frac{\alpha_z}2\psi_z(y)^2 .
\]
The two cross terms are the straightening cost, and each is $O(N/\beta^{3/2})$ explicitly: using $\alpha_z\le2L$ on the derivative event, $|t|\le r_N=K\sqrt{N/\beta}$, and the bound on $\psi_z$ above,
\begin{gather*}
 |\alpha_zt\,\psi_z(y)|
 \le2L\cdot K\sqrt{\tfrac N\beta}\cdot C_0R^2\frac{\sqrt N}\beta
 =2LKC_0R^2\,\frac N{\beta^{3/2}},\\
 \frac{\alpha_z}2\psi_z(y)^2
 \le LC_0^2R^4\,\frac N{\beta^2}
 \le LC_0^2R^4\,\frac N{\beta^{3/2}} .
\end{gather*}
Collecting the three error terms yields
\begin{equation}\label{eq:D-straightened-Taylor}
 H_N(\Phi_z(t,y))
 \ge H_N(z)-\frac{\alpha_z}{2}t^2
 +\frac12\langle y,B_zy\rangle
 -C_{\xi,b,R}\frac{N}{\beta^{3/2}},
\end{equation}
and the volume and inverse-metric distortions contribute at most
$\exp\{C_{\xi,b,R}N/\beta\}$ by \eqref{eq:C-geometry-error}.  Both estimates are uniform over the cut saddles, because every constant above comes from \Cref{lem:D-stable-manifold} and the deterministic derivative event \eqref{eq:derivative-event}.

Fix once and for all a smooth even function $\zeta:\R\to[0,1]$ with $\zeta=1$ on $[-\tfrac12,\tfrac12]$ and $\operatorname{supp}\zeta\subset(-1,1)$, and put
\begin{equation}\label{eq:D-Itrunc}
 J_{N,z}\defeq\int_{-r_N}^{r_N}\zeta\Bigl(\frac v{r_N}\Bigr)e^{-\beta\alpha_zv^2/2}\dd v,
 \qquad
 \varphi_{N,z}(t)\defeq
 \frac1{J_{N,z}}\,\zeta\Bigl(\frac t{r_N}\Bigr)e^{-\beta\alpha_zt^2/2},
\end{equation}
a smooth, even, probability density on $[-r_N,r_N]$ vanishing to all orders at the endpoints.  Let $\tau_z(x)=s-\psi_z(y)$ be the transverse coordinate on the slab $\Sigma_z$, and let $\chi_z:\R^{N-2}\to[0,1]$ be a smooth cutoff with $\chi_z=1$ on $\{h_z\le-Nb+\eta_0N/8\}$ and $\operatorname{supp}\chi_z\subset\{h_z<-Nb+\eta_0N/4\}$ (for instance a fixed smooth function of $(h_z(y)+Nb)/(\eta_0N)$).  Define
\begin{equation}\label{eq:D-rho-def}
 \rho_{N,z}
 \defeq\chi_z(y)\,\varphi_{N,z}(\tau_z)\,|\grad\tau_z|
 \quad\text{on }\Sigma_z,
 \qquad
 \rho_{N,z}=0\quad\text{off }\Sigma_z,
\end{equation}
which is $C^\infty$ on $\SphereN$, not merely continuous.  Indeed, $\psi_z$ --- hence $\tau_z$ and the nowhere-vanishing field $\grad\tau_z$ --- is $C^\infty$ (Step~4 of the proof of \Cref{lem:D-stable-manifold}), so $\rho_{N,z}$ is smooth on the open slab; and since $\operatorname{supp}\zeta\subset(-1,1)$ and $\operatorname{supp}\chi_z\subset\{h_z<-Nb+\eta_0N/4\}$, the product $\chi_z(y)\,\varphi_{N,z}(\tau_z)$ vanishes on a neighborhood of the topological boundary of $\Sigma_z$ ($\varphi_{N,z}$ near the transverse faces, $\chi_z$ near the lateral boundary), so the extension by zero is $C^\infty$.  This supplies the smoothness of the charging densities asserted in \Cref{thm:cocore-localization}.

For one saddle, denote its normalized Eyring--Kramers contribution by
\begin{equation}\label{eq:D-kz}
 k_{N,\beta}(z)
 \defeq\frac1{\vol(\SphereN)}
 \left(\frac{2\pi}{\beta}\right)^{(N-3)/2}
 e^{-\beta H_N(z)}
 \frac{\sqrt{\alpha_z}}{\sqrt{\det B_z}}.
\end{equation}
Thus $\cK_{N,\beta}(I)$ is the sum of \eqref{eq:D-kz} over the saddles in $I$.

\begin{lemma}[Charging and truncated profile energy]\label{lem:D-profile-energy}
There are $C_1=C_1(\xi,b,K)<\infty$ and $\beta_2=\beta_2(\xi,b,K)$ such that, for every fixed $\beta\ge\beta_2$, on the event of \Cref{lem:D-in-chart-cocore} the densities \eqref{eq:D-rho-def} satisfy, deterministically and simultaneously for every $z\in\cC_N$ and all large $N$:
\begin{enumerate}[label=\textup{(\roman*)},leftmargin=*,itemsep=0.3em]
 \item \textup{(Half-mass charging.)}  Every rectifiable subpath of $\Sigma_z\cap\{H_N<-Nb\}$ joining a transverse face of the slab to the central sheet $\{\tau_z=0\}$ satisfies $\int\rho_{N,z}\,\dd s\ge\tfrac12$.  Consequently every crossing subpath supplied by \Cref{lem:D-in-chart-cocore}\textup{(iii)} satisfies $\int\rho_{N,z}\,\dd s\ge1$.
 \item \textup{(Energy.)}
\begin{equation}\label{eq:D-profile-bound}
 \int_{\SphereN}\rho_{N,z}^2e^{-\beta H_N}\dd\nu_N
 \le
 C_1\exp\left\{\frac{C_1N}{\sqrt\beta}\right\}
 \max\left\{1,\frac1{\sqrt{N\alpha_z}}\right\}
 k_{N,\beta}(z).
\end{equation}
\end{enumerate}
\end{lemma}

\begin{proof}
\emph{(i).}  Let $\gamma$ be such a subpath, parametrized by arclength.  Its points satisfy $H_N<-Nb$ and $|\tau_z|\le r_N$, so by \eqref{eq:D-transverse-change} their $y$-components obey $h_z(y)\le-Nb+\eta_0N/8$; hence $\chi_z=1$ along $\gamma$.  Since $|\frac{\dd}{\dd s}\tau_z(\gamma(s))|\le|\grad\tau_z|$ pointwise,
\[
 \int_\gamma\rho_{N,z}\,\dd s
 \ge\int_\gamma\varphi_{N,z}(\tau_z)\,
 \Bigl|\frac{\dd\tau_z}{\dd s}\Bigr|\dd s
 \ge\int_0^{r_N}\varphi_{N,z}(v)\dd v
 =\frac12,
\]
because $\tau_z$ sweeps at least one of the intervals $[0,r_N]$ or $[-r_N,0]$ and $\varphi_{N,z}$ is an even probability density.  A crossing subpath as in \Cref{lem:D-in-chart-cocore}(iii) concatenates two such subpaths (face to sheet, sheet to face), giving total charge at least $1$.

\emph{(ii).}  In the slab chart, $\dd\vol=J_N(r)\dd t\,\dd y$ with $J_N\le1$ by \eqref{eq:C-Jacobian} and $\det D\Phi_z=1$, while $|\grad\tau_z|^2\le\|D\Phi_z^{-1}\|^2(1+O(c_0^2))\le8$.  Hence
\[
 \int\rho_{N,z}^2e^{-\beta H_N}\dd\nu_N
 \le\frac8{\vol(\SphereN)}
 \int_{Y_z}\int_{-r_N}^{r_N}
 \varphi_{N,z}(t)^2\,e^{-\beta H_N(\Phi_z(t,y))}\dd t\,\dd y .
\]
Fix $R=R(\xi,b)$ by the condition
\begin{equation}\label{eq:D-R-choice}
 \frac{\kappa R^2}{32}\ge1+\frac12\log\frac{8L}\kappa,
\end{equation}
and split the $y$-integration at $\|y\|=R\sqrt{N/\beta}$.

\emph{Core.}  For $\|y\|\le R\sqrt{N/\beta}$, the straightened Taylor estimate \eqref{eq:D-straightened-Taylor} gives
$\beta H_N(\Phi_z(t,y))\ge\beta H_N(z)-\beta\alpha_zt^2/2+\beta\langle y,B_zy\rangle/2-C_{\xi,b,R}N/\sqrt\beta$, so the core contribution is at most
\[
 e^{C_{\xi,b,R}N/\sqrt\beta}\,
 \frac{e^{-\beta H_N(z)}}{\vol(\SphereN)}
 \Bigl[\int_{-r_N}^{r_N}\varphi_{N,z}(t)^2e^{\beta\alpha_zt^2/2}\dd t\Bigr]
 \Bigl[\int_{\R^{N-2}}e^{-\beta\langle y,B_zy\rangle/2}\dd y\Bigr],
\]
where the $y$-integral was extended to $\R^{N-2}$ as an upper bound and equals $(2\pi/\beta)^{(N-2)/2}(\det B_z)^{-1/2}$.  For the $t$-integral, $\varphi_{N,z}(t)^2e^{\beta\alpha_zt^2/2}\le J_{N,z}^{-2}e^{-\beta\alpha_zt^2/2}$, and since the integrand of $J_{N,z}$ is even and decreasing in $|v|$ while $\zeta=1$ on the inner half,
\[
 J_{N,z}\ge\int_{|v|\le r_N/2}e^{-\beta\alpha_zv^2/2}\dd v
 \ge\tfrac12I_{N,z},
 \qquad
 I_{N,z}\defeq\int_{-r_N}^{r_N}e^{-\beta\alpha_zv^2/2}\dd v;
\]
hence the $t$-integral is at most $4I_{N,z}^{-1}$.

\emph{Annulus.}  For $y\in Y_z$ with $\|y\|\ge R\sqrt{N/\beta}$ and $|t|\le r_N$, integrate the $s$-derivative along the transverse segment: by the expansion of $\partial_sH_N$ at $z$ (Step~1 of the proof of \Cref{lem:D-stable-manifold}) and \eqref{eq:D-psi-bounds},
\begin{align*}
 H_N(\Phi_z(t,y))
 &\ge h_z(y)-\frac{\alpha_z}2t^2-\alpha_z|\psi_z(y)|r_N
 -\frac{4Lr_N}{\sqrt N}\bigl(\|y\|^2+r_N^2\bigr)\\
 &\ge H_N(z)+\frac\kappa8\|y\|^2-\frac{\alpha_z}2t^2
 -\frac{4LK^3N}{\beta^{3/2}},
\end{align*}
for $\beta\ge\beta_2$ chosen so that $(LC_0K+4LK)/\sqrt\beta\le\kappa/8$, using \eqref{eq:D-stable-energy}.  Hence the annulus contribution is at most
\[
 e^{4LK^3N/\sqrt\beta}\,
 \frac{e^{-\beta H_N(z)}}{\vol(\SphereN)}
 \cdot4I_{N,z}^{-1}\cdot
 \int_{\|y\|\ge R\sqrt{N/\beta}}e^{-\beta\kappa\|y\|^2/8}\dd y .
\]
The Gaussian tail in dimension $n'=N-2$ satisfies, by the standard chi-square bound and \eqref{eq:D-R-choice},
\begin{align*}
 \int_{\|y\|\ge R\sqrt{N/\beta}}e^{-\beta\kappa\|y\|^2/8}\dd y
 &\le\Bigl(\frac{8\pi}{\beta\kappa}\Bigr)^{n'/2}e^{-n'\kappa R^2/32}
 \le\Bigl(\frac{2\pi}{\beta}\Bigr)^{n'/2}(2L)^{-n'/2}e^{-n'}\\
 &\le\Bigl(\frac{2\pi}\beta\Bigr)^{n'/2}\frac{e^{-n'}}{\sqrt{\det B_z}},
\end{align*}
using $\det B_z\le\|Q_z\|_{\op}^{n'}\le(2L)^{n'}$ on the derivative event.  The annulus is therefore dominated by the core up to the factor $e^{-n'}$.

\emph{Conclusion.}  Adding the two pieces,
\begin{align}
 \int\rho_{N,z}^2e^{-\beta H_N}\dd\nu_N
 \le{}&
 C\exp\left\{\frac{C_1N}{\sqrt\beta}\right\}
 \frac{e^{-\beta H_N(z)}}{\vol(\SphereN)}
 \left(\frac{2\pi}{\beta}\right)^{(N-2)/2}
 \frac{I_{N,z}^{-1}}{\sqrt{\det B_z}}.
 \label{eq:D-profile-Gaussian}
\end{align}
The elementary one-dimensional bound
\begin{equation}\label{eq:D-I-bound}
 I_{N,z}^{-1}
 \le C\max\left\{\sqrt{\beta\alpha_z},\frac1{r_N}\right\}
\end{equation}
follows by considering separately $\beta\alpha_zr_N^2\ge1$ and $\beta\alpha_zr_N^2<1$.  Comparing \eqref{eq:D-profile-Gaussian} with \eqref{eq:D-kz}, and using
$r_N=K\sqrt{N/\beta}$, proves \eqref{eq:D-profile-bound}: the ratio of the two expressions is $(2\pi/\beta)^{1/2}I_{N,z}^{-1}/\sqrt{\alpha_z}$ up to constants, which \eqref{eq:D-I-bound} bounds by $C\max\{1,(N\alpha_z)^{-1/2}\}$.  Every constant above depends only on $\xi$, $b$, $K$, $L$, $\kappa$, and $R(\xi,b)$, which gives the stated uniformity over the cut and removes all $o(N)$ errors.
\end{proof}

Fix $A>0$ and define
\begin{equation}\label{eq:D-alpha-star}
 \alpha_*\defeq
 \exp\left\{-\frac{AN}{\sqrt\beta}\right\}.
\end{equation}
The two regimes $\alpha_z\ge\alpha_*$ (``good'') and
$\alpha_z<\alpha_*$ (``soft'') are an exhaustive dichotomy at the
single threshold $\alpha_*$: no intermediate window is left
uncovered, and the price of extending the good-saddle estimate all
the way down to $\alpha_*$ is exactly the factor
$e^{AN/(2\sqrt\beta)}$ visible in \eqref{eq:D-good-one} through
\eqref{eq:D-I-bound}.  For saddles with $\alpha_z\ge\alpha_*$, \Cref{lem:D-profile-energy} gives
\begin{equation}\label{eq:D-good-one}
 \int\rho_{N,z}^2e^{-\beta H_N}\dd\nu_N
 \le
 \exp\left\{\frac{(C_1+A/2)N}{\sqrt\beta}+o(N)\right\}
 k_{N,\beta}(z).
\end{equation}
Because the cut saddles are a subset of all index-one saddles in $I_{\xi,b,\delta}$, summation yields
\begin{equation}\label{eq:D-good-sum}
 \sum_{\substack{z\in\cC_N\\\alpha_z\ge\alpha_*}}
 \int\rho_{N,z}^2e^{-\beta H_N}\dd\nu_N
 \le
 \exp\left\{\frac{(C_1+A/2)N}{\sqrt\beta}+o(N)\right\}
 \cK_{N,\beta}(I_{\xi,b,\delta}).
\end{equation}
This proves the entire good-saddle part of \Cref{thm:cocore-localization}.

\subsection{Soft saddles and the marked Kac--Rice estimates}

For $\alpha_z<\alpha_*$, the same density \eqref{eq:D-rho-def} applies; in this regime $J_{N,z}\asymp I_{N,z}\asymp r_N$, so the profile is a smoothed flat transition across $[-r_N,r_N]$.  The calculation in \Cref{lem:D-profile-energy} gives
\begin{align}
 \int\rho_{N,z}^2e^{-\beta H_N}\dd\nu_N
 \le{}&
 \exp\left\{\frac{C_1N}{\sqrt\beta}+o(N)\right\}
 \frac1{\vol(\SphereN)}
 \left(\frac{2\pi}{\beta}\right)^{(N-2)/2}
 \frac{e^{-\beta H_N(z)}}{r_N\sqrt{\det B_z}}.
 \label{eq:D-soft-one}
\end{align}
Define the corresponding flat-profile sum on an interval $J$ by
\begin{equation}\label{eq:D-soft-flat-sum}
 \widetilde{\cK}^{\mathrm{soft}}_{N,\beta}(J;A)
 \defeq\frac1{\vol(\SphereN)}
 \left(\frac{2\pi}{\beta}\right)^{(N-2)/2}
 \frac1{r_N}
 \sum_{\substack{z:\,H_N(z)/N\in J,\ \ind Q_z=1\\
                   \alpha_z<\exp\{-AN/\sqrt\beta\}}}
 \frac{e^{-\beta H_N(z)}}{\sqrt{\det B_z}}.
\end{equation}
The Kac--Rice Jacobian for the summand in \eqref{eq:D-soft-flat-sum} satisfies
\begin{equation}\label{eq:D-soft-Jacobian}
 |\det Q_z|\frac1{\sqrt{\det B_z}}
 =\alpha_z\sqrt{\det B_z}.
\end{equation}
Thus the soft event supplies the explicit factor
$\alpha_z\le\exp\{-AN/\sqrt\beta\}$, while the remaining bulk mark is a half determinant.  The fixed-extreme-eigenvalue large-deviation principle of \citet[Theorem~A.1]{auffinger2013random}, the speed-$N^2$ empirical spectral-measure large-deviation principle of \citet{benarous1997large}, and the scalar-tilt Laplace bound \Cref{lem:B-tilt-laplace} for the random Hessian shift are combined in \Cref{lem:B-matrix-rates} to prove the following uniform estimate.

\begin{proposition}[Uniform soft-saddle marked Kac--Rice upper bound]\label{prop:soft-marked-KR}
Let $J\subset(-\infty,0)$ be compact.  For every fixed sufficiently large $\beta$ and every $A>0$,
\begin{align}
 \limsup_{N\to\infty}\frac1N
 \log\E\widetilde{\cK}^{\mathrm{soft}}_{N,\beta}(J;A)
 \le{}&
 -\frac12\log(\beta e)
 +\sup_{u\in J}
 \{-\beta u+\Xi^{\mathrm{EK}}_{1,\xi}(u)\}
 -\frac A{\sqrt\beta}.
 \label{eq:D-soft-KR-obligation}
\end{align}
The bound is uniform under the finite energy binning used in \Cref{thm:weighted-KR}.
\end{proposition}

\begin{proof}
The mark
$(\det B_N(u))^{-1/2}\mathbf1_{\{\alpha_N(u)<e^{-AN/\sqrt\beta}\}}$ is unbounded where a stable eigenvalue degenerates.  To justify the marked one-point Kac--Rice identity for it, apply the formula (exactly as in \eqref{eq:B-marked-KR}, but with this mark in place of $\alpha_N(u)|\det Q_N(u)|^{-1/2}$) to the bounded marks
$\min\{(\det B_N(u))^{-1/2},M\}\mathbf1_{\{\alpha_N(u)<e^{-AN/\sqrt\beta}\}}$
and let $M\uparrow\infty$ by monotone convergence on both sides.  The limiting right-hand side is finite because the Kac--Rice Jacobian cancels the singularity:
\[
 |\det Q_N(u)|(\det B_N(u))^{-1/2}
 =\alpha_N(u)(\det B_N(u))^{1/2}
 \le\|Q_N(u)\|_{\op}^{n/2+1},
\]
which is integrable by \Cref{lem:B-det-moments}.  Hence, by \eqref{eq:B-prefactor-rate},
\begin{align*}
 \E\widetilde{\cK}^{\mathrm{soft}}_{N,\beta}(J;A)
 ={}&\frac1{\vol(\SphereN)}
 \left(\frac{2\pi}\beta\right)^{(N-2)/2}
 \frac1{r_N}
 \int_Je^{N\varrho_N(u)}e^{-\beta Nu}\\
 &\times
 \E\Bigl[|\det Q_N(u)|(\det B_N(u))^{-1/2}
 \mathbf1_{\cI_{N,1}(u)}
 \mathbf1_{\{\alpha_N(u)<e^{-AN/\sqrt\beta}\}}\Bigr]\dd u.
\end{align*}
The Jacobian cancellation \eqref{eq:D-soft-Jacobian} and the soft-event indicator give the pointwise bound
\begin{align*}
 |\det Q_N(u)|(\det B_N(u))^{-1/2}
 \mathbf1_{\{\alpha_N(u)<e^{-AN/\sqrt\beta}\}}
 &=\alpha_N(u)(\det B_N(u))^{1/2}
 \mathbf1_{\{\alpha_N(u)<e^{-AN/\sqrt\beta}\}}\\
 &\le e^{-AN/\sqrt\beta}(\det B_N(u))^{1/2}.
\end{align*}
Fix $\eps>0$.  \Cref{lem:B-matrix-rates}(b) with $r=1/2$ (which already integrates the random shift) bounds the matrix expectation by
$\exp\{N(\Lambda^{(1/2)}_{1,\xi}(u)+\eps)-AN/\sqrt\beta\}$, uniformly for $u\in J$, and \Cref{lem:B-matrix-rates}(d) then converts the $u$-dependent exponent into
\[
 -\beta u+\varrho_\xi(u)+\Lambda^{(1/2)}_{1,\xi}(u)
 =-\beta u+\Xi^{\mathrm{EK}}_{1,\xi}(u),
\]
by \eqref{eq:Xi-EK}.
The geometric prefactor contributes
\[
 \frac1N\log\left[
 \frac1{\vol(\SphereN)}
 \left(\frac{2\pi}\beta\right)^{(N-2)/2}\frac1{r_N}
 \right]
 \longrightarrow
 \frac12\log\frac{2\pi}\beta-\frac12\log(2\pi e)
 =-\frac12\log(\beta e),
\]
by \eqref{eq:sphere-volume-asymptotic}, since $r_N=K\sqrt{N/\beta}$ is polynomial in $N$ at fixed $\beta$.  Bounding the integral over the compact interval $J$ by $|J|$ times the supremum of the integrand and letting $\eps\downarrow0$ proves \eqref{eq:D-soft-KR-obligation}.  All estimates above are uniform in $u\in J$, so the same bound, with the same $N$-independent constants, holds simultaneously for every compact subinterval of $J$; this is the uniformity required under the finite energy binning of \Cref{thm:weighted-KR}.
\end{proof}

The pointwise cancellation \eqref{eq:D-soft-Jacobian} is thus supplemented by the aggregate control of the exponentially small unstable-eigenvalue window, the half-determinant bulk mark, and the index-one constraint, all supplied by \Cref{lem:B-matrix-rates}.

Put
\begin{equation}\label{eq:D-pressure-P}
 P_{\xi,\beta}(b)
 \defeq-\frac12\log(\beta e)
 +\sup_{u\in[-E_1(\xi),-b]}
 \{-\beta u+\Xi^{\mathrm{EK}}_{1,\xi}(u)\}
 =\Lambda^{\mathrm{EK}}_{\xi,\beta}(b),
\end{equation}
the pressure of \eqref{eq:Lambda-EK}, which is the exponent of the
deterministic additive term in \eqref{eq:cocore-localization}.  The
completion of the proof of \Cref{thm:cocore-localization} below does
\emph{not} require a lower bound on the random conductance sum.

\subsection{Completion of the proof}

With \Cref{prop:soft-marked-KR} established, use the smaller interval
\begin{equation}\label{eq:D-Jbeta}
 J_\beta\defeq[-E_1(\xi)-\delta_\beta,-b]
 \subset I_{\xi,b,\delta}.
\end{equation}
By continuity of $\Xi^{\mathrm{EK}}_{1,\xi}$ and \eqref{eq:D-delta-beta},
\begin{align}
 \sup_{u\in J_\beta}
 \{-\beta u+\Xi^{\mathrm{EK}}_{1,\xi}(u)\}
 \le{}&
 \sup_{u\in[-E_1(\xi),-b]}
 \{-\beta u+\Xi^{\mathrm{EK}}_{1,\xi}(u)\}
 +O(\beta\delta_\beta)+o(\delta_\beta)
 \notag\\
 ={}&
 \sup_{u\in[-E_1(\xi),-b]}
 \{-\beta u+\Xi^{\mathrm{EK}}_{1,\xi}(u)\}
 +o(\beta^{-1/2}).
 \label{eq:D-endpoint-delta}
\end{align}
The probability bookkeeping is as follows.  Fix $\zeta>0$.  Markov's inequality applied to \eqref{eq:D-soft-KR-obligation}, together with \eqref{eq:D-endpoint-delta} and the identity \eqref{eq:D-pressure-P}, gives, for all large $N$,
\[
 \Pp\left(
 \widetilde{\cK}^{\mathrm{soft}}_{N,\beta}(J_\beta;A)
 >\exp\Bigl\{N\Bigl(\Lambda^{\mathrm{EK}}_{\xi,\beta}(b)
 -\tfrac{A-o(1)}{\sqrt\beta}+2\zeta\Bigr)\Bigr\}
 \right)\le e^{-N\zeta}.
\]
This holds for every fixed $\zeta>0$; a subsequence statement would not suffice for the theorem, so we upgrade it to the full sequence by a deterministic diagonal.  Set $\zeta_m=1/m$ and let $N_m$ be such that the displayed bound with $\zeta=\zeta_m$ holds for all $N\ge N_m$; enlarging the $N_m$ if necessary, we may assume $N_{m}>N_{m-1}$ and $N_m\ge m^2$.  Define the deterministic sequence $\zeta(N)=\zeta_m$ for $N_m\le N<N_{m+1}$.  Then $N\zeta(N)\ge N_m/m\ge m\to\infty$ and $N\zeta(N)=o(N)$, so, for all $N\ge N_1$,
\[
 \Pp\left(
 \widetilde{\cK}^{\mathrm{soft}}_{N,\beta}(J_\beta;A)
 >\exp\Bigl\{N\Bigl(\Lambda^{\mathrm{EK}}_{\xi,\beta}(b)
 -\tfrac{A-o(1)}{\sqrt\beta}+2\zeta(N)\Bigr)\Bigr\}
 \right)\le e^{-N\zeta(N)}\longrightarrow0
\]
along the full sequence $N\to\infty$, with a deterministic error $2N\zeta(N)=o(N)$.  Hence, with probability tending to one along the full sequence,
\begin{equation}\label{eq:D-soft-vs-K}
 \widetilde{\cK}^{\mathrm{soft}}_{N,\beta}(J_\beta;A)
 \le
 \exp\left\{N\Lambda^{\mathrm{EK}}_{\xi,\beta}(b)
 -\frac{(A-o(1))N}{\sqrt\beta}
 +o_{\Pp}(N)\right\}.
\end{equation}
No lower bound on the random conductance sum is used.  Combining \eqref{eq:D-soft-one}, whose sum over the soft cut saddles is bounded by $e^{C_1N/\sqrt\beta+o(N)}\,\widetilde{\cK}^{\mathrm{soft}}_{N,\beta}(J_\beta;A)$ because every $z\in\cC_N$ has energy density in $J_\beta$ by \eqref{eq:D-saddle-window}, with \eqref{eq:D-soft-vs-K} yields
\begin{equation}\label{eq:D-soft-sum-final}
 \sum_{\substack{z\in\cC_N\\\alpha_z<\alpha_*}}
 \int\rho_{N,z}^2e^{-\beta H_N}\dd\nu_N
 \le
 \exp\left\{N\Lambda^{\mathrm{EK}}_{\xi,\beta}(b)
 -\frac{(A-C_1-o(1))N}{\sqrt\beta}
 +o_{\Pp}(N)\right\}.
\end{equation}
Choose $A>2C_1+2$.  Then the soft contribution is bounded by
$\exp\{o_{\Pp}(N)\}\,e^{N\Lambda^{\mathrm{EK}}_{\xi,\beta}(b)}$.  Adding
this to the good-saddle estimate \eqref{eq:D-good-sum} gives
\eqref{eq:cocore-localization}, with a constant $C_{\xi,b}=C_1+A/2$
independent of $\beta$.  This completes the proof of
\Cref{thm:cocore-localization}.

\section{Simulation details}\label{app:sim}

\paragraph{$p$-spin dynamics.}  The Hamiltonian is realized with a
symmetrized coupling tensor so that
$\grad H_N(\sigma)=\tfrac{p}{N^{(p-1)/2}}\,J^{\mathrm{sym}}
[\sigma^{\otimes(p-1)}]$ is a single tensor contraction.  The spherical
constraint is enforced by projecting the drift and the noise onto the
tangent space and renormalizing $\|\sigma\|=\sqrt N$ after every Euler
step.  Energy-relaxation and aging runs use $p=4$, $N=60$, $\mathrm
dt=5\times10^{-3}$, $3\times10^4$ steps; the relaxation panel shows
three disorder samples per temperature (smoothed median and min--max
band) and the aging panel three independent thermal histories of one
disorder sample.  Escape-time runs use $p=4$, $N\in\{8,12,16\}$,
$\mathrm dt=10^{-2}$, at most $3\times10^5$ steps, fifteen disorder
samples per $(N,\beta)$; the starting point is the deepest of ten
projected-gradient-descent minima from independent random starts, and
the escape time is the first time the overlap with it drops below
$0.3$ (first passage to overlap $<-0.5$ is recorded separately as an
antipodal-transit proxy).  Runs that do not escape within the horizon
are recorded as censored; medians are reported only where fewer than
$40\%$ of runs are censored, and censored medians are displayed as
lower bounds and excluded from the fits.  Arrhenius slopes are fitted
to $\log$ medians and their confidence intervals are percentile
intervals over $1000$ bootstrap resamples of the disorder samples.
Step-size check: at $N=12$, $\beta=3.0$, repeating the escape
experiment with $\mathrm dt=5\times10^{-3}$ gives a median escape time
of $369$ (IQR $[175,1992]$) against $254$ (IQR $[57,2373]$) at
$\mathrm dt=10^{-2}$; the medians lie well within each other's
interquartile ranges, so discretization bias is small compared with
disorder fluctuations at these sizes.

\paragraph{Saddle enumeration.}  For $p=4$ and $N\in\{8,10,12\}$,
with eight disorder samples per size, critical points solve
$\grad H_N(\sigma)=\mu\sigma$, $\|\sigma\|^2=N$, with
$\mu=pH_N(\sigma)/N$ by Euler's identity.  We run damped Newton
iteration on the extended system in $(\sigma,\mu)$ with backtracking
on the KKT residual (tolerance $10^{-11}\sqrt N$ on the residual
norm, at most $100$ iterations) from $4{,}000$ starts per sample:
one half preconditioned by $15$--$60$ steps of projected gradient
descent, one quarter by $15$ steps of ascent, one quarter raw
uniform starts.  Points are deduplicated up to the antipodal
symmetry at Euclidean tolerance $10^{-5}\sqrt N$ (each stored
configuration represents the pair $\{z,-z\}$; binned counts and
weights carry the factor $2$).  At each critical point the
Riemannian Hessian $V^\top(\grad^2H_N-\mu I)V$ is computed in an
orthonormal tangent basis $V$ and diagonalized exactly; the index is
the number of negative eigenvalues, $\alpha_z$ the magnitude of the
negative eigenvalue for index one, and $\log|\det Q_z|$ the sum of
$\log$-magnitudes of the $N-1$ tangent eigenvalues.  Energy bins have
width $0.115$; binned counts and weighted sums are averaged over the
eight disorder samples.  Completeness diagnostic: the number of new
distinct pairs discovered in the final quarter of the $4{,}000$
starts is recorded per sample.  At
$N=8$ the enumeration is essentially saturated ($2$--$10$ new pairs
per sample in the final quarter, out of $\approx220$ per sample); at
$N=10$ and $N=12$ discovery continues ($\approx60$--$70$ and
$\approx140$--$180$ new pairs per quarter, respectively), dominated by
high-index, high-energy points, so counts at $N\ge10$ are lower
bounds---low-energy, low-index bins, which control the weighted sums
at moderate and large $\beta$, are the best-covered part of the
spectrum because half of the Newton starts are descent-preconditioned.
The dominance panel evaluates, per disorder
sample, the share of the largest term of \eqref{eq:W-main} over all
index-one saddles of that sample at each $\beta$ (computed in log
space), and reports mean $\pm$ s.d.\ over samples.

\paragraph{Phase-transition figures
(\Cref{fig:landscape}, \Cref{fig:arrest}, \Cref{fig:mechanism}).}
The arrest sweep uses $N=64$, $\mathrm dt=5\times10^{-3}$,
$2\times10^4$ steps, three disorder samples per temperature, quenching
from a uniformly random start; the plateau energy averages the last
quarter of each trajectory, and the equilibration time is the first
time the smoothed energy reaches $e_{\mathrm{eq}}(T)+0.05$ with
$e_{\mathrm{eq}}=\max(-1/T,-E_0)$ (the annealed equilibrium energy
$-1/T$ is exact at high temperature).  The arrest sweep uses $p=3$,
where $T_d=1/E_\infty$ exactly; it probes the dynamical transition of
\citet{cugliandolo1993analytical}, not the theorems.  The landscape
rendering evaluates a $p=4$, $N=3$ Hamiltonian on a $500\times500$
spherical grid.  The leading-term comparison evaluates, at $p=4$,
$\Lambda^{\mathrm{EK}}_{\xi,\beta}(b)-\beta b$ on a
$110\times90$ grid in $(\beta,b)$, computing $\Theta_{1,\xi}$, the
determinant rate $D_\xi$, and the GOE outlier rate $I_{1,\xi}$ by
quadrature as in \Cref{app:weighted-KR}; the region
$b\le E_2(4)$, which violates the necessary condition
$b>b_*(\xi)>E_2(\xi)$ of \Cref{thm:ek}, is masked, and $E_2(4)$ is
computed by root finding on $\Theta_{2,p}$.

\bibliographystyle{plainnat}
\bibliography{refs}

\end{document}